\documentclass[11pt,letterpaper]{amsart}

\usepackage{amsmath}
\usepackage{tikz-cd}

\usepackage{mathrsfs}
\usepackage{amsthm}
\usepackage{enumerate}
\usepackage{txfonts}

\usepackage{cite}
\usepackage{multirow}
\usepackage{hyperref}
\usepackage{comment}

\usepackage[utf8]{inputenc}
\usepackage[english]{babel}

\numberwithin{equation}{section}

\newtheorem{theo}{Theorem}[subsection]
\newtheorem{prop}[theo]{Proposition}
\newtheorem{lem}[theo]{Lemma}

\newtheorem{cor}[theo]{Corollary}

\newtheorem{conjecture}[theo]{Conjecture}
 
\theoremstyle{definition}

\newtheorem{definition}[theo]{Definition}
\newtheorem{condition}[theo]{Condition}

\theoremstyle{remark}
\newtheorem{remark}[theo]{Remark}

\newcommand{\bb}[1]{\mathbb{#1}}
\newcommand{\bbf}[1]{\mathbf{#1}}
\newcommand{\n}[1]{\mathcal{#1}}
\newcommand{\f}[1]{\mathfrak{#1}}
\newcommand{\s}[1]{\mathscr{#1}}
\newcommand{\Hom}{\mbox{\normalfont Hom}}

\newcommand{\Gal}{\mbox{Gal}}

\DeclareMathOperator{\Tr}{Tr}
\DeclareMathOperator{\id}{id}

\DeclareMathOperator{\Ext}{Ext}

\DeclareMathOperator{\QCoh}{QCoh}

\DeclareMathOperator{\FL}{FL}
\DeclareMathOperator{\Lie}{Lie}

\DeclareMathOperator{\Zar}{Zar}

\def \gr {{ \mathrm gr}}

\def \Sh {\mathrm{Sh}}
\def \Shan {\mathcal{S} h}
\def \tor {\mathrm{tor}}

\def \std {\mathrm{std}}
\def \KS {\mathrm{KS}}
\def \Betti {\mathrm{Betti}}
\def \BS {\mathrm{BS}}

\DeclareMathOperator{\an}{\begin{scriptsize}
an
\end{scriptsize}}
\DeclareMathOperator{\et}{\begin{scriptsize}
\acute{e}t
\end{scriptsize}}
\DeclareMathOperator{\proet}{\begin{scriptsize}
pro\acute{e}t
\end{scriptsize}}
\DeclareMathOperator{\ket}{\begin{scriptsize}
k\acute{e}t
\end{scriptsize}}
\DeclareMathOperator{\proket}{\begin{scriptsize}
prok\acute{e}t
\end{scriptsize}}

\DeclareMathOperator{\alg}{\begin{scriptsize}
alg
\end{scriptsize}}

\DeclareMathOperator{\sm}{\begin{scriptsize}
sm
\end{scriptsize}}

\DeclareMathOperator{\dR}{\begin{scriptsize}
dR
\end{scriptsize}}

\DeclareMathOperator{\cyc}{\begin{scriptsize}
cyc
\end{scriptsize}}

\DeclareMathOperator{\Hd}{\begin{scriptsize}
Hod
\end{scriptsize}}
\DeclareMathOperator{\sheaf}{\begin{scriptsize}
sheaf
\end{scriptsize}}

\newcommand{\OBdr}{\s{O}\!\bb{B}_{\dR,\log}}
\newcommand{\OBdR}[1]{\s{O}\!\bb{B}_{\dR,\log, {#1}}}
\newcommand{\OC}{\s{O}\!\bb{C}_{\log}}
\newcommand{\OCC}[1]{\s{O}\!\bb{C}_{\log,{#1}}}

\DeclareMathOperator{\sol}{\blacksquare}

\DeclareMathOperator{\Rep}{Rep}

\DeclareMathOperator{\Spec}{Spec}
\DeclareMathOperator{\Spd}{Spd}

\DeclareMathOperator{\Sym}{Sym}
\DeclareMathOperator{\Spa}{Spa}

\DeclareMathOperator{\Fil}{Fil}

\DeclareMathOperator{\Fl}{\mathscr{F}\!\ell}
\DeclareMathOperator{\HT}{HT}

\begin{document}

\title{Locally analytic completed cohomology}

\date{}
\author{Juan Esteban  Rodr\'iguez  Camargo}
\address{Max Planck Institute for Mathematics,  Vivatsgasse 7, 
53111 Bonn,
Germany}
\email{rodriguez@mpim-bonn.mpg.de}


\maketitle

\begin{abstract}
   We compute the geometric Sen operator for arbitrary Shimura varieties in terms of equivariant vector bundles of flag varieties and the Hodge-Tate period map. As an application, we obtain the  rational  vanishing of completed cohomology  in  the   Calegari-Emerton conjectures.
\end{abstract}

\tableofcontents

\section{Introduction}
\label{ch:Intro}

Let $p$ be a prime number. In this paper, we compute the geometric Sen operator of Shimura varieties, generalizing previous work of Pan in  \cite{LuePan} for the modular curve. As an application, we compute the locally analytic vectors of completed cohomology in terms of a sheaf of locally analytic functions of the infinite level-at-$p$ Shimura variety. We apply this description of locally analytic completed cohomology to deduce the vanishing of the rational completed cohomology groups above middle degree, proving a rational version of conjectures of Calegari-Emerton for Shimura varieties  \cite{CalegariEmerton}. The tools and techniques needed in this work are the Hodge-Tate period map \cite{ScholzeTorsion2015, CaraianiScholze2017}, the theory of log adic spaces and the logarithmic Riemann-Hilbert correspondence \cite{DiaoLogarithmic2019,DiaoLogarithmicHilbert2018}, geometric Sen theory \cite{RCGeoSenTheory} and the theory of solid locally analytic representations \cite{RRLocallyAnalytic,RRLocAnII}.

\subsection{Main results}
\label{Subsec:MainResults}

Let us start with the main application of the paper. Let $(\bbf{G},X)$ be a Shimura datum with reflex field $E$. For $K\subset \bbf{G}(\bb{A}^{\infty}_{\bb{Q}})$ a neat  compact open  subgroup, we let $\Sh_{K,E}$ be the canonical model over $E$ of the Shimura variety at level $K$, cf. \cite{DeligneShimura}. In this text all level structures will be assumed to be neat.  We fix  a prime-to-$p$ compact open subgroup $K^p\subset \bbf{G}(\bb{A}^{\infty,p}_{\bb{Q}})$, and consider compact open subgroups $K_p\subset \bbf{G}(\bb{Q}_p)$.

 Recall the definition of the completed cohomology groups from \cite{EmertonInterpolation}.
 
\begin{definition}
\label{DefinitionCompletedCohoIntro}
 We define the completed cohomology groups at level $K^p$ to be 
\[
\widetilde{H}^i(K^p,\bb{Z}_p) = \varprojlim_s \varinjlim_{K_p\subset \bbf{G}(\bb{Q}_p)} H^i_{\Betti}(\Sh_{K^pK_p,E}(\bb{C}),\bb{Z}/p^s)
\]
and
\[
\widetilde{H}^i_c(K^p,\bb{Z}_p) = \varprojlim_s \varinjlim_{K_p\subset \bbf{G}(\bb{Q}_p)} H^i_{\Betti,c}(\Sh_{K^pK_p,E}(\bb{C}),\bb{Z}/p^s),
\]
where $H^i_{\Betti}$ and $H^i_{\Betti,c}$ are the singular cohomology and the singular cohomology with compact support of the underlying topological space of the complex analytic manifold.
\end{definition}

The completed cohomology of Emerton is introduced for general locally symmetric spaces.  In particular, the conjectures of \cite{CalegariEmerton} are stated in such a generality. When restricted to Shimura varieties  the vanishing part of the Calegari-Emerton conjectures states the following:

\begin{conjecture}[Calegari-Emerton]
\label{ConjectureCE}
Let $d$ be the dimension of the Shimura variety, then for $i>d$ we have
\[
\widetilde{H}^i(K^p,\bb{Z}_p)=\widetilde{H}^i_c(K^p,\bb{Z}_p)=0.
\]
\end{conjecture}

The first big step towards  Conjecture \ref{ConjectureCE} is due to  Scholze \cite{ScholzeTorsion2015},  where the strategy  is to rewrite  completed cohomology    as sheaf cohomology of a certain space of cohomological dimension $d$. More precisely, let $\bb{C}_p$ be the completed algebraic closure of $\bb{Q}_p$ and let us fix an isomorphism $\bb{C}\simeq \bb{C}_p$, this fixes in particular an embedding $\iota:E \hookrightarrow  \bb{C}_p$. Let $\Shan_{K,\bb{C}_p}$ be the adic space  associated to the $\bb{C}_p$-base change of $\Sh_{K,E}$, see \cite{HuberEtaleCohomology}. When the Shimura datum is of Hodge type, Scholze proved that the limit 
\[
\Shan_{K^p,\infty,\bb{C}_p}=\varprojlim_{K_p} \Shan_{K^pK_p,\bb{C}_p}
\]
has a natural structure of a perfectoid space \cite[Theorem IV.1.1]{ScholzeTorsion2015}.  Then, completed cohomology with compact support can be computed via sheaf cohomology of a suitable compactification of $\Shan_{K^p,\infty,\bb{C}_p}$, see \cite[Theorem IV.2.1]{ScholzeTorsion2015}. Scholze's argument can be pushed to the non-compact support case  using perfectoid  toroidal compactifications of Hodge-type Shimura varieties (see \cite{PilloniStroh} and \cite{LanCLosedToroidal}) and the theory of log adic spaces of \cite{DiaoLogarithmic2019}. 

A next important step is the work of Hansen-Johansson \cite{HansenJohanssonPerf}, where they prove the vanishing for preabelian type Shimura varieties. Their method involves proving that preabelian type Shimura varieties are perfectoid at infinite level and deduce the vanishing for $\widetilde{H}^i_c$ using Scholze's argument,  then they prove the vanishing for $\widetilde{H}^i$ by studying the boundary of minimal compactifications.

In this paper we prove a weaker version of  Conjecture  \ref{ConjectureCE} with the advantage that it does not depend on the perfectoidness at infinite level, but  purely on the $p$-adic Hodge theory of Shimura varieties. In particular, it holds for \textit{arbitrary} Shimura varieties:

\begin{theo}[Corollary \ref{CoroCalegari}]
\label{TheoIntroCalegariEmerton}
Conjecture \ref{ConjectureCE} holds after inverting $p$,  namely, 
\[
\widetilde{H}^i(K^p,\bb{Z}_p)[\frac{1}{p}]=\widetilde{H}^i_c(K^p,\bb{Z}_p)[\frac{1}{p}]=0
\]
for $i>d$.
\end{theo}

We now explain the main ingredients involved in the proof of Theorem  \ref{TheoIntroCalegariEmerton}.     For simplicity in the introduction, we shall assume that the maximal $\bb{Q}$-anisotropic $\bb{R}$-split torus of the  center of $\bbf{G}$ is zero, this implies that for $K'\subset K$ an open normal subgroup, the map $\Sh_{K',E}\to \Sh_{K,E}$ is an \'etale $K/K'$-torsor.  Let $\Sh_{K^pK_p,E}^{\tor}$ denote  the toroidal compactfication of $\Sh_{K^pK_p,E}$ as in \cite{Pink1990ArithmeticalCO}, depending on a fixed cone decomposition $\Sigma$, which we assume to be smooth  and projective relative to $K^pK_p$.  Let $\Shan^{\tor}_{K^pK_p,\bb{C}_p}$ be the adic analytification over $\bb{C}_p$ of $\Sh^{\tor}_{K^pK_p,E}$. The toroidal compactification  is naturally endowed with the structure of a log adic space with reduced normal crossing  divisors as in \cite{DiaoLogarithmic2019,DiaoLogarithmicHilbert2018}. Then, we can consider the pro-Kummer-\'etale site $\Shan^{\tor}_{K^pK_p,\bb{C}_p,\proket}$ of \textit{loc. cit.}, and its completed structural sheaf $\widehat{\s{O}}_{\Shan}$ mapping a log affinoid perfectoid $S=(\Spa(R,R^+),\n{M})$ to 
\[
\widehat{\s{O}}_{\Shan}(S)=R.
\]
We will denote by  $\bb{Z}_p$ and $\bb{Q}_p$ the pro-Kummer-\'etale sheaves  over $\Shan^{\tor}_{K^pK_p,\bb{C}_p,\proket}$ mapping a log affinoid perfectoid  $S$ as before to the space of continuous functions  $C(|\Spa(R,R^+)|,\bb{Z}_p)$ and $C(|\Spa(R,R^+)|,\bb{Q}_p)$  respectively, where $|\Spa(R,R^+)|$ is  the underlying topological space of $S$.

Let us fix a bottom level $K_p$ and a toroidal compactification $\Sh_{K^pK_p,E}^{\tor}$ as above. For $K_p'\subset K_p$ an open subgroup, the map of Shimura varieties $\Sh_{K^pK_p',E}\to \Sh_{K^pK_p,E}$ is finite \'etale, and if $K_p'$ is normal it is even Galois with Galois group $K_p/K_p'$. There is a unique extension of $\Sh_{K^pK_p',E}$  to a toroidal compactification $\Sh^{\tor}_{K^pK_p',E}$ endowed with a finite Kummer-\'etale map $\Sh^{\tor}_{K^pK_p',E}\to \Sh^{\tor}_{K^pK_p,E}$ extending the original finite \'etale map of Shimura varieties\footnote{The toroidal compactifications $\Sh^{\tor}_{K^pK_p',E}$ might not longer be smooth, but this is irrelevant for our results. The only important feature that we need is that the maps between the toroidal compactifications are finite Kummer-\'etale.},  and if  $K_p'\subset K_p$ is normal then  $\Sh^{\tor}_{K^pK_p',E}\to \Sh^{\tor}_{K^pK_p,E}$ is in addition Galois with Galois group $K_p/K_p'$. See Section \ref{Subsec:SetUp} for the justification of this set up.

   We let $\Shan^{\tor}_{K^p,\infty,\bb{C}_p}=\varprojlim_{K_p'} \Shan^{\tor}_{K^pK_p',\bb{C}_p}$ be the infinite level Shimura variety, viewed  as an object in $\Shan^{\tor}_{K^pK_p,\bb{C}_p,\proket}$. Thanks to the $K_p$-torsor $\pi_{K_p}^{\tor}:\Shan^{\tor}_{K^p,\infty,\bb{C}_p}\to \Shan^{\tor}_{K^pK_p,\bb{C}_p}$, given a continuous $K_p$-representation $V$ over $\bb{Z}_p$, we can form the pro-Kummer-\'etale $\bb{Z}_p$-linear sheaf $V_{\ket}$ over $\Shan^{\tor}_{K^pK_p,\bb{C}_p}$ by descending the constant sheaf $\underline{V}$ on $\Shan^{\tor}_{K^p,\infty,\bb{C}_p}$ mapping a log affinoid perfectoid $S=(\Spa(R,R^+),\n{M})$ to $C(|\Spa(R,R^+)|,V)$.  
 
 Let $\mu$ be a $\bbf{G}(\bb{C})$-conjugacy class  of Hodge-cocharacters.    The conjugacy class  $\mu$ is defined over $E$, and there are two corresponding flag varieties $\FL_{E}$ and $\FL_E^{\std}$ over $E$.  The flag varieties are determined by the fact that they admit an holomorphic Borel embedding  $X\to \FL^{\std}_E(\mathbb{C})$  and an antiholomorphic embedding $ X\to \FL_E(\mathbb{C})$ respectively. 
 
 For any field extension $F/E$ we let $\FL_{F}$ be the base change of $\FL_E$. We shall fix a finite extension $L/\bb{Q}_p$ inside $\bb{C}_p$ containing $E$ such that there is a  Hodge-cocharacter over $L$,  by an abuse of notation we write it as  $\mu:\bb{G}_{m,L}\to \bbf{G}_{L}$.  We let $\bbf{P}_{\mu}$ be the parabolic subgroup defined over $L$  so that  $\FL_{L}$ classifies parabolic subgroups in the conjugacy class of $\bbf{P}_{\mu}$, we have an isomorphism of schemes   $\FL_{L}=\bbf{G}_L/\bbf{P}_{\mu}$. For a non-archimedean field  $K$ with a map $E\to K$, we let $\Fl_{K}$ be the analytification to an adic space over $K$ of the scheme $\FL_{K}$.
  
     Thanks to the Riemann-Hilbert correspondence of \cite{DiaoLogarithmicHilbert2018}, the local systems $V_{\ket}$ constructed from  finite-dimensional $\bb{Q}_p$-representations of $\bbf{G}$ are associated to filtered vector bundles with flat connection $V_{\dR}$ (as in \cite[Definition 7.5]{ScholzeHodgeTheory2013}, see also Definition \ref{DefinitiondeRhamLocalSys} down below).   Looking at  the Hodge-Tate filtration of $V_{\ket}$,  one obtains a $K_p\times \Gal_{L}$-equivariant map of ringed sites
  \begin{equation}
  \label{eqIntroHTmap}
  \pi^{\tor}_{\HT}:(\Shan^{\tor}_{K^p,\infty,\bb{C}_p,\proket}, \widehat{\s{O}}_{\Shan}) \to (\Fl_{\bb{C}_p,\an},\s{O}_{\Fl})
  \end{equation}
called the Hodge-Tate period map, see also \cite{ScholzeTorsion2015,CaraianiScholze2017} for the construction of the Hodge-Tate period map in the Hodge-type case.

In \cite[Theorem 3.3.4]{RCGeoSenTheory} the author constructed a geometric Sen operator for a pro-Kummer-\'etale tower of a log smooth rigid space $X_{\infty}\to X$ over $\bb{C}_p$ with pro-Kummer-\'etale Galois group  given by $p$-adic Lie group $G$.  The geometric Sen operator is a map of pro-Kummer-\'etale $\widehat{\s{O}}_X$-modules 
\[
\theta_{X_{\infty}}:(\Lie G)_{\ket}^{\vee}\otimes_{\bb{Q}_p} \widehat{\s{O}}_X\to \Omega^1_{X}(\log) \otimes_{\s{O}_X} \widehat{\s{O}}_X(-1)
\]
where $(\Lie G)_{\ket}$ is the pro-Kummer-\'etale local system over $X$ attached to the adjoint representation, $\Omega^1_{X}(\log)$ is the sheaf of log differentials over $X_{\ket}$, and $\widehat{\s{O}}_X(-1)$ is the inverse of a Hodge-Tate twist of the completed structural sheaf of the pro-Kummer-\'etale site of $X$.  The principal application of the geometric Sen operator in this paper is to compute pro-Kummer-\'etale cohomology as in \cite[Theorem 3.3.2 (2)]{RCGeoSenTheory}.

 Our first main result is the computation of the geometric Sen operator for the $K_p$-torsor $\pi_{K_p}^{\tor}:\Shan^{\tor}_{K^p,\infty,\bb{C}_p}\to \Shan^{\tor}_{K^pK_p,\bb{C}_p}$ in terms of $\pi_{\HT}^{\tor}$ and $\bbf{G}$-equivariant vector bundles over $\Fl_{\bb{C}_p}$. 

Let $\f{g}:=\Lie \bbf{G}(\bb{Q}_p)$ and let $\f{g}^0= \s{O}_{\Fl}\otimes_{\bb{Q}_p}  \f{g}$ be the $\bbf{G}$-equivariant vector bundle obtained from the adjoint representation. Let $\f{n}^0\subset \f{g}^0$ be the subbundle whose sections at a point $x$ is the Lie algebra of the unipotent radical of the parabolic subgroup fixing $x$.  The following holds:

\begin{theo}[Theorem \ref{TheoComputationSenOperator}]
\label{TheoIntroSenShimura}
The geometric Sen operator 
\[
\theta_{\Shan}:\f{g}^{\vee}_{\ket} \otimes_{\bb{Q}_p} \widehat{\s{O}}_{\Shan}\to \Omega^1_{\Shan}(\log) \otimes_{\s{O}_{\Shan}} \widehat{\s{O}}_{\Shan}(-1)
\]
of the torsor $\pi_{K_p}^{\tor}$ is naturally isomorphic to the (descent along $\pi_{K_p}^{\tor}$ of the) pullback under $\pi_{\HT}$ of the surjective map 
\[
\f{g}^{0,\vee}\to \f{n}^{0,\vee}
\]
of $\bbf{G}$-equivariant vector bundles over $\Fl_{\bb{C}_p}$, where the $K_p$-equivariant isomorphism $\pi_{\HT}^*(\f{n}^{0,\vee}) \cong \pi_{K_p}^{\tor,*} (\Omega_{\Shan}^{1}(\log))(-1)$ is (essentially) the Kodaira-Spencer map. 
\end{theo}

\begin{remark}
A crucial step in the proof of Theorem \ref{TheoIntroSenShimura} is the description of the period sheaf $\OCC{\Shan}$ (see Definition \ref{DefinitionPeriodSheaves} (6)  down below) in terms of $\bbf{G}$-equivariant vector bundles over $\Fl_{\bb{C}_p}$, see Theorem \ref{TheoSenOperator}.
\end{remark}

\begin{remark}
In \cite{he2024perfectoidness}, He proved Conjecture \ref{ConjectureCE} in its integral version. His proof relies in a point-wise perfectoid criteria in terms of non-vanishing of geometric Sen operators for valuation fields, and the computation of the geometric Sen operators for Shimura varieties of Theorem \ref{TheoIntroSenShimura}.
\end{remark}

 The  underlying topological space of the infinite level Shimura variety is given by  $|\Shan^{\tor}_{K^p,\infty,\bb{C}_p}|=\varprojlim_{K_p} |\Shan^{\tor}_{K^pK_p,\bb{C}_p}|$. We can then define a sheaf of locally analytic functions over $|\Shan^{\tor}_{K^p,\infty,\bb{C}_p}|$.

\begin{definition}[Definition \ref{DefinitionOla}]
Let $\Shan^{\tor}_{K^p,\infty,\bb{C}_p}$ be the infinite level Shimura variety. We define $\s{O}^{la}_{\Shan}$  to be the sheaf over the underlying topological space of $\Shan^{\tor}_{K^p,\infty,\bb{C}_p}$  mapping a qcqs open subspace $U$ to the space
\[
\s{O}^{la}_{\Shan}(U):= \widehat{\s{O}}_{\Shan}(U)^{K_U-la}
\] 
of locally analytic sections of the completed structural sheaf, where $K_U$ is the (necessarily open) stabilizer of $U$ in $K_p$. See Lemma \ref{LemmaSheafOla} for the fact that $\s{O}^{la}_{\Shan}$ is  a sheaf. 
\end{definition}

The main result relating the sheaf $\s{O}^{la}_{\Shan}$ with completed cohomology is the following theorem:

\begin{theo}[Theorem \ref{TheoMainVanishingOla}]
\label{TheoIntroOla}
There are  $\Gal_{L}\times K_p$-equivariant isomorphisms of cohomology groups
\[
(\widetilde{H}^i(K^p,\bb{Z}_p)\widehat{\otimes}_{\bb{Z}_p} \bb{C}_p)^{la} \cong H^i_{\sheaf}(|\Shan^{\tor}_{K^p,\infty,\bb{C}_p}|, \s{O}^{la}_{\Shan}),
\]
where the left hand side are the $K_p$-locally analytic vectors of the ($p$-adically completed) $\bb{C}_p$-base change of completed cohomology, and the right hand side is the sheaf cohomology of $\s{O}^{la}_{\Shan}$.  A similar statement holds for the completed cohomology with compact support. 
\end{theo}

One  deduces Theorem \ref{TheoIntroCalegariEmerton} from Theorem \ref{TheoIntroOla} by density of the locally analytic vectors for admissible representations \cite[Theorem 7.1]{SchTeitDist}, and the fact that $\Shan^{\tor}_{K^p,\infty,\bb{C}_p}$ has cohomological dimension $d$ as topological space, cf. \cite[proof of Corollary IV.2.2]{ScholzeTorsion2015}. 

A consequence of the study of the sheaf $\s{O}^{la}_{\Shan}$ via geometric Sen theory is the vanishing of the action of $\f{n}^0$ by derivations: 

\begin{theo}[Corollary \ref{CoroSenOpKillOla}]\label{Theo:DiffEquationsOla}
Consider the action of  $\f{g}^0_{\Shan}:=\s{O}^{la}_{\Shan}\otimes_{\bb{Q}_p}  \Lie \bbf{G}$ by derivations on $\s{O}^{la}_{\Shan}$ given by
\[
(f\otimes \mathfrak{X})\cdot  h= f \mathfrak{X}(h)
\] 
for $f,h\in \s{O}^{la}_{\Shan}$ and $\mathfrak{X}\in \Lie \bbf{G}$. Then the action of the sub Lie algebroid $\f{n}^0_{\Shan}:= \s{O}^{la}_{\Shan}\otimes_{\s{O}_{\Fl_{\bb{C}_p}}} \f{n}^0  \subset \f{g}^0_{\Shan}$  on $\s{O}^{la}_{\Shan}$ vanishes. 
\end{theo}

Finally, we show that the sheaf $\s{O}^{la}_{\Shan}$ admits an arithmetic Sen operator that can be computed in terms of representation theory:

\begin{theo}[Theorem \ref{TheoArithSenOla}]\label{Theo:ArithSenOperator}
Keep the notation of Theorem \ref{Theo:DiffEquationsOla}. Let $\theta_{\mu}\in  \f{g}^0_{\Shan}/\f{n}^0_{\Shan}$ be the element corresponding to the derivation along the Hodge-cocharacter $\mu$. Then the sheaf $\s{O}^{la}_{\Shan}$ admits an arithmetic Sen operator (in the sense of Definition \ref{DefinitionDecompletionByLocAn})  given by  $-\theta_{\mu}$. The same holds for the locally analytic completed cohomology groups of Theorem \ref{TheoIntroOla}.  
\end{theo}

\subsection{Overview of the paper}
\label{SubsecOverview}

In Section \ref{Section:Preliminaries}, we introduce some  notations from the theory of log adic spaces from \cite{DiaoLogarithmic2019,DiaoLogarithmicHilbert2018},  we recall some basic facts of the theory of solid locally analytic representations of \cite{RRLocallyAnalytic,RRLocAnII}, and we explain how to promote certain pro-Kummer-\'etale cohomology groups to solid abelian groups.  In Section \ref{Section:EquivariantFlag}, we  construct some $\bbf{G}$-equivariant vector bundles over flag varieties that we shall need in the proof of Theorem \ref{TheoIntroSenShimura}. Then, in Section \ref{Section:Shimura}, we set up the framework of Shimura varieties and explain how the results of \cite{DiaoLogarithmicHilbert2018} lead to the Hodge-Tate period map \eqref{eqIntroHTmap}.   In Section \ref{Section:SenOperators}, we compute the geometric Sen operators of Shimura varieties proving Theorem \ref{TheoIntroSenShimura}. In Section  \ref{Section:CompletedCohomology},  we introduce the completed cohomology groups of Emerton  and relate them with pro-Kummer-\'etale cohomology of infinite level Shimura varieties (Corollary \ref{CoroCompletedCohoAsProkummerCoho}).  Then, we use geometric Sen theory to  prove   Theorems \ref{TheoIntroOla} and  \ref{Theo:DiffEquationsOla} deducing Theorem \ref{TheoIntroCalegariEmerton} as  a corollary.   Finally, in Section \ref{SectionArithmeticSen},  we define the arithmetic Sen operator for  solid $\bb{C}_p$-semilinear representations of the Galois group of finite extensions of $\bb{Q}_p$ and show Theorem \ref{Theo:ArithSenOperator}.

\subsection*{ Acknowledgments}

First and foremost,  I would like to thank my advisor Vincent Pilloni who guided me throughout the years of my PhD.  I want to thank George Boxer and Joaqu\'in Rodrigues Jacinto for their patience and enlightenment after trespassing to their office several times for discussing different stages of the work.   I am grateful  to the  participants of the study group in Lyon in the Spring of 2021 on Pan's work on locally analytic completed cohomology for the modular curve.     I want to thank Gabriel Dospinescu for several discussions regarding  the locally analytic vectors of the infinite level  Lubin-Tate space.   I would like to thank Arthur-C\'esar Le Bras for several fruitful conversations  and  his invitation to speak in the LAGA seminar about this work.  I also want to thank Fabrizio Andreatta,  Andrew Graham,  Valentin Hernandez,  Sean Howe,  Ben Heuer,    Damien Junger, Lue Pan and Yujie Xu for fruitful exchanges.   Special gratitude to the anonymous referee for their comments and corrections that lead into an important simplification and better exposition of the paper.  Finally, I thank  the  Max Planck Institute for Mathematics, Columbia university and  the Simons society of fellows for their hospitality and support during the correction stage of the document. This work has been done while the author was a PhD student at the ENS de Lyon.

\section{Preliminaries}
\label{Section:Preliminaries}

In this section we recollect some basic aspects of the theory of logarithmic adic spaces  \cite{DiaoLogarithmic2019,DiaoLogarithmicHilbert2018} and  the theory of solid locally analytic representations \cite{RRLocallyAnalytic,RRLocAnII} that will be used throughout the rest of the paper.

\subsection{Pro-Kummer \'etale site}
\label{Subsec:ProKummerEtale}

In the next  paragraph we briefly recall the main features of the pro-Kummer-\'etale site and the theory of log adic spaces of \cite{DiaoLogarithmic2019,DiaoLogarithmicHilbert2018}.

Let $(K,K^+)$ be a discretely valued complete non-archimedean extension of $(\bb{Q}_p,\bb{Z}_p)$ with perfect residue field, and let $X$ be a log smooth adic space over $(K,K^+)$ (cf. \cite[Definition 3.1.1]{DiaoLogarithmic2019}), where $X$ has log structure induced by a normal crossing divisor $D$ as in \cite[Example 2.3.17]{DiaoLogarithmic2019}. We let $X_{\an}$, $X_{\ket}$ and $X_{\proket}$ denote the analytic, Kummer-\'etale and pro-Kummer-\'etale sites of $X$ respectively, the last two considered as in \cite[Definitions 4.1.16 and 5.1.2]{DiaoLogarithmic2019}. If $D$ is empty, the (pro-)Kummer-\'etale site of $X$ is the same as the (pro)\'etale site of \cite{ScholzeHodgeTheory2013}, and we simply denote them by $X_{\et}$ and $X_{\proet}$.  By \cite[Proposition 5.3.12]{DiaoLogarithmic2019}, log affinoid perfectoids (cf. \cite[Definition 5.3.1]{DiaoLogarithmic2019}) form a basis of $X_{\proket}$. 

One can define a sheaf of log differentials $\Omega_X^1(\log)$ of $X$  over $(K,K^+)$, see \cite[Construction 3.3.2]{DiaoLogarithmic2019}. If $f:Y\to X$ is a Kummer-\'etale morphism then there is a natural isomorphism of log differentials
\[
\Omega^1_Y(\log) \cong f^* \Omega_X^1(\log), 
\]
see \cite[Theorem 3.3.17]{DiaoLogarithmic2019}. We also  write   $\s{O}_X$  and $\Omega^1_Y(\log) $  for  the inverse images to $X_{\proket}$ of the structural sheaf  and the sheaf of log differentials of $X_{\ket}$ respectively. 

\begin{definition}
\label{DefinitionPeriodSheaves}
Over $X_{\proket}$ we have different period sheaves that we recall next, see \cite[Section 2.2]{DiaoLogarithmicHilbert2018}. 

\begin{enumerate}

\item The completed constant sheaves $\bb{Z}_p$ and $\bb{Q}_p$ (denoted as $\widehat{\bb{Z}}_p$ and $\widehat{\bb{Q}}_p$ in \textit{loc. cit.}), whose values at a log affinoid perfectoid $S=(\Spa(R,R^+),\n{M})$ are given by the continuous functions $C(|\Spa(R,R^+)|,\bb{Z}_p)$ and $C(|\Spa(R,R^+)|,\bb{Q}_p)$, where $|\Spa(R,R^+)|$ is the underlying topological space of $\Spa(R,R^+)$.

\item More generally, given $V$ a topological abelian group, we  denote by  $\underline{V}$  the pro-Kummer-\'etale sheaf over $X$ mapping a log affinoid perfectoid $S=(\Spa(R,R^+),\n{M})$ to $C(|\Spa(R,R^+)|,V)$. 

\item The completed structural sheaf $\widehat{\s{O}}_X$ (resp. $\widehat{\s{O}}_X^+$) mapping a log affinoid perfectoid $S=(\Spa(R,R^+),\n{M})$ to 
\[
\widehat{\s{O}}_X(S)=R \hspace{10pt} \mbox{ (resp. $\widehat{\s{O}}_X^+(S)=R^+$)}.
\]

\item The de Rham period sheaf $\bb{B}_{\dR,X}$ (resp. $\bb{B}_{\dR,X}^+$) mapping a log affinoid perfectoid  $S=(\Spa(R,R^+),\n{M})$ to 
\[
\bb{B}_{\dR,X}(S)= \bb{B}_{\dR}(R) \hspace{10pt} \textit{ (resp.  $\bb{B}_{\dR,X}^+(S)= \bb{B}_{\dR}^+(R)$)}.
\]
There is a Fontaine map $\theta: \bb{B}_{\dR,X}^+\to \widehat{\s{O}}_X$ whose kernel is principal locally in the pro-Kummer-\'etale topology. 

\item The geometric de Rham period sheaves $\OBdR{X}$ and  $\OBdR{X}^+$ of \cite[Definition 2.2.10]{DiaoLogarithmicHilbert2018}. The sheaf $\OBdR{X}^+$ is an $\s{O}_X$-module endowed with a filtration and a flat connection $\nabla$ satisfying Griffiths transversality. Moreover, by \cite[Corollary 2.4.6]{DiaoLogarithmicHilbert2018} the Poincar\'e lemma holds, namely, there is a de Rham long exact sequence
\[
0\to \bb{B}_{\dR,X}^+\to \OBdR{X}^+\xrightarrow{\nabla} \OBdR{X}^+\otimes_{\s{O}_X} \Omega_X^1(\log) \to \cdots \to \OBdR{X}^+\otimes_{\s{O}_X} \Omega_X^{\dim X}(\log) \to 0.
\]
Let $t$ be a local generator of $\ker \theta$. Consider $\OBdR{X}^+[\frac{1}{t}]$ endowed with the convolution filtration making $t$ of degree $1$. Then $\OBdR{X}$ is the completion of $\OBdR{X}^+[\frac{1}{t}]$ with respect to that filtration. It is an $\s{O}_X$-module endowed with a filtration and a flat connection satisfying Griffiths transversality.

\item The Hodge-Tate sheaf $\OCC{X}:= \gr^0(\OBdR{X})$. The sheaf $\OCC{X}$ is endowed with an $\widehat{\s{O}}_X$-linear Higgs field
\[
\overline{\nabla}: \OCC{X}\to \OCC{X}\otimes_{\s{O}_X} \Omega^1_X(\log) (-1),
\]
where $\s{F}(i)$  denotes the $i$-th Hodge-Tate twist of $\s{F}$. The map $\overline{\nabla}$ is obtained by taking graded pieces of the connection $\nabla$. 

\end{enumerate}
\end{definition}

\begin{remark}
\label{RemarkDifferentConstructionOC}
The ring $\OCC{X}$ has a natural increasing filtration arising from the filtered algebra $\OBdR{X}^+$. Indeed, we have an inclusion of the $-1$-twist of the Faltings extension $\gr^1 \OBdR{X}^+ (-1) \to \OCC{X}$, and if $e:\widehat{\s{O}}_X\to \gr^1 \OBdR{X}^{+}(-1)$ is the natural map, we have a presentation 
\[
\Sym_{\widehat{\s{O}}_X}^{\bullet} (\gr^1 \OBdR{X}^{+}(-1))/(1-e(1)) = \OCC{X},
\]
where $1$ is the unit in the symmetric algebra, and $e(1)$ is the image of $1$ under the map $e$.  This defines an increasing filtration  $\Fil_n \OCC{X}$ such that 
\[
\Fil_n \OCC{X} \cong \Sym^n_{\widehat{\s{O}}_X} (\gr^1 \OBdR{X}^{+}(-1)). 
\]
\end{remark}

\subsection{De Rham local systems and the Hodge-Tate filtration}
\label{Subsec}

Let $(K,K^+)$ be as before and let $X$ be a log smooth adic space over $(K,K^+)$ with log structure arising from normal crossing divisors $D\subset X$.   Following \cite[Definition 8.3]{ScholzeHodgeTheory2013} and \cite[Theorem 3.2.12]{DiaoLogarithmicHilbert2018} we make the following definition:

\begin{definition}
\label{DefinitiondeRhamLocalSys}
Let $\bb{L}$ be  a pro-Kummer-\'etale  $\bb{Z}_p$-local system with unipotent  monodromy along the boundary $D$. We say that $\bb{L}$ is de Rham if there is a filtered log-connection $(\s{F},\nabla, \Fil^{\bullet})$ and an isomorphism of pro-Kummer-\'etale sheaves
\begin{equation}
\label{eqRHEquality}
\bb{L}\otimes_{\bb{Z}_p} \OBdR{X}\cong \s{F}\otimes_{\s{O}_X} \OBdR{X}
\end{equation}
compatible with filtrations and log-connections.

 Equivalently, let $\nu: X_{\proket}\to X_{\ket}$ be the natural projection of sites. Then  $\bb{L}$ is de Rham if and only if the filtered log connection $\n{RH}(L)$ has same rank as $\bb{L}$, where 
\[
\n{RH}(\bb{L}):= \nu_* (\bb{L}\otimes_{\bb{Z}_p} \OBdR{X}).
\]
We call $\s{F}$ the associated filtered log-connection of $\bb{L}$.
\end{definition}

\begin{remark}
\cite[Theorem 3.2.7]{DiaoLogarithmicHilbert2018} provides a more general Riemann-Hilbert correspondence without the unipotent assumption. In that situation, \eqref{eqRHEquality} is not  necessarily an isomorphism. Since the local systems associated to Shimura varieties are unipotent at the boundary, we will restrict ourselves to Definition \ref{DefinitiondeRhamLocalSys}.
\end{remark}

Let $\bb{L}$ be  a de Rham local system  as in Definition \ref{DefinitiondeRhamLocalSys} with associated filtered log-connection $\s{F}$, we write $\bb{M}= \bb{L}\otimes_{\bb{Z}_p} \bb{B}_{\dR,X}^+$ and $\bb{M}^0= (\s{F}\otimes_{\s{O}_X} \OBdR{X}^+)^{\nabla=0} $ for the two $\bb{B}^+_{\dR,X}$-lattices of $\bb{L}\otimes_{\bb{Z}_p} \bb{B}_{\dR,X}$.  We endow $\bb{M}$ and $\bb{M}^0$ with the $(\ker \theta)$-adic filtrations. 

\begin{definition}
\label{DefinitionHodgeTateFiltration}
The Hodge-Tate filtration of $\bb{L}\otimes_{\bb{Z}_p} \widehat{\s{O}}_X$ is defined as the increasing filtration
\[
\Fil_{-j} (\bb{L}\otimes_{\bb{Z}_p} \widehat{\s{O}}_X)= (\bb{M}\cap\Fil^j \bb{M}^0 )/ (\Fil^1 \bb{M}\cap \Fil^j \bb{M}^0 ).
\]
\end{definition}

\begin{lem}
\label{LemmaGradedpieces}
Let $\bb{L}$ be a de Rham $\bb{Z}_p$-local system over $X$ with unipotent monodromy along the boundary divisor. Let $\s{F}$ be its  associated  filtered log connection. Then there are natural isomorphisms between the graded pieces of the  Hodge-Tate and Hodge filtrations
\[
\gr_j(\bb{L}\otimes_{\bb{Z}_p} \widehat{\s{O}}_X) \cong \gr^j(\s{F}) \otimes_{\s{O}_X} \widehat{\s{O}}_X(-j).
\]
\end{lem}
\begin{proof}
This is a direct consequence of the  log analogue of the first statement of \cite[Proposition 7.9]{ScholzeHodgeTheory2013}, where the only input needed is the isomorphism \eqref{eqRHEquality} compatible with filtrations and log-connections, see also \cite[Section 4.4.38]{BoxerPilloniHigherColeman}.
\end{proof}

\subsection{Locally analytic representations}
\label{Subsec:LocallyAnalyticRepresentations}

The main protagonists of this paper are the  locally analytic vectors of completed cohomology. The tools we use are geometric Sen theory \cite{RCGeoSenTheory} and the theory of solid locally analytic representations \cite{RRLocallyAnalytic,RRLocAnII}.  In what follows we shall recall some basic facts of the theory of solid locally analytic representations that will be systematically used throughout the rest of the work.

Let $G$ be a compact $p$-adic Lie group and $\n{K}=(K,K^+)$ a complete non-archimedean extension of $\bb{Q}_p$. In \cite[Definition 3.2.1]{RRLocAnII}  we define a derived category $\Rep^{la}_{\n{K}}(G)$ of $\n{K}$-linear solid locally analytic representations of $G$. Moreover, in \cite[Definition 4.40]{RRLocallyAnalytic} and \cite[Definition 3.1.4 (3)]{RRLocAnII} we define a functor of derived locally analytic vectors, sending a solid $G$-representation $V$ to the locally analytic representation
\[
V^{Rla}:= R\Gamma(G, V\otimes^L_{\bb{Q}_p,\sol} C^{la}(G,\bb{Q}_p)_{\star_1}).
\]
where
\begin{itemize}

\item $C^{la}(G,\bb{Q}_p)_{\star_1}$ is the space of locally analytic functions of $G$ endowed with the left regular action, i.e. for $f\in C^{la}(G,\bb{Q}_p)$  and $g,h\in G$ we have $(g \star_1 f)(h)=f(g^{-1}h)$.

\item The solid tensor product $V\otimes^L_{\bb{Q}_p,\sol} C^{la}(G,\bb{Q}_p)_{\star_1}$ is endowed with the diagonal action of $G$.

\item The group cohomology is as solid abelian groups \cite[Definition 5.1]{RRLocallyAnalytic}.

\end{itemize}
 When $V$ is a classical Banach representation of $G$ (which will be the main case of interest for us), the solid tensor product $V\otimes^L_{\bb{Q}_p,\sol} C^{la}(G,\bb{Q}_p)_{\star_{1}}$ is nothing but the space $C^{la}(G,V)$ of locally analytic functions $f:G\to V$ endowed with the action 
\[
(g\star_{1,3} f)(h)= g\cdot f(g^{-1}h)
\]
for $g,h\in G$, and the group cohomology is the natural solid enhancement of   the usual continuous group cohomology thanks to \cite[Lemma 5.2]{RRLocallyAnalytic}. In particular, the cohomology groups of  $V^{Rla}$ are a solid enhancement of the derived locally analytic vectors of \cite[Definition 2.2.1]{LuePan}.

In the case when $V$ is a Banach admissible representation (in the sense of \cite{SchTeitDist}) we have the following vanishing result:

\begin{prop}
\label{PropAdmissiblenoHigher}
Let $K$ be a finite extension of $\mathbb{Q}_p$ and  $V$  a Banach admissible representation of $G$ over $K$, then $V^{Rla}=V^{la}[0]$ is concentrated in degree $0$, i.e., the higher locally analytic vectors of $V$ vanish.
\end{prop}
\begin{proof}
We shall use the notation of \cite[Proposition 4.48]{RRLocallyAnalytic}. By \textit{loc. cit.} the derived $\bb{G}^{(h^+)}$-analytic vectors $V^{R\bb{G}^{(h^+)}-an}$ of $V$ are concentrated in degree $0$. Since $V^{Rla}=\varinjlim_{h} V^{R\bb{G}^{(h^+)}-an}$ by \cite[Corollary 3.1.10]{RRLocAnII}, one deduces that $V^{Rla}$ is also in degree $0$ as wanted. 
\end{proof}

We finish this section with a couple of lemmas regarding the passage to (locally) analytic vectors of normal subgroups. For that, we need to recall the definition of the functor of analytic vectors.   Let $\bb{G}$ be a rigid analytic group over $\bb{Q}_p$ isomorphic as a rigid space to a  finite disjoint union of polydiscs. Let $C(\bb{G},\bb{Q}_p)$ be its space of rigid analytic functions, considered as a Banach algebra, and let $\n{D}(\bb{G},\bb{Q}_p)$ be its dual  in solid $\bb{Q}_p$-vector spaces (cf. \cite[Definition 4.7]{RRLocallyAnalytic}). Let us consider the $p$-adic Lie group  $G=\bb{G}(\bb{Q}_p)$. In the rest of the section all the homological algebra occurs within the world of solid $\bb{Q}_p$-linear vector spaces.

  Let $V$ be a solid $\bb{Q}_p$-linear representation of $G$,  the derived $\bb{G}$-analytic vectors of $V$ are given by the $G$-representation  (cf. \cite[Definition 4.9]{RRLocallyAnalytic})
  \[
  V^{R\bb{G}-an}:= R\Gamma(G, (V\otimes^{L}_{\bb{Q}_{p,\sol}} C(\bb{G},\bb{Q}_p)_{\sol})_{\star_{1,3}}),
  \]
where the action of $G$ in the tensor product of the right term is the diagonal action on $V$ and the left regular action on the space of analytic functions. 
By \cite[Corollary 2.19]{RRLocallyAnalytic} we can rewrite the space of analytic vectors to be 
\begin{equation}\label{eqomasdapwdqw}
V^{R\bb{G}-an}=R\Gamma(G, R\underline{\Hom}_{\bb{Q}_p}( \n{D}(\bb{G}, \bb{Q}), V )),
\end{equation}
where $R\underline{\Hom}_{\bb{Q}_p}( \n{D}(\bb{G}, \bb{Q}), V )$ has the $G$-action induced by the action on $V$, and the left regular action on the distributions. 

\begin{lem}\label{LemmaAnRepNormalSubgroup}
Let $\bb{H}\subset \bb{G}$ be a normal Zariski closed immersion of affinoid analytic groups over $\bb{Q}_p$, both isomorphic to finite disjoint unions of polydiscs as rigid spaces.  Suppose that there is a section $\bb{G}/\bb{H}\to \bb{G}$ of rigid varieties with image $\bb{X}$.

 Let $G=\mathbb{G}(\bb{Q}_p)$ and $H=\mathbb{H}(\bb{Q}_p)$ be the underlying compact $p$-adic Lie groups. Let $V\in D(\bb{Q}_{p,\sol}[G])$ be a solid  $\bb{Q}_p$-linear representation of $G$. Then the  space of derived $\mathbb{H}$-analytic vectors 
\[
V^{R\mathbb{H}-an}:=R\Gamma(H, (V\otimes^L_{\bb{Q}_p,\sol} C(\mathbb{H},\bb{Q}_p)_{\sol})_{\star_{1,3}})
\]
(see \cite[Definition 4.29]{RRLocallyAnalytic} for the notations) has a natural structure of $G$-representation. More precisely, the natural map 
\begin{equation}\label{eqomdaosdawdasd}
\bb{Q}_{p,\sol}[G]\otimes_{\bb{Q}_{p,\sol}[H]} \n{D}(\bb{H}, \bb{Q}_p)\to   \n{D}(\bb{G},\bb{Q}_p)
\end{equation}
is injective with image  a subalgebra $\n{D}(G;\bb{H}, \bb{Q}_p)$ of $\n{D}(\bb{G},\bb{Q}_p)$ containing $\bb{Q}_{p,\sol}[G]$, and there is a natural isomorphism  of $H$-representations
\[
V^{R\bb{H}-an}=R\Gamma(G, R\underline{\Hom}_{\bb{Q}_p}(\n{D}(G;\bb{H},\bb{Q}_p), V)),
\]
where the $G$-cohomology is taken with respect to the natural action on $V$ and the left multiplication  on $\n{D}(G;\bb{H},\bb{Q}_p)$, and the $H$-action on the right term arises from the right multiplication on $\n{D}(G;\bb{H},\bb{Q}_p)$.  Thus, $V^{R\bb{H}-an}$ has a structure of left $\n{D}(G;\bb{H},\bb{Q}_p)$-module compatible with the action of $H$ (arising via the right multiplication in $\n{D}(G;\bb{H},\bb{Q}_p)$), and so it has a structure of $G$-representation.  
\end{lem}
\begin{proof}
We first show that \eqref{eqomdaosdawdasd} is injective and that its image is a subalgebra.  By hypothesis,  the map $\bb{G}\to \bb{G}/\bb{H}$ has a section given by a subvariety $\bb{X}\subset \bb{G}$. Hence, we have isomorphisms of left $\bb{H}$-equivariant spaces $\bb{G}=\bb{H}\bb{X}\cong \bb{H}\times \bb{G}/\bb{H}$, and of right  $\bb{H}$-equivariant spaces $\bb{G}=\bb{X}\bb{H} \cong \bb{X}\times \bb{H}$. In particular,  we have isomorphisms of right $\bb{Q}_{p,\sol}[H]$-modules
\begin{equation}\label{qeq0qo3malwd}
\bb{Q}_{p,\sol}[G]\cong \bb{Q}_{p,\sol}[G/H]\otimes_{\bb{Q}_{p,\sol}} \bb{Q}_{p,\sol}[H]
\end{equation}
and similarly as las $\bb{Q}_{p,\sol}[H]$-modules.  We have similar equivariant descriptions of distribution algebras $\n{D}(\bb{G},\bb{Q}_p)\cong \n{D}(\bb{G}/\bb{H},\bb{Q}_p) \otimes_{\bb{Q}_{p,\sol}}  \n{D}(\bb{H},\bb{Q}_p)$ as right $\n{D}(\bb{H},\bb{Q}_p)$-modules (resp. as left modules). 
 In particular, the tensor product of \eqref{eqomdaosdawdasd} is also derived, and the map is injective as so it is $\bb{Q}_{p,\sol}[H]\to \n{D}(\bb{H},\bb{Q}_p)$, and $\bb{Q}_{p,\sol}[G/H]$ is a flat $\bb{Q}_p$-module for the solid tensor product thanks to \cite[Lemma 3.21]{RRLocallyAnalytic}. Now, the idea is that, since $\bb{H}$ is normal, the image of  $\n{D}(\bb{H},\bb{Q}_p)\otimes_{\bb{Q}_{p,\sol}[H]} \bb{Q}_{p,\sol}[G]$ in $\n{D}(\bb{G},\bb{Q}_p)$ agrees with the image of \eqref{eqomdaosdawdasd} and therefore this subspace forms a subalgebra of $\n{D}(\bb{G},\bb{Q}_p)$. To justify this, we argue in the dual side as follows.
 
  Let us write $\n{D}(G;\bb{H},\bb{Q}_p)$ for the tensor product \eqref{eqomdaosdawdasd}.  By \eqref{qeq0qo3malwd}  we have that 
\[
\n{D}(G;\bb{H},\bb{Q}_p)= \bb{Q}_{p,\sol}[G/H]\otimes_{\bb{Q}_{p,\sol}} \n{D}(\bb{H},\bb{Q}_p)
\]
 as solid $\bb{Q}_p$-vector space. Both factors of the tensor product are Smith spaces (i.e. duals of Banach $\bb{Q}_p$-vector spaces, see \cite[Definition 3.2]{RRLocallyAnalytic}), and the duality between Smith and Banach spaces of \cite[Theorem 3.40]{RRLocallyAnalytic} yields that 
 \[
 \begin{aligned}
 \underline{\Hom}_{\bb{Q}_p}(\n{D}(G;\bb{H},\bb{Q}_p), \bb{Q}_p) & =  \underline{\Hom}_{\bb{Q}_p}(\bb{Q}_{p,\sol}[G/H], \bb{Q}_p) \otimes_{\bb{Q}_{p,\sol}} \underline{\Hom}_{\bb{Q}_p}(\n{D}(\bb{H}, \bb{Q}_p), \bb{Q}_p)\\
 & =C(G/H,\bb{Q}_p)\otimes_{\bb{Q}_{p,\sol}} C(\bb{H}, \bb{Q}_p) =: C(G; \bb{H}, \bb{Q}_p). \\
 \end{aligned}
 \]
Notice that $C(G; \bb{H}, \bb{Q}_p)$ naturally identifies with the space of functions of the trivial  pro-finite-\'etale  extension of $\bb{H}$ given, as diamond, by $(\underline{G/H})\times \bb{H}^{\diamond}=\underline{G}\bb{H}^{\diamond}\subset \bb{G}^{\diamond}$. In particular, we have an inclusion of Banach algebras $C(G;\bb{H},\bb{Q}_p)\subset C(G,\bb{Q}_p)$.

 Therefore, to see that the image of  \eqref{eqomdaosdawdasd} is a subalgebra of $\n{D}(\bb{G}, \bb{Q}_p)$, it suffices to see that its dual  $C(G;\bb{H},\bb{Q}_p)\subset C(G,\bb{Q}_p)$ is stable under the co-multiplication map on continuous functions, but this follows from the fact that $\underline{G}\mathbb{H}^{\diamond}=\mathbb{H}^{\diamond}\underline{G}\subset \mathbb{G}^{\diamond}$ is a subgroup as $\bb{H}$ is normal. 
 

Next, we prove the statement about $\bb{H}$-analytic vectors.  By applying $\otimes$-$\Hom$ adjunctions, we have natural equivalences of $G$-representations
\[
\begin{aligned}
 R\underline{\Hom}_{\bb{Q}_p}(\n{D}(G;\bb{H},\bb{Q}_p), V) & = R\underline{\Hom}_{\bb{Q}_p}(\bb{Q}_{p,\sol}[G]\otimes^L_{\bb{Q}_{p,\sol}[H]} \n{D}(\bb{H},\bb{Q}_p), V) \\ 
 & = R\underline{\Hom}_{\bb{Q}_{p,\sol}[H]}(\n{D}(\bb{H},\bb{Q}_p), R\underline{\Hom}_{\bb{Q}_p}( \bb{Q}_{p,\sol}[G],V))
 \end{aligned}
\]
where in the first term $G$ acts via the left multiplication on $\n{D}(G;\bb{H}, \bb{Q}_p)$ and on $V$. In the last term,  the  $\Hom$ over $\bb{Q}_{p,\sol}[H]$ is taken for the left action on $\n{D}(\bb{H},\bb{Q}_p)$ and the right action on $\bb{Q}_{p,\sol}[G]$, and the group $G$ acts via the left multiplication on $\bb{Q}_{p,\sol}[G]$  and on $V$. Hence, taking $G$-cohomology, and noticing that the actions of $G$ and $H$ commute, we have that
\[
\begin{aligned}
R\Gamma(G,R\underline{\Hom}_{\bb{Q}_p}(\n{D}(G;\bb{H},\bb{Q}_p), V) ) & =  R\Gamma(G,R\underline{\Hom}_{\bb{Q}_{p,\sol}[H]}(\n{D}(\bb{H},\bb{Q}_p), R\underline{\Hom}_{\bb{Q}_p}( \bb{Q}_{p,\sol}[G],V))) \\ 
	& =R\underline{\Hom}_{\bb{Q}_{p,\sol}[H]}(\n{D}(\bb{H},\bb{Q}_p), R\Gamma(G, R\underline{\Hom}_{\bb{Q}_p}( \bb{Q}_{p,\sol}[G],V)))\\ 
	& =  R\underline{\Hom}_{\bb{Q}_{p,\sol}[H]}(\n{D}(\bb{H},\bb{Q}_p), V) \\ 
	& = V^{R\bb{H}-an}
\end{aligned}
\]
In the previous we have used that $R\Gamma(G,-)$ is the same as $R\underline{\Hom}_{\bb{Q}_{p,\sol}[G]}(\bb{Q}_p,-)$ to make the $\Hom$ space commute via $\otimes$-$\Hom$ adjunctions, and that $ R\Gamma(G, R\underline{\Hom}_{\bb{Q}_p}( \bb{Q}_{p,\sol}[G],V))=V$ being an induced representation. 
\end{proof}

\begin{lem}\label{LemmaSemidirectAnVectors}
Let $\mathbb{G}=\mathbb{H}\ltimes \mathbb{U}$ be an affinoid rigid group over $\bb{Q}_p$ written as a semidirect product of rigid affinoid groups. Suppose that $\bb{H}$ and $\bb{U}$ are isomorphic to finite disjoint unions of polydics over $\bb{Q}_p$.

 Let $G=\mathbb{G}(\bb{Q}_p)$, $H=\bb{H}(\bb{Q}_p)$ and $U=\bb{U}(\bb{Q}_p)$. Let $V\in D(\bb{Q}_{p,\sol}[G])$ be a solid $\bb{Q}_p$-linear representation of $G$. Then there is a natural quasi-isomorphism in $D(\bb{Q}_{p,\sol})$
\[
V^{R\bb{G}-an}\cong (V^{R\bb{U}-an})^{R\bb{H}-an}.
\]
Informally,  derived  $\bb{G}$-analytic vectors can be computed by first computing the derived analytic $\bb{U}$-analytic vectors, and then the derived $\bb{H}$-analytic vectors. 
\end{lem}
\begin{proof}
By Lemma \ref{LemmaAnRepNormalSubgroup} we can write 
\[
V^{R\bb{U}-an}=R\Gamma(G, R\underline{\Hom}_{\bb{Q}_p}(\n{D}(G;\bb{U},\bb{Q}_p), V))
\]
as a $G$-representation, where the action of $G$ is induced by the right multiplication on $\n{D}(G;\bb{U},\bb{Q}_p)$.  Consider the space 
\begin{equation}\label{eqspdmapsd}
\begin{aligned}
R\underline{\Hom}_{\bb{Q}_p}(\n{D}(\bb{H},\bb{Q}_p), R\underline{\Hom}_{\bb{Q}_p}(\n{D}(G;\bb{U},\bb{Q}_p), V)) & = R\underline{\Hom}_{\bb{Q}_p}(\n{D}(G;\bb{U},\bb{Q}_p)\otimes_{\bb{Q}_{p,\sol}}^L \n{D}(\bb{H},\bb{Q}_p), V) \\ 
& =  R\underline{\Hom}_{\bb{Q}_p}(\n{D}(G;\bb{U},\bb{Q}_p)\otimes_{\bb{Q}_{p,\sol}} \n{D}(\bb{H},\bb{Q}_p), V). 
\end{aligned}
\end{equation}
It has an action of $G\times H$, where $G$ acts on $V$ and via the left multiplication on $\n{D}(G;\bb{U},\bb{Q}_p)$, and $H$ acts via the right multiplication on $\n{D}(G;\bb{U},\bb{Q}_p)$ and the left action on $\n{D}(\bb{H},\bb{Q}_p)$. Thanks to \eqref{eqomasdapwdqw} and Lemma \ref{LemmaAnRepNormalSubgroup}, we see that $(V^{R\bb{U}-an})^{R\bb{H}-an}$ is the $H$-cohomology of  the $G$-cohomology of this representation. But this agrees with the $G\times H$-cohomology which can also be computed as the $G$-cohomology of the $H$-cohomology (thanks to $\otimes$-$\Hom$ adjunctions). Hence, by using $\otimes$-$\Hom$ adjunctions we see that the $H$-cohomology of \eqref{eqspdmapsd} is given by 
\[
\begin{gathered}
R\Gamma(H,R\underline{\Hom}_{\bb{Q}_p}(\n{D}(G;\bb{U},\bb{Q}_p)\otimes_{\bb{Q}_{p,\sol}} \n{D}(\bb{H},\bb{Q}_p), V)) \\ = R\underline{\Hom}_{\bb{Q}_{p,\sol}[H]}(\bb{Q}_p,R\underline{\Hom}_{\bb{Q}_p}(\n{D}(G;\bb{U},\bb{Q}_p)\otimes_{\bb{Q}_{p,\sol}} \n{D}(\bb{H},\bb{Q}_p), V)) \\ 
=R\underline{\Hom}_{\bb{Q}_p}(\bb{Q}_p\otimes^L_{\bb{Q}_{p,\sol}[H]} \big(\n{D}(G;\bb{U},\bb{Q}_p)\otimes_{\bb{Q}_{p,\sol}} \n{D}(\bb{H},\bb{Q}_p) \big), V) \\
= R\underline{\Hom}_{\bb{Q}_p}(\n{D}(G;\bb{U},\bb{Q}_p)\otimes_{\bb{Q}_{p,\sol}[H]}^L  \n{D}(\bb{H},\bb{Q}_p), V   )
\end{gathered}
\]
where the first equivalence is the definition of $H$-cohomology. The second equivalence is a $\otimes$-$\Hom$ adjunction. The last equivalence follows from a standard computation of homology for Hopf algebras, see \cite[Proposition 1.2.8 (4)]{RRLocAnII}.

 Since $\bb{U}\subset \bb{G}$ is normal, we have that 
\[
\n{D}(G;\bb{U},\bb{Q}_p)=\bb{Q}_{p,\sol}[G]\otimes^L_{\bb{Q}_p[U]} \n{D}(\bb{U},\bb{Q}_p)=\n{D}(\bb{U},\bb{Q}_p) \otimes^L_{\bb{Q}_p[U]} \bb{Q}_{p,\sol}[G]. 
\]
Thus, we deduce that 
\[
\n{D}(G;\bb{U},\bb{Q}_p)\otimes_{\bb{Q}_{p,\sol}[H]}^L  \n{D}(\bb{H},\bb{Q}_p) = \n{D}(\bb{U}, \bb{Q}_p)\otimes^L _{\bb{Q}_{p,\sol}[U]}\bb{Q}_{p,\sol}[G]\otimes^L_{\bb{Q}_{p,\sol}[H]} \n{D}(\bb{H},\bb{Q}_p).
\]
By the semidirect product decomposition of $G$, the natural left $U$-equivariant and right $H$-equivariant map  
\[
\n{D}(\bb{U}, \bb{Q}_p)\otimes^L_{\bb{Q}_{p,\sol}[U]} \bb{Q}_{p,\sol}[G]\otimes^L_{\bb{Q}_{p,\sol}[H]} \n{D}(\bb{H},\bb{Q}_p)\to \n{D}(\bb{G}, \bb{Q}_p)
\]
is an equivalence, which implies that the map of left $G$-representations 
\[
\n{D}(G;\bb{U},\bb{Q}_p)\otimes_{\bb{Q}_{p,\sol}[H]}^L  \n{D}(\bb{H},\bb{Q}_p)\to \n{D}(\bb{G}, \bb{Q}_p)
\]
is an isomorphism. We deduce that 
\[
\begin{gathered}
(V^{R\bb{U}-an})^{R\bb{H}-an} = R\Gamma(H, R\Gamma(G, R\underline{\Hom}_{\bb{Q}_p}(\n{D}(G;\bb{U},\bb{Q}_p)\otimes_{\bb{Q}_{p,\sol}} \n{D}(\bb{H},\bb{Q}_p), V)  )) \\
 = R\Gamma(G, R\Gamma(H, R\underline{\Hom}_{\bb{Q}_p}(\n{D}(G;\bb{U},\bb{Q}_p)\otimes_{\bb{Q}_{p,\sol}} \n{D}(\bb{H},\bb{Q}_p), V)  ))  \\
= R\Gamma(G,  R\underline{\Hom}_{\bb{Q}_p}(\n{D}(G;\bb{U},\bb{Q}_p)\otimes_{\bb{Q}_{p,\sol}[H]}  \n{D}(\bb{H},\bb{Q}_p), V)  )\\  
\cong  R\Gamma(G, R\underline{\Hom}_{\bb{Q}_p}(\n{D}(\bb{G}, \bb{Q}_p), V )) \\ = V^{R\bb{G}-an}
\end{gathered} 
\]
proving what we wanted. 
\end{proof}

\begin{lem}\label{LemmaInvariantCompactPreservesLocAn}
Let $G$ be a compact $p$-adic Lie group and $H\subset G$ a normal subgroup. Let $V$ be a derived locally analytic representation of $G$. Then the $G/H$-representation $R\Gamma(H,V)$ is locally analytic. 
\end{lem}
\begin{proof}
Since $V$ is derived $G$-locally analytic, the natural map 
\[
R\Gamma(G,V\otimes^L_{\bb{Q}_{p},\sol} C^{la}(G,\bb{Q}_p)_{\star_1})\to V
\]
is an isomorphism of solid $G$-representations, where  the action of $G$ on the left term arises from the right regular action. Taking cohomology with respect to $H$ we see that  $R\Gamma(H,V)$  is naturally equivalent to 
\[
\begin{aligned}
R\Gamma(H,R\Gamma(G,V\otimes^L_{\bb{Q}_{p},\sol} C^{la}(G,\bb{Q}_p)_{\star_1}))  & = R\Gamma(G\times H,V\otimes^L_{\bb{Q}_{p},\sol} C^{la}(G,\bb{Q}_p)_{\star_1})  \\ 
& = R\Gamma(G, R\Gamma(H,V\otimes^L_{\bb{Q}_{p},\sol} C^{la}(G,\bb{Q}_p)_{\star_1} ))  \\ 
& = R\Gamma(G, V \otimes^{L}_{\bb{Q}_{p,\sol}} R\Gamma(H, C^{la}(G,\bb{Q}_p)))\\ 
& = R\Gamma(G,  V\otimes^L_{\bb{Q}_{p,\sol}} C^{la}(G/H,\bb{Q}_p))\\
& = R\Gamma(G/H, R\Gamma(H,V\otimes^L_{\bb{Q}_{p,\sol}} C^{la}(G/H,\bb{Q}_p) )) \\ 
& = R\Gamma(G/H, R\Gamma(H,V)\otimes^L_{\bb{Q}_{p,\sol}} C^{la}(G/H,\bb{Q}_p) )\\ 
& = R\Gamma(H,V)^{R(G/H)-la},
\end{aligned}
\]
where  the first and second  equivalences follow from Hochschild-Serre, the third one holds from the projection  formula of $H$-cohomology as the action on $V$ is trivial (direct consequence of \cite[Theorem 5.19]{RRLocallyAnalytic}), the fourth equivalence is clear, the fifth equivalence follows from  Hochschild-Serre, the sixth equivalence follows from the projection formula on $H$-cohomology and the fact that $H$ acts trivially on $C^{la}(G/H,\bb{Q}_p)$ for the left regular action (as $H$ is a normal subgroup), the last equivalence is the definition of derived $(G/H)$-locally analytic vectors. This shows that the map $R\Gamma(H,V)^{R(G/H)-la}\to R\Gamma(H,V)$ is an equivalence, and so that $R\Gamma(H,V)$ is a locally analytic $G/H$-representation as wanted.m  
\end{proof}

\subsection{Pro-Kummer-\'etale cohomology as condensed abelian groups}
\label{SubsecProKummerAsCondensed}

In order to justify some computations of derived locally analytic vectors of pro-Kummer-\'etale cohomology groups, we will need to promote the pro-Kummer-\'etale cohomology to condensed mathematics. We let $C$ be a complete algebraically closed non-archimedean extension of $\bb{Q}_p$, and let $\n{O}_{C}\subset C$ be its  valuation subring of power bounded elements (it will suffice for us to take $C=\bb{C}_p$).  Let $*_{\proet,1}$ and $*_{\proet,2}$  be the pro\'etale site of the point $*=\Spa(C,\n{O}_C)$ as in \cite{ScholzeHodgeTheory2013} and  \cite{scholze2022etale} respectively. The underlying categories of both pro\'etale sites are the same, they are just the category of profinite sets, however, the Grothendieck topologies differ in both cases. In $*_{\proet,1}$   covers  of profinite sets  are generated  by disjoint unions and open surjective  maps, while in $*_{\proet,2}$ covers are generated by disjoint unions and    surjective maps.   In particular, covers in $*_{\proet,1}$ are covers in $*_{\proet,2}$ and we have a morphism of sites $g:*_{\proet,2}\to *_{\proet,1}$.

Note that the (derived) category of condensed abelian groups \cite[Lecture II]{ClausenScholzeCondensed2019} is the (derived) category\footnote{The actual definition of condensed set involves some accessibility condition as in \cite[Proposition 2.9]{ClausenScholzeCondensed2019}.}  of abelian sheaves on $*_{\proet,2}$. 

We recall the following result that follows from the theory of \cite{DiaoLogarithmic2019}:

\begin{lem}\label{LemmaProperBaseChangeEtaleSheaves}
Let $X$ be a locally noetherian log adic space over $(C,\n{O}_C)$ and let $\nu_X: X_{\proket}\to X_{\ket}$ be the projection of sites. Then, for an object $\s{F}\in D^+(X_{\ket},\bb{Z})$ in the bounded-to-the-left derived category of Kummer-\'etale abelian sheaves of $X$,  the natural map $\s{F}\to R\nu_{X,*} \nu^{-1}_X \s{F}$ is a quasi-isomorphism. In other words, the pullback functor $\nu_X^{-1}: D^+(X_{\ket}, \bb{Z})\to D^+(X_{\proket}, \bb{Z})$ of derived abelian sheaves is fully faithful.

 Moreover, let $f:X\to Y$ be a qcqs map of locally noetherian log adic spaces, and let $f_{\proket}: X_{\proket}\to Y_{\proket}$ and $f_{\ket}:X_{\ket}\to Y_{\ket}$ be the natural map of sites. Then the natural transformation
\[
  \nu_Y^{-1} Rf_{\ket,*}  \s{F} \to   Rf_{\proket,*} \nu_{X}^{-1} \s{F}
\]
from $D^+(X_{\ket},\bb{Z})\to D^+(Y_{\proket},\bb{Z})$ is an equivalence. 
\end{lem}
\begin{proof}
The first claim when $\s{F}$ an abelian sheaf is \cite[Proposition 5.1.7]{DiaoLogarithmic2019}. For a general object in the derived category, it suffices to show that the map $\s{F}\to R\nu_{X,*} \nu^{-1}_X \s{F}$ is an equivalence in cohomology groups, this follows from  \cite[Proposition 5.1.6]{DiaoLogarithmic2019} since $H^i(\s{F})$ is the sheafification of the presheaf mapping $U$ to $H^i_{\ket}(U,\s{F})$.

 For the second claim, by localizing at $Y$ we can assume without loss of generality that both $Y$ and $X$ are qcqs. Let $U=\varprojlim_{i} U_i\in Y_{\proket}$ be qcqs, then we have that 
\[
\begin{aligned}
R\Gamma_{\proket}(U, \nu_{Y}^{-1} Rf_{\ket,*}\s{F}) & \cong \varinjlim_{i} R\Gamma_{\ket}(U_i, Rf_{\ket,*} \s{F}) \\
						& \cong \varinjlim_{i} R\Gamma_{\ket}(f^{-1}(U_i),\s{F}) \\ 
						& \cong R\Gamma_{\proket}(f^{-1}(U),\nu_X^{-1} \s{F})
\end{aligned}
\]
where in the first equivalence we use  \cite[Proposition 5.1.6]{DiaoLogarithmic2019}, the second equivalence is the composition of two right derived functors,   and the last equality follows from \cite[Proposition 5.1.6]{DiaoLogarithmic2019} and the fact that $f$ is qcqs. This proves the lemma. 
\end{proof}

We say that an object  $\s{F} \in D^+(X_{\proket},\bb{Z})$ is Kummer-\'etale if it is the pullback of an object in $D^+(X_{\ket},\bb{Z})$. The following proposition promotes some pro-Kummer-\'etale cohomologies to solid abelian groups. 

\begin{prop}
\label{LemmaSolidEnhancement}
Let $X$ be a log smooth adic space over $(C,\n{O}_C)$, let $U\in X_{\proket}$ be a qcqs object and let $\s{F}=\varinjlim_{i} \s{F}_i$ be a filtered colimit in $D^+(X_{\proket}, \bb{Z})$ where each $\s{F}_i$ is derived $p$-complete and such that the derived quotients  $\s{F}_i/^{\bb{L}} p$ are  Kummer-\'etale. Then $R\Gamma_{\proket}(U,\s{F})$ has a natural structure of solid abelian group.

 More precisely, consider the morphisms of sites
\[
f_U: U_{\proket}\to *_{\proet,1}  \leftarrow *_{\proet,2}:g.
\]
Then  $g^*Rf_{U,*} \s{F}$  is a solid object in $D(*_{\proet,2})=D(\mathrm{CondAb})$, that is, it belongs to the full subcategory $D(\mathrm{Solid})\subset D(\mathrm{CondAb})$ given by the derived category of solid abelian groups. 
\end{prop}
\begin{proof}
For convenience we work with the derived $\infty$-categories of abelian sheaves $\s{D}(U_{\proket},\bb{Z})$, $\s{D}(*_{\proet_1}, \bb{Z})$ and $\s{D}*_{(\proet_2},\bb{Z})$. The statement of the proposition does not depend of this change of framework, namely, being solid is a property that can be detected at the level of cohomology groups.  In this case, the object $Rf_{U,*}\s{F}$ is a sheaf on $*_{\proet,1}$ in the $\infty$-category $\mathcal{D}(*_{\proet,1}, \mathbb{Z})$.

  Let us first show the lemma when $U=\varprojlim_i U_i$ is qcqs written as a limit of  Kummer-\'etale maps $U_i\to X$ and $\s{F}$ is a Kummer-\'etale complex. We see the point $*=\Spa(C,\n{O}_C)$ with the trivial log structure, so that $*_{\proket}=*_{\proet,1}$ and $*_{\ket}=*_{\et}$. We claim that the sheaf $Rf_{U,*} \s{F}$ arises from the \'etale site of the point, this implies that it is a sheaf for the Grothendieck topology of $*_{\proet,2}$ thanks to the fully faithful embedding of \cite[Proposition 14.10]{scholze2022etale}. Since it is discrete (which is equivalent to arise as a pullback from the \'etale site of the point)  it is  also solid. Indeed, we can change $X$ by any of the  $U_i$ and suppose that $X$ is qcqs and that the maps $U_i\to X$ are qcqs. Let $f: X_{\proket}\to *_{\proket}$, then  $Rf_{U,*}\s{F}$ is the same as 
the derived pushforward along $f$ of the sheaf $\varinjlim_{i} Rj_{U_i,*}\s{F}|_{U_i}$ where $j_{U_i}:U_i\to X$ is the natural map. Since $X$ is qcqs we have that 
\[
Rf_* \s{F}=\varinjlim_{i} Rf_* Rj_{U_i,*}\s{F}|_{U_i} = \varinjlim_i Rf_{U_i,*} \s{F}|_{U_i}.
\]
By Lemma \ref{LemmaProperBaseChangeEtaleSheaves} each object $Rf_{U_i,*} \s{F}|_{U_i}$ arises from the \'etale site of the point, and so does their filtered colimit.

  Now, if $\s{F}$ is derived $p$-complete with Kummer-\'etale special fiber, then 
\[
Rf_{U,*} \s{F}  = Rf_{U,*} (R\varprojlim_n \s{F}/^{\bb{L}} p^n) =  R\varprojlim_nRf_{U,*}  ( \s{F}/^{\bb{L}} p^n)
\]
is a limit of sheaves for the Grothendieck topology of $*_{\proet,2}$ (as $Rf_{U,*}  ( \s{F}/^{\bb{L}} p^n)$ arises from the \'etale site of the point and then it is a sheaf for $*_{\proet,2}$ by the previous step), and so it is a sheaf. It is also a limit of discrete objects and so it is   solid   by stability under limits of solid objects \cite[Theorem 5.8]{ClausenScholzeCondensed2019}.  Finally, if $\s{F}=\varinjlim \s{F}_i$ is a filtered colimit of derived $p$-complete sheaves, then, since $U$ is qcqs,  we have that 
\[
g^* Rf_{U,*} \s{F} = \varinjlim_{i} g^* Rf_{U,*} \s{F}_i
\]
is a filtered colimit of solid  objects, and hence solid   by  stability under colimits of solid objects \cite[Theorem 5.8]{ClausenScholzeCondensed2019}. 
\end{proof}

\begin{remark}
Lemma \ref{LemmaSolidEnhancement} is \textit{ad. hoc.} for this paper, a good definition of condensed cohomology can be found in \cite[Section 2]{Bosco2021padic}. The reason why we use Lemma \ref{LemmaSolidEnhancement} is due to the difference between the pro-Kummer-\'etale topology of \cite{DiaoLogarithmic2019} and the pro\'etale topology of \cite{scholze2022etale}. A better way to solve this incompatibility is to construct a diamond functor from log adic spaces to $v$-stacks producing an equivalence between the Kummer-\'etale topos of $X$ and the \'etale topos of $X^{\lozenge}$, we do not pursue this idea in this paper. 
\end{remark}

\section{Equivariant vector bundles over flag varieties}
\label{Section:EquivariantFlag}

As a preparation for computing the geometric Sen operator of Shimura varieties we need to   construct some equivariant vector bundles on flag varieties.

\subsection{$\bbf{G}$-equivariant vector bundles on flag varieties}
\label{Subsec:GequiFlag}

Let $K$ be a field of characteristic $0$ and let $\bbf{H}$ be an algebraic group over $K$. Consider the algebraic stack $(\Spec K)/\bbf{H}$, see \cite[\href{https://stacks.math.columbia.edu/tag/026N}{Tag 026N}]{stacks-project}. In this paper an algebraic representation of $\bbf{H}$ is by definition a quasi-coherent  sheaf on $(\Spec K)/\bbf{H}$, equivalently, a co-module for the Hopf algebra $\s{O}(\bbf{H})$ of algebraic functions of $\bbf{H}$. If we want to stress that the algebraic representation  $V$ is finite dimensional, we  say that $V$ is a finite dimensional representation of $\bbf{H}$. Note that any algebraic representation is a union of finite dimensional representations.

 Let $\bbf{G}$ be a reductive group over $K$ and let $\mu:\bb{G}_m\to \bbf{G}$ be a minuscule cocharacter. We let $\bbf{P}^{\std}_{\mu}$ and $\bbf{P}_{\mu}$ denote the parabolic subgroups of $\bbf{G}$ parametrizing decreasing and increasing $\mu$-filtrations on $\bbf{G}$. We let $\bbf{M}_{\mu}=\bbf{P}_{\mu}\cap \bbf{P}^{\std}_{\mu}$ be the Levi factor, equivalently, $\bbf{M}_{\mu}$ is the centralizer of $\mu$ in $\bbf{G}$.  We denote by $\FL^{\std}=\bbf{G}/\bbf{P}^{\std}_{\mu}$ and  $\FL= \bbf{G}/\bbf{P}_{\mu}$ the flag varieties defined by $\bbf{P}^{\std}_{\mu}$ and $\bbf{P}_{\mu}$ respectively.  Let $\bbf{N}_{\mu}\subset \bbf{P}_{\mu}$ and $\bbf{N}^{\std}_{\mu}\subset \bbf{P}^{\std}_{\mu}$ be the unipotent radicals. We recall a classical fact about $\bbf{G}$-equivariant quasi-coherent sheaves of $\FL$. Let $\bbf{G}-\QCoh(\FL)$ be the category of $\bbf{G}$-equivariant quasi-coherent sheaves on $\FL$, and let $\Rep^{\alg}_{K}\bbf{P}_{\mu}$ be the category of $K$-linear algebraic representations of $\bbf{P}_{\mu}$. 

\begin{prop}\label{PPropEqGSheavesFL}
Let $\pi: \bbf{G}\to \FL$.  The pullback at the image of $1$ in $\FL$,  $\iota_1: \Spec K \to  \FL$, induces an equivalence of categories 
\[
\iota_1^{*}:\bbf{G}-\QCoh(\FL) \xrightarrow{\sim}  \Rep^{\alg}_{K}\bbf{P}_{\mu}. 
\]  Moreover, the inverse $\n{W}$ of $\iota_1^{*}$ is given by mapping a $\bbf{P}_{\mu}$-representation $V$ to the $\bbf{G}$-equivariant vector bundle $\n{W}(V)$ whose global sections at $U\subset \FL$ are given by 
\begin{equation}
\label{eqVBAssociatedToRep}
\n{W}(V)(U)= (\s{O}(\pi^{-1}(U))\otimes_K V)^{\bbf{P}_{\mu}},
\end{equation}
where $\s{O}(\pi^{-1}(U))$ are the algebraic functions of $\pi^{-1}(U)\subset \bbf{G}$ endowed with the right regular action,  and $\bbf{P}_{\mu}$ acts diagonally on the tensor product. The action of $\bbf{G}$ on $\n{W}(V)$ arises from the left regular action of $\bbf{G}$ on (the translations of) $\s{O}(\pi^{-1}(U))$. 
\end{prop}
\begin{proof}
The proposition can be proved by hand after unravelling the constructions. A more conceptual and direct proof follows from the isomorphism of Artin stacks
\[
(\Spec K)/\bbf{P}_{\mu} = (\bbf{G}\backslash \bbf{G})/\bbf{P}_{\mu} = \bbf{G}\backslash \FL, 
\]
and the fact that $\QCoh(*/\bbf{P}_{\mu})$ and  $\QCoh(\bbf{G}\backslash \FL)$ are  $\Rep^{\alg}_{K}\bbf{P}_{\mu}$ and $\bbf{G}-\QCoh(\FL)$ respectively. 
\end{proof}

\begin{remark}
In \eqref{eqVBAssociatedToRep} the invariants can be taken with respect to the $K$-points of $\bbf{P}_{\mu}$ since $K$ is infinite. More conceptually, the tensor $\s{O}(\pi^{-1}(U))\otimes_K V$ is an algebraic representation of $\bbf{P}_{\mu}$ (i.e. an object in $\QCoh(\Spec K/\bbf{P}_{\mu})$) and the $\bbf{P}_{\mu}$-invariants are the pushforward along the map $(\Spec K)/\bbf{P}_{\mu}\to \Spec K$. 
\end{remark}

\subsection{Some equivariant Lie algebroids}
\label{Subsec:EquiLieAlgebroid}

In the following we construct some $\bbf{G}$-equivariant Lie algebroids over $\FL$ that are the main players in the localization theory of Beilinson-Bernstein, we refer to \cite{BeilinsonBernstein} for more details.

Let $\f{g}$, $\f{p}_{\mu}$, $\f{n}_{\mu}$ and $\f{m}_{\mu}$ denote the Lie algebras of $\bbf{G}$, $\bbf{P}_{\mu}$, $\bbf{N}_{\mu}$ and $\bbf{M}_{\mu}$ respectively. Consider the action of $\f{g}$ on $\s{O}_{\FL}$ by taking derivations of the action of $\bbf{G}$, it defines a Lie algebroid $\f{g}^0=\s{O}_{\FL} \otimes_K \f{g}$ and an anchor map 
\[
\alpha: \f{g}^0 \to \n{T}_{\FL},
\]
where $\n{T}_{\FL}$ is the tangent space of $\FL$. 

The group $\bbf{P}_{\mu}$ acts on $\f{p}_{\mu}$, $\f{n}_{\mu}$ and $\f{m}_{\mu}$ via the adjoint action. By the equivalence of  Proposition \ref{PPropEqGSheavesFL}, one has a filtration of $\bbf{G}$-equivariant vector bundles
\begin{equation*}
\f{n}^0_{\mu}\subset \f{p}_{\mu}^0 \subset \f{g}^0 \mbox{ and } \f{m}^0_{\mu}= \f{p}^0_{\mu}/\f{n}^0_{\mu} 
\end{equation*}
corresponding to $\bbf{P}_{\mu}$-equivariant maps
\[
\f{n}_{\mu}\subset \f{p}_{\mu} \subset \f{g} \mbox{ and } \f{m}_{\mu}= \f{p}_{\mu}/\f{n}_{\mu} .
\]

Furthermore, the vector bundles $\f{n}^0_{\mu}$ and $\f{p}^0_{\mu}$ are ideals of $\f{g}^0$ and the anchor map $\alpha$ induces an isomorphism 
\begin{equation}
\label{eqIsoAnchormap}
\alpha:\f{g}^0/\f{p}^0_{\mu}\xrightarrow{\sim} \n{T}_{\FL}. 
\end{equation}

\subsection{Regular representation of $\bbf{P}_{\mu}$}
\label{Subsec:RegularRepP}

For a scheme $X$ we let $\s{O}(X)$ denote its algebra of global sections.  We finish with a slightly more explicit description of the  left regular representation of $\bbf{P}_{\mu}$ that will be used in Section \ref{Section:SenOperators}. Let $\s{O}(\bbf{P}_{\mu})$ be the left regular representation of $\bbf{P}_{\mu}$. The presentation $\bbf{P}_{\mu}=\bbf{N}_{\mu} \rtimes \bbf{M}_{\mu}$ as semi-direct product induces a decomposition 
 \begin{equation}
 \label{eqTensorDecomposition}
\s{O}( \bbf{P}_{\mu})\cong \s{O}(\bbf{N}_{\mu})\otimes_K \s{O}(\bbf{M}_{\mu}).
 \end{equation}
The decomposition \eqref{eqTensorDecomposition} is $\bbf{P}_{\mu}$-equivariant when the right hand side is endowed with the following action:

\begin{itemize}

\item $\bbf{P}_{\mu}$ acts on  $\s{O}(\bbf{M}_{\mu})$   via the left regular action of the  projection $\bbf{P}_{\mu}\to \bbf{M}_{\mu}$.

\item The action of $\bbf{P}_{\mu}$   on $\s{O}(\bbf{N}_{\mu})$ arises from the action of schemes $\bbf{P}_{\mu}\times \bbf{N}_{\mu}\to \bbf{N}_{\mu}$ given by $(p_{\mu},n_{\mu})\mapsto  n(p_{\mu})m(p_{\mu})n_{\mu} m(p_{\mu})^{-1}$, where $n_{\mu}\in \bbf{N}_{\mu}$ and $p_{\mu}=(n(p_{\mu}),m(p_{\mu}))\in \bbf{P}_{\mu}=\bbf{N}_{\mu}\rtimes \bbf{M}_{\mu} $. In particular, the restriction to $\bbf{M}_{\mu}$ is the natural adjoint action while the restriction to $\bbf{N}_{\mu}$ is the left regular action. 
\end{itemize}

Since $\bbf{M}_{\mu}$ is reductive, if $\bbf{M}_{\mu}$ is split over $K$,  the $\bbf{M}_{\mu}\times \bbf{M}_{\mu}$-representation $\s{O}(\bbf{M}_{\mu})$ is the direct sum of  $V\otimes V^{\vee}$ where $V$ runs over the irreducible representations of $\bbf{M}_{\mu}$ indexed by their highest weight, see \cite[II Proposition 4.20]{Jantzen}.  It is left to better describe $\s{O}(\bbf{N}_{\mu})$. 

The group $\bbf{N}_{\mu}$ is unipotent, thus the exponential map $\exp: \f{n}_{\mu}\to \bbf{N}_{\mu}$ is an isomorphism of schemes, cf. \cite[Proposition 15.31]{MilneAlggroups}. Moreover, since $\mu$ is minuscule, $\bbf{N}_{\mu}$ is abelian and $\exp$ is  an isomorphism of group schemes. We have an isomorphism of algebras
\[
\s{O}(\bbf{N}_{\mu}) \cong \Sym^{\bullet}_K \f{n}^{\vee}_{\mu}. 
\]
The natural filtration of the symmetric algebra induces a filtration $\s{O}(\bbf{N}_{\mu})^{\leq n}$ on $\s{O}(\bbf{N}_{\mu})$. The weight decomposition of $\s{O}(\bbf{N}_{\mu})$ with respect to $\mu$ implies that $\s{O}(\bbf{N}_{\mu})^{\leq n}$ is indeed a $\bbf{P}_{\mu}$-stable subrepresentation of $\s{O}(\bbf{N}_{\mu})$. We deduce the following proposition.

\begin{prop}
\label{PropEquivFLag}
The exponential map $\exp: \f{n}_{\mu}\to \bbf{N}_{\mu}$ induces a $\bbf{P}_{\mu}$-stable increasing filtration $\s{O}(\bbf{N}_{\mu})^{\leq n}$ of $\s{O}(\bbf{N}_{\mu})$. Moreover, the following hold:

\begin{enumerate}
 \item There are natural $\bbf{P}_{\mu}$-equivariant isomorphisms 
\[
\gr_n(\s{O}(\bbf{N}_{\mu})) \cong \Sym^n_K \f{n}^{\vee}_{\mu}. 
\] 

\item The natural map 

\[
\Sym_{K}^{\bullet} (\s{O}(\bbf{N}_{\mu})^{\leq 1})/(1-e(1))\to \s{O}(\bbf{N}_{\mu}),
\]
is a  $\bbf{P}_{\mu}$-equivariant isomorphism, where $1$ is the unit in the symmetric algebra and  $e(1)$ is the image of $1\in K$ along the natural inclusion  $e:K\to \s{O}(\bbf{N}_{\mu})^{\leq 1}$. 
\end{enumerate}

\end{prop}
\begin{proof}
The weight decomposition of $\s{O}(\bbf{P}_{\mu})$ with respect to $\mu$ shows that $\s{O}(\bbf{N}_{\mu})^{\leq n}$ is a $\bbf{P}_{\mu}$-stable filtration of $\s{O}(\bbf{N}_{\mu})$. Part (1) follows from  the definition of the exponential map and the fact that $\bbf{N}_{\mu}$ is abelian as $\mu$ is minuscule. For part (2), the map is clearly $\bbf{P}_{\mu}$-equivariant. It is an isomorphism since $\bbf{N}_{\mu}$ is isomorphic to the vector bundle $\f{n}$ and 
\[
\Sym_{L}^{\bullet}(\f{n}^{\vee})\cong \Sym_K^{\bullet}(K\oplus \f{n}^{\vee})/(1-e(1)) 
\]
where $1$ is the unit in the symmetric algebra and  $e(1)$ is the image of $1\in K$ along the natural inclusion  $e:K\to K\oplus \f{n}^{\vee}$. 
\end{proof}

\section{Shimura varieties and the Hodge-Tate period map}
\label{Section:Shimura}

In this section we introduce the standard theory of Shimura varieties following \cite{DeligneShimura} and \cite{MilneIntroShimura}. We also recall the definition of the Hodge-Tate period map for infinite level Shimura varieties, see \cite{ScholzeTorsion2015}  and \cite{CaraianiScholze2017}, constructed  in the most general form using the  Riemann-Hilbert correspondence of \cite{DiaoLogarithmicHilbert2018}. 

\subsection{Set up}
\label{Subsec:SetUp}

Let $\bbf{G}$ be a reductive group over $\bb{Q}$ and  $(\bbf{G}, X)$  a Shimura datum (\cite[Section 2.1.1]{DeligneShimura} or \cite[Definition  5.5]{MilneIntroShimura}).    Let $E/\bb{Q}$ be the reflex field of $(\bbf{G}, X)$.  For $K\subset \bbf{G}(\bb{A}_{\bb{Q}}^{\infty})$ a neat compact open   subgroup we let $\Sh_{K,E}$ denote the canonical model of the  Shimura variety at level $K$ over $\Spec E$.  All compact open  subgroups of $\bbf{G}(\bb{A}^{\infty}_{\bb{Q}})$  considered in this paper are supposed to be neat.  From now on we will fix $K^p\subset \bbf{G}(\bb{A}_{\bb{Q}}^{\infty,p})$ a  compact open  subgroup at level prime to $p$.  Given $K_p\subset \bbf{G}(\bb{Q}_p)$ a compact open  subgroup  we let $\Sh_{K^pK_p,E}$ denote the Shimura variety at level $K^pK_p$.   We will be interested in  the tower
\[
\{\Sh_{K^p K_p,E}\}_{K_p\subset \bbf{G}(\bb{Q}_p)}.  
\]

Let  $\bbf{Z}$ be the center of $\bbf{G}$ and $\bbf{Z}_c\subset \bbf{Z}$ its maximal $\bb{Q}$-anisotropic torus which is $\bb{R}$-split.   Given two levels $K'\subset K \subset \bbf{G}(\bb{A}_{\bb{Q}}^{\infty})$ with $K'$ normal in $K$,  the map of Shimura varieties $\Sh_{K',E}\to \Sh_{K,E}$ is  a finite \'etale Galois cover  with Galois  group isomorphic to $K/ (K'\cdot K\cap \overline{\bbf{Z}(\bb{Q})})$,  where $\overline{\bbf{Z}(\bb{Q})}$ is the closure of $\bbf{Z}(\bb{Q})$ in $\bbf{Z}(\bb{A}_{\bb{Q}}^{\infty})$,  cf \cite[Section 2.1.9]{DeligneShimura}.   Since $K$ is neat,  $K\cap \overline{\bbf{Z}(\bb{Q})}\subset \bbf{Z}_c(\bb{A}_{\bb{Q}}^{\infty})$, and so if $\bbf{Z}_c=0$ the map $\Sh_{K',E}\to \Sh_{K,E}$ is a $K/K'$-torsor.    In general,  the limit $\varprojlim_{K_p} \Sh_{K^pK_p,E}$ (considered just as a scheme) is a Galois cover of $\Sh_{K^pK_p,E}$ with group 
\[
\widetilde{K}_p:= K^pK_p /( K^p \cdot K^pK_p \cap  \overline{\bbf{Z}(\bb{Q})}).  
\]
Note that $\widetilde{K}_p$ is a quotient of $K_p$ and so it has a natural structure of a $p$-adic Lie group. Therefore,  if $K_p'\subset K_p$ is normal,  the Galois group of  $\Sh_{K^pK_p',E}\to \Sh_{K^pK_p,E}$ is given by $\widetilde{K}_p/ \widetilde{K}_p'$.  We shall write  $\widetilde{\f{g}} = \Lie \widetilde{K}_p$; since any other inclusion $\widetilde{K}_p'\subset \widetilde{K}_p$ is open, the Lie algebra is independent of the level $K_p$.

Let $\bbf{G}^{c}$ denote the quotient of $\bbf{G}$ by $\bbf{Z}_c$,  and let $\f{g}^c= \Lie \bbf{G}^c(\bb{Q}_p)$.  From our previous discussion,  there is a map $\widetilde{\f{g}}\to \f{g}^c$,  and the   obstruction  of this map  to be an isomorphism is given by  Leopoldt's conjecture.  Indeed,  this  map is an isomorphism of Lie algebras if and only if the closure of the image of $\bbf{Z}_c(\bb{Q})$ in $\bbf{Z}_c(\bb{Q}_p)$ is open.    Given an algebraic subgroup  $\bbf{H}$ of $\bbf{G}$ we denote by $\bbf{H}^c$ its image in $\bbf{G}^c$,  similarly for the subgroups $K\subset \bbf{G}(\bb{A}_{\bb{Q}}^{\infty})$,  $K^p$ and $K_p$.   Note that we have maps $K_{p}\to \widetilde{K}_p \to K_{p}^c$ whose kernels are central. 

Let $\mu$ be a $\bbf{G}(\bb{C})$-conjugacy class of Hodge-cocharacters, it is  defined over $E$ and so it gives rise flag varieties $\FL^{\std}$ and $\FL$. If $F/E$ is any field extension where the group is split,  we fix a representative $\mu:\bb{G}_{m,F}\to \bbf{G}_{F}$ of $\mu$.  By the axioms of Shimura varieties, $\mu$ is a minuscule cocharacter. We shall adopt the representation theory notation of Section \ref{Section:EquivariantFlag}. In particular $\bbf{P}^{\std}_{\mu}$ and $\bbf{P}_{\mu}$ denote the parabolic subgroups of $\bbf{G}_F$ parametrizing decreasing and increasing $\mu$-filtrations in $\bbf{G}_F$ respectively. The $F$-base change of the flag varieties admit the presentation   $\FL^{\std}_F=\bbf{G}_{F}/\bbf{P}^{\std}_{\mu}$ and  $\FL_F= \bbf{G}_F/ \bbf{P}_{\mu}$.

 Given a Shimura variety $\Sh_{K,E}$ and  an auxiliary  $K$-admissible cone decomposition $\Sigma$, we shall denote by $\Sh^{\tor}_{K,E}$ the toroidal compactification as in \cite{Pink1990ArithmeticalCO} (see \cite{FaltingsChai} for an algebraic construction in the Siegel case).  We let $D\subset \Sh^{\tor}_{K,E}$ be the boundary seen as a reduced divisor.   Since the choice of the cone decomposition will not be important in the paper we will omit any labeling referring to $\Sigma$. To guarantee that we can use the results of \cite{DiaoLogarithmicHilbert2018}, we need to make the following assumptions on the toroidal compactification: 
 
  \begin{itemize}
\item We fix a bottom level $K_p\subset  \bbf{G}(\bb{Q}_p)$ and consider a toroidal compactification $\Sh_{K^pK_p,E}^{\tor}$ that is smooth, projective, with boundary divisor given by normal crossings.  This can be guarantee thanks to \cite[Theorem 9.21]{Pink1990ArithmeticalCO}.

\item For $K_p'\subset K_p$ an open subgroup, there is a unique toroidal compactification of $\Sh_{K^pK_p',E}$ making the map $\Sh^{\tor}_{K^pK_p',E}\to \Sh^{\tor}_{K^pK_p,E}$ finite Kummer-\'etale. If in addition $K_p'$ is normal then this map is Galois with Galois group $\widetilde{K}_p/\widetilde{K}_p'$.  Indeed, the underlying scheme $\Sh^{\tor}_{K^pK_p',E}$ is constructed as the normalization of $\Sh^{\tor}_{K^pK_p,E}$ in $\Sh_{K^pK_p',E}$, and one defines the log structure as the one defined by the (reduced divisor given by the) preimage of $D$. By \cite[Section 6.7 (a)]{Pink1990ArithmeticalCO},  $\Sh^{\tor}_{K^pK_p',E}$ is precisely the toroidal compactification of \textit{loc. cit.} On the other hand, \cite[Theorem 10.2 and Remark 10.3]{MR4342362}  imply that the map $\Sh^{\tor}_{K^pK_p',E}\to \Sh^{\tor}_{K^pK_p,E}$ is finite Kummer-\'etale.

 \end{itemize}

  We must highlight that the toroidal compactification $\Sh^{\tor}_{K^pK_p'}$ might not be smooth, but by Abhyankar's lemma it becomes smooth locally in the Kummer-\'etale topology of $\Sh^{\tor}_{K^pK_p,E}$ (more precisely, it becomes smooth after extracting roots to the local coordinates defining the boundary divisor).

\subsection{Hodge-Tate period map}
\label{Subsec:HodgeTateMap}

Let us fix an isomorphism $\bb{C}\simeq \bb{C}_p$ which gives rise an inclusion $E\hookrightarrow \bb{C}_p$. We let $L\subset \bb{C}_p$ be a finite extension  of $\bb{Q}_p$ containing $E$ such that $\bbf{G}_{L}$ is split.    We let $\Shan_{K,L}$ denote the adic space over $\Spa(L,\n{O}_L)$ attached to $\Sh_{K,L}:=\Sh_{K,E}\times_{\Spec E} \Spec L$,  cf.  \cite{HuberEtaleCohomology}.  We also denote by $\Shan^{\tor}_{K,L}$ the adic space attached to the toroidal compactification and see it as a log adic space with log structure defined by the reduced normal crossing divisor of the boundary. By an abuse of notation we shall write by $\s{O}_{\Shan}$ and $\Omega^1_{\Shan}(\log)$ for the sheaf of functions and log differentials in the analytic and Kummer-\'etale  sites of  $\Shan^{\tor}_{K,L}$. Let $\nu: \Shan^{\tor}_{K,L,\proket}\to \Shan^{\tor}_{K,L,\ket}$ be the projection of sites, we also write $\s{O}_{\Shan}$ and  $\Omega^1_{\Shan}(\log)$ for  the inverse image along $\nu$ of the sheaf of functions and log differentials respectively.

We let $\Fl$ and $\Fl^{\std}$ be the adic analytification of the  $L$-base change of the algebraic flag varieties $\FL$ and $\FL^{\std}$ respectively.

Let $\Shan_{K^p,\infty,L}:= \varprojlim_{K_p} \Shan_{K^pK_p,L}$ be the infinite level Shimura variety considered  as an object in $\Shan_{K^pK_p,L,\proet}$. Similarly, we see $\Shan^{\tor}_{K^p,\infty,L}:= \varprojlim_{K_p} \Shan_{K_pK^p,L}^{\tor}$ as an object living in the pro-Kummer-\'etale site of $\Shan^{\tor}_{K^pK_p,L}$.  By construction, the map $\pi_{K_p}: \Shan_{K^p,\infty,L}\to \Shan_{K^pK_p,L}$ is a pro\'etale $\widetilde{K}_p$-torsor, analogously $\pi_{K_p}^{\tor}: \Shan_{K^p,\infty,L}^{\tor}\to \Shan_{K^pK_p,L}^{\tor}$ is a pro-Kummer-\'etale $\widetilde{K}_p$-torsor.  Then, for any finite dimensional $\bb{Q}_p$-linear  representation $V\in \Rep_{\bb{Q}_p} \bbf{G}^c$, we have attached a pro\'etale local system $V_{\et}$ on $\Shan_{K^pK_p,L}$, obtained from the constant $\widetilde{K}_p$-equivariant  pro\'etale sheaf $\underline{V}$ on $\Shan_{K^p,\infty,L}$ (see Definition \ref{DefinitionPeriodSheaves} (2)). Similarly, we have a pro-Kummer-\'etale local system $V_{\ket}$ on $\Shan^{\tor}_{K^pK_p,L}$ constructed using the torsor $\pi^{\tor}_{K_p}$.

Note that for $n\geq 1$ the quotients $V_{\et}/p^n$ and $V_{\ket}/p^n$ arise from the \'etale and Kummer-\'etale sites of the Shimura varieties.   Let $j_{K_p,\ket}: \Shan_{K^pK_p,L,\et}\to \Shan_{K^pK_p,L,\ket}^{\tor}$ be the natural morphism of sites. By the purity theorem \cite[Theorem 4.6.1]{DiaoLogarithmic2019}, the derived pushforward $Rj_{K_p,\ket,*} V_{\et}/p^n$ sits in degree $0$, and is equal to the Kummer-\'etale local system $V_{\ket}/p^n$.

To light the notation, we will use the subscript $\Shan$ instead of $\Shan^{\tor}_{K^pK_p,L}$ for the period sheaves of Definition \ref{DefinitionPeriodSheaves}. In this way, $\widehat{\s{O}}_{\Shan}$ is the completed structural sheaf of the pro-Kummer-\'etale site of the Shimura variety, and $\OBdR{\Shan}^+$ is the  big de Rham sheaf.

Let us recall the logarithmic $p$-adic Riemann-Hilbert correspondence for the local systems $V_{\ket}$ for $V\in \Rep_{\bb{Q}_p} \bbf{G}^c$ following \cite[Section 5.2]{DiaoLogarithmicHilbert2018} and \cite[Section 4.4.38]{BoxerPilloniHigherColeman}. After the discussion of \cite[Section 5.2]{DiaoLogarithmicHilbert2018}, the local systems $V_{\ket}$ have unipotent monodromy along the boundary.  By Theorem 5.3.1 of \textit{loc. cit}, the local systems $V_{\ket}$ are de Rham in the sense of Definition \ref{DefinitiondeRhamLocalSys}, with associated filtered log-connections $(V_{\dR}, \nabla, \Fil^{\bullet})$. By definition, we have a natural isomorphism on $V$
\begin{equation}
\label{eqRHLocalSystems}
V_{\ket}\otimes_{\bb{Q}_p} \OBdR{\Shan} \cong V_{\dR} \otimes_{\s{O}_{\Shan}} \OBdR{\Shan}
\end{equation}
compatible with connections and filtrations. The functor  $V\mapsto V_{\dR}$ from finite dimensional representations of $\bbf{G}^c$ to Hodge-filtered vector bundles with flat connection is then an exact $\otimes$-functor. By  Tannakian formalism, after forgetting the filtration and flat connection,  the functor $V\mapsto V_{\dR}$ defines a $\bbf{G}^c$-torsor $\bbf{G}_{\dR}^c$  over the adic space $\Shan_{K^pK_p,L}^{\tor}$ for the analytic topology. Moreover, by forgetting the flat connection but keeping the Hodge filtration,   one has a reduction of  $\bbf{G}_{\dR}^c$ to  a $\bbf{P}^{\std,c}_{\mu}$-torsor  that we denote by $\bbf{P}^{\std,c}_{\mu,\dR}$. The pushout $\bbf{M}^{c}_{\mu,\dR}= \bbf{M}^c_{\mu}\times^{\bbf{P}^{\std,c}_{\mu}}\bbf{P}^{\std,c}_{\mu,\dR}$ is the  $\bbf{M}^c_{\mu}$-torsor of automorphic vector bundles on the toroidal compactification of the Shimura variety, this is the $p$-adic base change of the canonical extension of the $\bbf{M}^c_{\mu}$-torsor of automorphic vector bundles as in \cite[Section 4]{HarrisFunctorialToroidal}.

By Lemma \ref{LemmaGradedpieces} and \eqref{eqRHLocalSystems}, we have the following relation between the graded pieces of the Hodge and Hodge-Tate filtrations for $V\in \Rep_{\bb{Q}_p} \bbf{G}^c$:

\begin{equation}
\label{eqGradedLocSys}
\gr_j(V_{\ket}\otimes_{\bb{Q}_p} \widehat{\s{O}}_{\Shan}) \cong \gr^j(V_{\dR})\otimes_{\s{O}_{\Shan}} \widehat{\s{O}}_{\Shan}(-j). 
\end{equation}

The relation \eqref{eqGradedLocSys} implies the following theorem which is a consequence of \cite[Theorem 4.6.1]{DiaoLogarithmicHilbert2018}; see also \cite[Theorem 4.4.40]{BoxerPilloniHigherColeman}.

\begin{theo}
\label{TheoHodgeTatePeriod}
The $\widetilde{K}_p$-torsor $\Shan^{\tor}_{K^p,\infty,L}\to \Shan^{\tor}_{K^pK_p,L}$ together with the Hodge-Tate filtration of the local systems $V_{\ket}$ for $V\in \Rep_{\bb{Q}_p}\bbf{G}^c$ define a $K_p$-equivariant morphism of ringed sites
\begin{equation}
\label{eqHTMap}
\pi_{\HT}^{\tor}: (\Shan^{\tor}_{K^p,\infty,L,\proket},\widehat{\s{O}}_{\Shan}) \to (\Fl_{\an}, \s{O}_{\Fl}).
\end{equation}
Moreover, let   $\bbf{M}_{\mu,\Fl}^c= \bbf{G}_E^c/\bbf{N}^c_{\mu}$ be the $\bbf{M}_{\mu}^c$-torsor  over $\Fl$. There is a natural $\widetilde{K}_p$-equivariant isomorphism of $\bbf{M}_{\mu}^c$-torsors 
\begin{equation}
\label{eqEquivTorsors}
\pi_{\HT}^{\tor,*}(\bbf{M}_{\mu,\Fl}^{c})\cong \pi_{K_p}^*(\bbf{M}^c_{\mu,\dR})\times^{\bb{G}_m,\mu} \bb{G}_m(-1)
\end{equation}
where  $\bb{G}_m(-1)= \mathrm{Isom}_{\widehat{\s{O}}_{\Shan}}( \widehat{\s{O}}_{\Shan}, \widehat{\s{O}}_{\Shan}(-1))$ is the $(-1)$-Hodge-Tate twist of $\bb{G}_m$ in $\Shan^{\tor}_{K^p,\infty,L,\proket}$, and $\mu: \bb{G}_m\to \pi_{K_p}^*(\bbf{M}^c_{\mu,\dR})$ is the immersion along the Hodge cocharacter.
\end{theo}

\begin{remark}
Both sites in \eqref{eqHTMap} have a basis consisting on the spectrum of sous-perfectoid rings. In that situation, the Tannakian formalism  is discussed in \cite[Appendix to Lecture 19]{ScholzeWeinspadicgeometry}. 
\end{remark}

\begin{proof}
Let $(\Spa(R,R^+),\n{M})\in \Shan^{\tor}_{K^p,\infty,L,\proket}$ be a log affinoid perfectoid. For $V\in \Rep_{\bb{Q}_p} \bbf{G}^c$ we have a natural $\widetilde{K}_p$-equivariant trivialization 
\[
V_{\ket}|_{\Shan^{\tor}_{K^p,\infty,L}}=\underline{V}.
\]
This shows that 
\[
(V_{\proket}\otimes_{\bb{Q}_p} \widehat{\s{O}}_{\Shan})(\Spa(R,R^+),\n{M})= V\otimes_{\bb{Q}_p} R. 
\]
Then the Hodge-Tate filtration defines a $\widetilde{K}_p$-equivariant increasing $\mu$-filtration of $(V_{\ket}\otimes_{\bb{Q}_p} \widehat{\s{O}}_{\Sh})|_{\Shan_{K^p,\infty,L}^{\tor}}$. Indeed, this follows form \eqref{eqRHLocalSystems} and \eqref{eqGradedLocSys} and the fact that the Hodge filtration is a decreasing $\mu$-filtration.  We obtain an increasing $\mu$-filtration on $V\otimes_{\bb{Q}_p} R$. This produces the morphism \eqref{eqHTMap} of ringed sites  as wanted. The $\widetilde{K}_p$-equivariance follows from the $\widetilde{K}_p$-equivariance of the Hodge-Tate filtration. 

Finally, the statement about the equivalence  \eqref{eqEquivTorsors} of  torsors follows from  \eqref{eqGradedLocSys} and  Tannakian formalism. 
\end{proof}

\section{Geometric Sen operator of Shimura varieties}
\label{Section:SenOperators}

The goal of this section is to compute the geometric Sen operator of the $\widetilde{K}_p$-torsor $\Shan^{\tor}_{K^p,\infty,L}\to \Shan^{\tor}_{K^pK_p,L}$ of \cite[Theorem 3.3.4]{RCGeoSenTheory}, in terms of $\bbf{G}^c$-equivariant vector bundles of $\Fl$ and the Hodge-Tate period map, cf Sections  \ref{Subsec:PullbacksEquivariant} and  \ref{Subsec:GeometricSen}.  We keep the representation theory notation of Section \ref{Section:EquivariantFlag}.

\subsection{Pullbacks of equivariant vector bundles}
\label{Subsec:PullbacksEquivariant}
 
We have constructed the Hodge-Tate period map as a $K_p$-equivariant morphism of ringed sites \eqref{eqHTMap}. There is also a map of ringed sites
\begin{equation}
\label{eqMapFlagsSites}
(\Fl_{\an}, \s{O}_{\Fl})\to (\FL_{\Zar},\s{O}_{\FL})
\end{equation}
from the analytic flag variety to the schematic flag variety. In particular, we can take pullbacks by $\pi^{\tor}_{\HT}$ of $\bbf{G}^c$-equivariant sheaves over $\FL$.  The best way to define a category of (equivariant) quasi-coherent sheaves on rigid spaces is by using the language of condensed mathematics and analytic geometry of Clausen and Scholze \cite{ClausenScholzeCondensed2019,ClauseScholzeAnalyticGeometry}.    
 For us, it will suffice to consider sheaves on $\FL$ which are filtered colimits of vector bundles. In particular, their pullback to the rigid variety will be also a filtered colimit of vector bundles.

  Our strategy to compute the geometric Sen operator of Shimura varieties is to describe the pullbacks  of $\bbf{G}^c$-equivariant maps along $\pi_{\HT}^{\tor}$ in terms of vector bundles over the Shimura varieties and the Faltings extension.   By Proposition \ref{PropEquivFLag} the category of $\bbf{G}^c$-equivariant quasi-coherent sheaves on $\FL$ is equivalent to the category of algebraic $\bbf{P}^c_{\mu}$-representations. Therefore, since any finite dimensional representation of $\bbf{P}^c_{\mu}$ embeds into the (left) regular representation $\s{O}(\bbf{P}^c_{\mu})$, we shall focus only on this last case.

 The equation \eqref{eqTensorDecomposition} shows that $\s{O}(\bbf{P}^c_{\mu})$ has a $\bbf{P}^c_{\mu}$-equivariant decomposition $\s{O}(\bbf{P}^c_{\mu})= \s{O}(\bbf{N}^c_{\mu})\otimes_{L} \s{O}(\bbf{M}^c_{\mu})$ which induces a $\bbf{G}^c$-equivariant isomorphism of quasi-coherent sheaves on $\Fl$
 \begin{equation}
 \label{eqSheavesequivON}
 \n{W}(\s{O}(\bbf{P}^c_{\mu}))\cong \n{W}(\s{O}(\bbf{N}^c_{\mu}))\otimes_{\s{O}_{\Fl}} \n{W}( \s{O}(\bbf{M}^c_{\mu})).
 \end{equation}
The action of $\bbf{P}^c_{\mu}$ on $\s{O}(\bbf{M}^c_{\mu})$ is via the left regular representation of the Levi, and the action on $\s{O}(\bbf{N}^c_{\mu})$ is as in Section \ref{Subsec:RegularRepP} (in particular it is not the adjoint action).

By Theorem \ref{TheoHodgeTatePeriod} and the Peter-Weyl decomposition of $\s{O}(\bbf{M}^c_{\mu})$, we already know how to describe the pullback $\pi_{\HT}^{\tor,*} (\n{W}(\s{O}(\bbf{M}^c_{\mu})))$ in terms of automorphic vector bundles explicitly. 

\begin{definition}
\label{DefinitionAutoVB}
Let $W\in \Rep_L \bbf{M}^c_{\mu}$ be a finite dimensional  algebraic representation of the Levi subgroup. We let $W_{\Hd}$ denote the automorphic vector over $\Shan^{\tor}_{K^pK_p,L}$ defined by $W$ via the torsor $\bbf{M}^c_{\mu,\dR}$. 
\end{definition}

The following is a direct consequence of Theorem \ref{TheoHodgeTatePeriod}.

\begin{cor}
\label{CoroHodgeComparison}
Let $W\in \Rep_L \bbf{M}^c_{\mu}$ be an irreducible  representation of $\mu$-weight $\mu(W)\in \bb{Z}$. Then there is a natural $K_{p}$-equivariant isomorphism of $\widehat{\s{O}}_{\Shan}$-modules on $(\Shan^{\tor}_{K^p,\infty,L,\proket},  \widehat{\s{O}}_{\Shan})$
\begin{equation}
\label{eqIsosheaves}
\pi_{\HT}^{\tor,*}(\n{W}(W))\cong \pi_{K_p}^*(W_{\Hd})\otimes_{\widehat{\s{O}}_{\Shan}} \widehat{\s{O}}_{\Shan} (-\mu(W))
\end{equation}
where the twist in the right hand side is a Tate twist. 
\end{cor}

We will need to make more explicit the isomorphism \eqref{eqIsosheaves} for the graded pieces of the adjoint  representation of $\bbf{G}^c$.  Let $\f{g}^c=\Lie \bbf{G}^c$, and let $\f{n}^c_{\mu}\subset \f{p}^c_{\mu} \subset \f{g}^c$ be the Lie algebras of $\bbf{N}^c_{\mu}$ and $\bbf{P}^{c}_{\mu}$ respectively. Note that the natural surjective map $\bbf{P}_{\mu}\to \bbf{P}^c_{\mu}$ induces an isomorphism on radicals $\bbf{N}_{\mu}\cong \bbf{N}_{\mu}^c$ so that $\f{n}_{\mu}\cong \f{n}_{\mu}^c$.  We denote in a similar way the Lie algebras $\f{p}^{\std,c}_{\mu}$ and $\f{n}^{\std,c}_{\mu}$ of the opposite parabolic and unipotent radical. Finally, we  let $\f{m}^c_{\mu}=   \f{p}^c_{\mu}/\f{n}^{c}_{\mu} = \f{p}^{\std,c}_{\mu}/ \f{n}^{\std,c}_{\mu}$ be the Lie algebra of the Levi quotient.

Similarly, let $\f{g}^{\mathrm{der}}$  be the derived Lie algebra of $\f{g}$ and let $\f{n}_{\mu}\subset \overline{\f{p}}_{\mu}\subset \f{g}^{\mathrm{der}}$ be its $\bbf{P}_{\mu}$-filtration with Levi quotient $\overline{\f{m}}_{\mu}=\overline{\f{p}}_{\mu}/\f{n}_{\mu}$ (resp. $\f{n}^{\std}_{\mu}\subset \overline{\f{p}}_{\mu}^{\std}\subset \f{g}^{\mathrm{der}}$ for the $\bbf{P}^{\std}_{\mu}$-filtration).  We let $\f{n}_{\mu}^0$, $\overline{\f{p}}_{\mu}^0$, $\overline{\f{m}}_{\mu}^0$ and $\f{g}^{\mathrm{der},0}$ be the associated equivariant sheaves over the flag variety (resp. for $\f{n}_{\mu}^{\std, 0}$, $\overline{\f{p}}_{\mu}^{\std,0}$ and the standard flag variety). Notice that we have a natural isomorphism $\f{g}^{\mathrm{der}}/\overline{\f{p}}_{\mu} = \f{g}^c/\f{p}^c_{\mu}$.

\begin{prop}
\label{PropKodairaSpencer}
Let $\f{g}^{\mathrm{der}}$ be the derived algebra of $\f{g}$ and let  $\f{g}^{\mathrm{der}}_{\dR}$ be the vector bundle with filtered log connection attached to $\f{g}^{\mathrm{der}}$. Then the log connection 
\begin{equation}
\label{eqLocNablag}
\nabla: \f{g}^{\mathrm{der}}_{\dR} \to \f{g}^{\mathrm{der}}_{\dR}\otimes_{\s{O}_{\Shan}} \Omega^{1}_{\Shan}(\log)
\end{equation}
induces a Kodaira-Spencer isomorphism 
\begin{equation}
\label{eqKS}
\KS: \gr^{1}(\f{g}^{\mathrm{der}}_{\dR})\xrightarrow{\sim} \Omega^1_{\Shan}(\log). 
\end{equation}
In particular, we have a $K_p$-equivariant isomorphism of $\widehat{\s{O}}_{\Shan}$-modules on $\Shan^{\tor}_{K^p,\infty,L,\proket}$
\[
\KS: \pi_{\HT}^*(\f{g}^{\mathrm{der},0}/\overline{\f{p}}_{\mu}^0) \xrightarrow{\sim} \pi_{K_p}^*(\Omega^1_{\Shan}(\log)) \otimes_{\widehat{\s{O}}_{\Shan}} \widehat{\s{O}}_{\Shan}(-1).
\]

\end{prop}
\begin{proof}

Let us briefly recall the construction of the Kodaira-Spencer map. The adjoint representation $\f{g}^{\mathrm{der}}$ has weights $[-1,1]$, so that $\f{g}^{\mathrm{der}}_{\dR}$ has Hodge filtration concentrated in degrees $[-1,1]$. Moreover, it has graded pieces
\[
\gr^i \f{g}^{\mathrm{der}}_{\dR} =\begin{cases} 
\f{n}^{\std}_{\mu,\Hd} & i=1, \\ 
\overline{\f{m}}_{\mu,\Hd} & i=0, \\
(\f{g}^{\mathrm{der}}/\overline{\f{p}}_{\mu}^{\std})_{\Hd} & i= -1. 
\end{cases}
\]

By Griffiths transversality, we have an $\s{O}_{\Shan}$-linear map for the $\gr^1$ graded piece of \eqref{eqLocNablag}
\begin{equation}
\label{equationKSmap1}
\f{n}^{\std}_{\mu,\Hd} \xrightarrow{\overline{\nabla}} \overline{\f{m}}_{\mu,\Hd}\otimes_{\s{O}_{\Shan}} \Omega^1_{\Shan}(\log).
\end{equation}

 Since the functor $W\mapsto W_{\Hd}$ from $\bbf{M}^c_{\mu}$-representations to vector bundles is an exact $\otimes$-functor, taking adjoints we get a map 
\begin{equation}
\label{eqKSmap2}
(\overline{\f{m}}^{\vee}_{\mu}\otimes_{\s{O}_{\Shan}} \f{n}^{\std}_{\mu})_{\Hd} \to \Omega^1_{\Shan}(\log).
\end{equation}
The natural  adjoint action $\overline{\f{m}}_{\mu}\otimes_L \f{n}^{\std}_{\mu} \to \f{n}^{\std}_{\mu}$ has an adjoint  map
\begin{equation}
\label{eqAdjointmonn}
\f{n}^{\std}_{\mu} \to \overline{\f{m}}^{\vee}_{\mu} \otimes_{L} \f{n}^{\std}_{\mu}.
\end{equation}
Precomposing \eqref{equationKSmap1} with the functor $(-)_{\Hd}$ applied to \eqref{eqAdjointmonn} we get the desired map \eqref{eqKS}. 

To prove that $\KS$ is an isomorphism we can use GAGA twice, from the rigid space to the scheme and the scheme to the complex analytic space (via the fixed isomorphism $\bb{C}_p\simeq \bb{C}$),  and  then prove it over the complex analytic realization of the  toroidal compactification of the Shimura variety. Indeed, by \cite[Theorem 5.3.1]{DiaoLogarithmicHilbert2018} the $p$-adic Riemann-Hilbert correspondence and the complex analytic Riemann-Hilbert correspondence are compatible with the Betti vs \'etale comparison of local systems. In particular, the map $\KS$ can be constructed purely over the complex analytic space in the same way. 

On the other hand, by the canonical extensions of the automorphic vector bundles of \cite[Theorem 4.2 and Proposition 4.4]{HarrisFunctorialToroidal}, it suffices to prove that they are isomorphic in the open complex analytic Shimura variety.  Recall that
\[
\Sh_{K,E}(\bb{C})= \bbf{G}(\bb{Q})\backslash (X\times \bbf{G}(\bb{A}^f_{\bb{Q}}))/K
\]
where $(\bbf{G},X)$ is the Shimura datum, and that there is a $\bbf{G}(\bb{R})$-equivariant holomorphic Borel embedding
\[
\pi_{B}:X \hookrightarrow \FL^{\std}(\bb{C}).
\]
 Since the vector bundle with  connection $\f{g}^{\mathrm{der}}_{\dR}$ and its Hodge filtration are constructed via descent from $X$ and $\pi_{B}$, it suffices to show that the analogue of the  construction  of the Kodaira-Spencer map over $\FL^{\std}$ is an isomorphism. But now it is classical to see that the  construction above is equivalent (using the Killing form of the semisimple Lie algebra $\f{g}^{\mathrm{der}}$) to  the dual of the quotient of the  anchor map \eqref{eqIsoAnchormap}  which is an isomorphism. 
\end{proof}

Having understood the pullback of the  semisimplification of \eqref{eqSheavesequivON}, it remains to compute the pullback of its unipotent part.

\begin{theo}
\label{TheoSenOperator}
There is a natural $\widetilde{K}_p$-equivariant isomorphism of $\widehat{\s{O}}_{\Shan}$-algebras over $\Shan^{\tor}_{K^p,\infty,L,\proket}$ 
\begin{equation}
\label{eqIsoAlgebrasSenShimura}
\OC|_{\Shan_{K^p,\infty,L}^{\tor}} \cong \pi_{\HT}^{\tor, *}(\n{W}(\s{O}(\bbf{N}^c_{\mu})))
\end{equation}
where the $\bbf{P}^c_{\mu}$-action on $\s{O}(\bbf{N}^c_{\mu})$ is as in Section \ref{Subsec:RegularRepP}. More precisely, we have a natural $\widetilde{K}_p$-equivariant isomorphism of extensions
\[
\begin{tikzcd}
0 \ar[r]  &  \widehat{\s{O}}_{\Shan^{\tor}_{K^p,\infty,L}}\ar[r] \ar[d, "\id"] & \pi^{\tor,*}_{\HT}(\n{W}(\s{O}(\bbf{N}^c_{\mu})^{\leq 1})) \ar[r]  \ar[d, "\alpha"]& \pi^{\tor,*}_{\HT}(\f{n}^{c,0,\vee}_{\mu}) \ar[r] \ar[d, "-\KS"] & 0 \\ 
0 \ar[r] &  \widehat{\s{O}}_{\Shan^{\tor}_{K^p,\infty,L}} \ar[r] & \gr^1 \OBdr|_{\Shan^{\tor}_{K^p,\infty,L}}(-1) \ar[r, "\overline{\nabla}"] &  \Omega^1_{\Shan}(\log)\otimes_{\s{O}_{\Shan}} \widehat{\s{O}}_{\Shan^{\tor}_{K^p,\infty,L}} (-1) \ar[r] & 0
\end{tikzcd}
\]
where $\KS$ is the Kodaira-Spencer map of Proposition \ref{PropKodairaSpencer} and we have identified $\f{n}^{c,\vee}_{\mu} \cong \f{g}^{c}/\f{p}^{c}_{\mu}$ via the  Killing form of $\f{g}^{\mathrm{der}}$. 

\end{theo}

\begin{proof}
Let $\s{F}$ be a $K_p^c$-equivariant  sheaf on $\Fl$ arising from a $\bbf{G}^c$-equivariant quasi-coherent sheaf on $\FL$.  Throughout this proof we will also denote by $\pi^{\tor,*}_{\HT}(\s{F})$ the pro-Kummer-\'etale sheaf over $\Shan^{\tor}_{K^pK_p,L}$ obtained via descent from the $\widetilde{K}_p$-equivariant $\widehat{\s{O}}_{\Shan}$-module over $\Shan^{\tor}_{K^p,\infty,L,\proket}$. In the following we identify $\bbf{N}_{\mu}\cong \bbf{N}^c_{\mu}$ and study the representation $\s{O}(\bbf{N}_{\mu})$ instead. 

   Let $\f{g}^{\mathrm{der}}_{\ket}$ and $\f{g}^{\mathrm{der}}_{\dR}$ be the pro-Kummer-\'etale local system  and  vector bundle with filtered  log connection attached to the adjoint representation $\f{g}^{\mathrm{der}}$ respectively.  By the Riemann-Hilbert correspondence \eqref{eqRHLocalSystems} we have an isomorphism of pro-Kummer-\'etale sheaves over $\Shan_{K_pK^p,L}$
\[
\f{g}^{\mathrm{der}}_{\ket}\otimes_{\bb{Q}_p} \OBdR{\Shan} \cong \f{g}^{\mathrm{der}}_{\dR}\otimes_{\s{O}_{\Shan}} \OBdR{\Shan}
\]
compatible with fitrations and connections. Let us write $\bb{M}=\f{g}^{\mathrm{der}}_{\ket}\otimes_{\bb{Q}_p} \bb{B}^+_{\dR}$ and $\bb{M}^0=(\f{g}^{\mathrm{der}}_{\dR}\otimes_{\s{O}_{\Shan}} \OBdR{\Shan}^+)^{\nabla=0}$ for the two $\bb{B}^+_{\dR}$ lattices in $\f{g}^{\mathrm{der}}_{\ket}\otimes_{\bb{Q}_p} \bb{B}_{\dR}$. We endow the $\bb{B}_{\dR}^+$-lattices with the natural $(\ker \theta)$-adic filtration  where  $\theta: \bb{B}_{\dR}^+\to \widehat{\s{O}}_{\Shan}$ is  Fontaine's map.

  By definition of $\pi_{\HT}^{\tor}$, the Hodge-Tate filtration of $\f{g}^{\mathrm{der}}_{\ket}\otimes_{\bb{Q}_p} \widehat{\s{O}}_{\Shan}$ is the pullback along the Hodge-Tate map of the filtration $\f{n}^{0}_{\mu}\subset \overline{\f{p}}_{\mu}^0 \subset \f{g}^{\mathrm{der},0}$. 

On the other hand,  since $\f{g}^{\mathrm{der}}_{\dR}$ has Hodge filtration concentrated in $[-1,1]$, the connection of $\f{g}^{\mathrm{der}}_{\dR}\otimes_{\s{O}_{\Shan}} \OBdR{\Shan}^+$ induces a short exact sequence after taking $\gr^0$ pieces
\begin{equation}
\label{eqSESFil0}
0\to \frac{\bb{M}\cap \bb{M}_0}{\Fil^1 \bb{M}\cap \bb{M}_0}\to \gr^0(\f{g}^{\mathrm{der}}_{\dR} \otimes_{\s{O}_{\Shan}} \OBdR{\Shan}^+)\to \gr^0( \f{g}^{\mathrm{der}}_{\dR} \otimes_{\s{O}_{\Shan}} \OBdR{\Shan}^+\otimes_{\s{O}_{\Shan}} \Omega^1_{\Shan}(\log)) \to 0. 
\end{equation}
By  Lemma \ref{LemmaGradedpieces} (which is based on  \cite[Proposition 7.9]{ScholzeHodgeTheory2013}) and Corollary \ref{CoroHodgeComparison}  we have natural isomorphisms
\begin{align*}
 \frac{\bb{M}\cap \bb{M}_0}{\Fil^1 \bb{M}\cap \bb{M}_0}& = \pi_{\HT}^{\tor,*}(\overline{\f{p}}_{\mu}^0), \\
 \gr^0(\f{g}^{\mathrm{der}}_{\dR} \otimes_{\s{O}_{\Shan}} \OBdR{\Shan}^+)&=\gr^{-1}(\f{g}^{\mathrm{der}}_{\dR} ) \otimes_{\s{O}_{\Shan}} \gr^1 \OBdR{\Shan}^+ \oplus  \gr^0(\f{g}^{\mathrm{der}}_{\dR} ) \otimes_{\s{O}_{\Shan}} \widehat{\s{O}}_{\Shan} ,\\
 & \cong \pi^{\tor,*}_{\HT}(\f{n}^{0}_{\mu})\otimes_{\widehat{\s{O}}_{\Shan}}  \gr^{1} \OBdR{\Shan}^+(-1) \oplus  \pi_{\HT}^{\tor,*}(\overline{\f{m}}^{0}_{\mu}) \\
 \gr^0(\f{g}^{\mathrm{der}}_{\dR} \otimes_{\s{O}_{\Shan}}  \OBdR{\Shan}^+\otimes_{\s{O}_{\Shan}} \Omega^1_{\Shan}(\log))& = \gr^{-1}(\f{g}^{\mathrm{der}}_{\dR} )\otimes_{\s{O}_{\Shan}} \widehat{\s{O}}_{\Shan} \otimes_{\s{O}_{\Shan}} \Omega^1_{\Shan}(\log)  \\
 & \cong \pi_{\HT}^{\tor,*}(\f{n}^{0}_{\mu})(-1)\otimes_{\s{O}_{\Shan}} \Omega^1_{\Shan}(\log).
\end{align*}

Therefore, the short exact sequence \eqref{eqSESFil0} is nothing but 
\[
\begin{gathered}
0\to \pi_{\HT}^{\tor, *}(\overline{\f{p}}_{\mu}^0)\to \pi^{\tor,*}_{\HT}(\f{n}^{0}_{\mu})\otimes_{\widehat{\s{O}}_{\Shan}}  \gr^{1} \OBdR{\Shan}^+(-1) \oplus  \pi_{\HT}^{\tor,*}(\overline{\f{m}}^{0}_{\mu})  \\ \xrightarrow{\overline{\nabla}\oplus \widetilde{\KS}}   \pi_{\HT}^{\tor,*}(\f{n}^{0}_{\mu}) (-1)\otimes_{\s{O}_{\Shan}} \Omega^1_{\Shan}(\log) \to 0 
\end{gathered}
\]
where $\overline{\nabla}$ is the reduction of the connection of $\OBdR{\Shan}$ tensored with $ \pi^{\tor,*}_{\HT}(\f{n}^{0}_{\mu})$, and $\widetilde{\KS}$ is the twisted Kodaira-Spencer map obtained by base change from  the graded zero-th piece of $\nabla: \f{g}^{\mathrm{der}}_{\dR}\to \f{g}^{\mathrm{der}}_{\dR}\otimes_{\s{O}_{\Shan}} \Omega^1_{\Shan}(\log)$ (this map is the adjoint of the map \eqref{equationKSmap1} after identifying $\overline{\f{m}}^{\vee}_{\mu}\cong \overline{\f{m}}_{\mu}$ and $\f{n}^{\vee}_{\mu} \cong \f{g}^{\mathrm{der}}/\overline{\f{p}}_{\mu}$ via the Killing form of $\f{g}^{\mathrm{der}}$). 

This extension defines a class 
\[
\eta\in \Ext^1_{\widehat{\s{O}}_{\Shan}}( \pi_{\HT}^{*}(\f{n}^{0}) (-1)\otimes_{\s{O}_{\Shan}} \Omega^1_{\Shan}(\log), \pi_{\HT}^{\tor, *}(\overline{\f{p}}_{\mu}^0)) .
\]
   Tensoring with  $\pi_{\HT}^{\tor,*}(\f{n}^{0,\vee}_{\mu})$ we obtain a class  
   \begin{equation}\label{eqClasstensor}
   \tilde{\eta} \in \Ext^1_{\widehat{\s{O}}_{\Shan}}(\pi^{\tor,*}_{\HT}(\f{n}^{0}_{\mu}\otimes_{\s{O}_{\Fl}} \f{n}_{\mu}^{0,\vee})(-1)\otimes_{\s{O}_{\Shan}} \Omega^1_{\Shan}(\log),  \pi_{\HT}^{\tor,*}(\overline{\f{p}}_{\mu}^0\otimes_{\s{O}_{\Fl}} \f{n}_{\mu}^{0,\vee})).
   \end{equation}
       We have $\bbf{P}^c_{\mu}$-equivariant maps $L \hookrightarrow \f{n}_{\mu}\otimes_L\f{n}^{\vee}_{\mu} $  and $\overline{\f{p}}_{\mu}\otimes_L \f{n}^{\vee}_{\mu}  \twoheadrightarrow\s{O}(\bbf{N}_{\mu})^{\leq 1}$,  where the first is the dual of the trace map $\f{n}_{\mu}\otimes_L\f{n}^{\vee}_{\mu} \to 1$ and the second is given by  derivations of $\overline{\f{p}}_{\mu}$ on polynomials of degree $1$; see Proposition \ref{PropEquivFLag}.   Taking pushout and pullback diagrams  of $\widetilde{\eta}$,  one obtains a class 
       \[
       \eta'\in \Ext^1_{\widehat{\s{O}}_{\Shan}}\left(\widehat{\s{O}}_{\Shan}(-1)\otimes_{\s{O}_{\Shan}} \Omega^1_{\Shan}(\log),  \pi_{\HT}^{\tor,*}(\n{W}(\s{O}(\bbf{N}_{\mu})^{\leq 1}))  \right)
       \]   
       that has the following description:
       
\begin{lem}       
     The extension $\eta'$ has the form
 \begin{equation}
 \label{eqExtensionFinal}
 \begin{gathered}
 0\to\pi_{\HT}^{\tor, *}\left( \n{W}(\s{O}(\bbf{N}_{\mu})^{\leq 1}) \right)  \xrightarrow{(\alpha,\beta)} \gr^{1} \OBdR{\Shan}^+(-1) \oplus \pi_{\HT}^{\tor,*}(\f{n}_{\mu}^{0,\vee}) \\  \xrightarrow{(\overline{\nabla},  \KS)} \Omega^1_{\Shan}(\log)\otimes_{\s{O}_{\Shan}} \widehat{\s{O}}_{\Shan}(-1)\to 0,
 \end{gathered}
 \end{equation}
where  $\KS$ is the $\widehat{\s{O}}_{\Shan}(-1)$-extension of scalars of the  Kodaira-Spencer isomorphism \eqref{eqKS} composed with the isomorphism $\f{n}^{0,\vee}_{\mu}\cong \f{g}^{\mathrm{der},0}/\overline{\f{p}}_{\mu}^0$.  
\end{lem}

\begin{proof}
The extension class $\eta'$ has the form 
\[
0\to \pi_{\HT}^{\tor, *}\left( \n{W}(\s{O}(\bbf{N}_{\mu})^{\leq 1}) \right)\to \s{E}\to \Omega^1_{\Shan}(\log)\otimes_{\s{O}_{\Shan}} \widehat{\s{O}}_{\Shan}(-1)\to 0.
\]
By counting ranks of $\widehat{\s{O}}_{\Shan}$-vector bundles, it suffices to construct a commutative diagram 
\begin{equation}\label{eqKeyCommutativeSquareKS}
\begin{tikzcd}
\gr^{1} \OBdR{\Shan}^+(-1) \oplus \pi_{\HT}^{\tor,*}(\f{n}_{\mu}^{0,\vee}) \ar[r,"(\overline{\nabla}{,}\KS)"] \ar[d,"\gamma"]  & \Omega^1_{\Shan}(\log)\otimes_{\s{O}_{\Shan}} \widehat{\s{O}}_{\Shan}(-1)  \ar[d, "\id"]\\
\s{E} \ar[r]  &  \Omega^1_{\Shan}(\log)\otimes_{\s{O}_{\Shan}} \widehat{\s{O}}_{\Shan}(-1).
\end{tikzcd}
\end{equation}
with vertical isomorphisms.  First, we define the map $\gamma$. For this, consider the extension class $\widetilde{\eta}$ of \eqref{eqClasstensor}. By construction, its middle term is the direct sum 
\begin{equation}\label{eqDirectSumetatilde}
\pi^{\tor,*}_{\HT} (\f{n}^{0,\vee}_{\mu} \otimes_{\s{O}_{\Fl}}\f{n}^{0}_{\mu}) \otimes_{\widehat{\s{O}}_{\Shan}} \gr^1 \OBdR{\Shan}^+(-1) \oplus \pi_{\HT}^{\tor,*}(\f{n}^{0,\vee}_{\mu}\otimes_{\s{O}_{\Fl}} \overline{\f{m}}^{0}_{\mu}). 
\end{equation}
Let  $e:\s{O}_{\Fl}\to \f{n}^{0,\vee}_{\mu} \otimes_{\s{O}_{\Fl}}\f{n}^{0}_{\mu}$ be the natural map defined by the trace, it gives rise to a morphism 
\[
e\otimes \id: \gr^1 \OBdR{\Shan}^+(-1)\to  \pi_{\HT}^{\tor,*}(\f{n}^{0,\vee}_{\mu} \otimes_{\s{O}_{\Fl}}\f{n}^{0}_{\mu}) \otimes_{\widehat{\s{O}}_{\Shan}} \gr^1 \OBdR{\Shan}^+(-1).  
\]
Similarly, the map $\f{n}^{0,\vee}_{\mu}\otimes_{\s{O}_{\Fl}} \overline{\f{m}}^{0}_{\mu}\to \f{n}^{0,\vee}_{\mu}$, induced by the adjoint action of $\overline{\f{m}}_{\mu}$ on $\f{n}^{\vee}_{\mu}$,  is adjoint to a map
\[
\f{n}^{0,\vee}_{\mu}\to \f{n}^{0,\vee}_{\mu}\otimes_{\s{O}_{\Fl}} \overline{\f{m}}^{0,\vee}_{\mu}.
\]
Using the Killing form  of $\f{g}^{\mathrm{der}}$ we have a natural isomorphism $\overline{\f{m}}^{0,\vee}_{\mu}\cong \overline{\f{m}}^{0}_{\mu}$ obtaining in this way a map 
\[
\f{n}^{0,\vee}_{\mu}\to \f{n}^{0,\vee}_{\mu}\otimes_{\s{O}_{\Fl}} \overline{\f{m}}^{0}_{\mu}.
\]
Taking pullbacks along $\pi_{\HT}^{\tor}$  and direct sums, we have produced a map $\widetilde{\gamma}$ from $\gr^{1} \OBdR{\Shan}^+(-1) \oplus \pi_{\HT}^{\tor,*}(\f{n}_{\mu}^{0,\vee}) $ to \eqref{eqDirectSumetatilde}.  

On the other hand, by construction the class $\eta'$ factors through the pullback of the class $\widetilde{\eta}$ along the induced map attached to the trace $L\to \f{n}_{\mu}\otimes_L \f{n}^{\vee}_{\mu}$. Therefore, using the description of the Kodaira-Spencer map of Proposition \ref{PropKodairaSpencer} as a pullback from $\pi_{\HT}^{\tor}$, one verifies that $\widetilde{\gamma}$ factors through the subquotient $\s{E}$ of \eqref{eqDirectSumetatilde}. A bookkeeping  of the construction shows that the square \eqref{eqKeyCommutativeSquareKS} is commutative proving the lemma. 
\end{proof}

Therefore, since the Kodaira-Spencer map of \eqref{eqExtensionFinal}  is an isomorphism,   the map $\alpha$     is so, obtaining the second claim of the theorem. To obtain the isomorphism \eqref{eqIsoAlgebrasSenShimura} as algebras, note that by Remark \ref{RemarkDifferentConstructionOC} and Proposition \ref{PropEquivFLag}  we can write 
\[
\OCC{\Shan} =\Sym^{\bullet}_{\widehat{\s{O}}_{\Shan}} (\gr^1\OBdR{\Shan}(-1)) /(1-g(1))
\]
and 
\[
\s{O}(\bbf{N}_{\mu}) = \Sym_{L}^{\bullet}( \s{O}(\bbf{N}_{\mu})^{\leq 1}) / (1-\widetilde{g}(1)),
\]
where   $g: \widehat{\s{O}}_{\Shan}\to \gr^1 \OBdR{\Shan}^+ (-1)$ and $\widetilde{g}:L\to \s{O}(\bbf{N}_{\mu})^{\leq 1}$ are the natural inclusions. 
\end{proof}

\begin{cor}\label{CoroConnectionNablaPiHT}
Keep the notation of Theorem \ref{TheoSenOperator}. Under the equivalence \eqref{eqIsoAlgebrasSenShimura} the map $-\overline{\nabla}:\OCC{\Shan} \to \OCC{\Shan}(-1) \otimes_{\s{O}_{\Shan}} \Omega^1_{\Shan}(\log)  $ is identified with the pullback along $\pi_{\HT}^{\tor}$ of the $\bbf{G}^c$-equivariant map over $\FL$ attached to the $\bbf{P}_{\mu}^c$-equivariant map
 \[
 \nabla_{\bbf{N}_{\mu}^c}: \s{O}(\bbf{N}^c_{\mu}) \to  \s{O}(\bbf{N}^c_{\mu}) \otimes_{L} \f{n}^{c,\vee}_{\mu}
 \]
 given by the connection of $\bbf{N}_{\mu}^{c}$.  Here  we identify 
 \[
 \pi_{\HT}^{\tor}(\n{W}(\s{O}(\bbf{N}^c_{\mu}))\otimes_{\s{O}_{\Fl}}\f{n}^{c,0,\vee}_{\mu})\cong \OC(-1) \otimes_{\s{O}_{\Shan}} \Omega^1_{\Shan}(\log)
 \]
  via the Kodaira-Spencer isomorphism \eqref{eqKS}. 
\end{cor}
\begin{proof}
The map $-\overline{\nabla}$ is the unique $\widehat{\s{O}}_{\Shan}$-linear connection of the algebra $\OCC{\Shan}$ inducing the opposite of the Faltings extension $\gr^1\OBdR{\Shan} (-1)\xrightarrow{-\nabla} \widehat{\s{O}}_{\Shan}(-1)\otimes_{\s{O}_{\Shan}} \Omega^1_{\Shan}(\log)$ by taking the $\leq 1$-filtered part. On the other hand, $\n{W}(\s{O}(\bbf{N}^c_{\mu}))$ is an $\s{O}_{\Fl}$-algebra obtained from the algebra $\s{O}(\bbf{N}^c_{\mu})$ via Proposition \ref{PPropEqGSheavesFL}. The connection $\nabla_{\bbf{N}^c_{\mu}}$ is precisely the map  in Proposition \ref{PropEquivFLag} that we have used to identify $\s{O}(\bbf{N}^c_{\mu})^{\leq 1}$ as an extension 
\begin{equation}\label{eqExtensionAlgebraON}
0\to L\to \s{O}(\bbf{N}^c_{\mu})^{\leq 1} \to \f{n}^{c,\vee}_{\mu}\to 0. 
\end{equation}
But Theorem \ref{TheoSenOperator} has identified the pullback of \eqref{eqExtensionAlgebraON} along $\pi_{\HT}$ with the opposite of the Faltings extension. This shows that the pullback of $\nabla_{\bbf{N}^c_{\mu}}$ is precisely $-\overline{\nabla}$ as wanted. 
\end{proof}

\subsection{Geometric Sen operator}
\label{Subsec:GeometricSen}

In this section we use Theorem \ref{TheoSenOperator} to compute the geometric Sen operator of the tower of Shimura varieties $\pi_{K_p}:\Shan^{\tor}_{K^p,\infty,L}\to \Shan^{\tor}_{K^pK_p,L}$. First, we recall the main results in geometric Sen theory \cite[Theorems 3.3.2 and 3.3.4]{RCGeoSenTheory}.

\begin{theo}
\label{TheoMainSenTheory}
Let $X$ be a log smooth adic space over $L$ of dimension $d$ with normal crossing divisors. Let $\s{F}$ be an ON relative locally analytic $\widehat{\s{O}}_X$ sheaf on $X_{\proket}$, see \cite[Definition 3.2.1]{RCGeoSenTheory}. 

\begin{enumerate}

\item There is a natural map of pro-Kummer-\'etale  $\widehat{\s{O}}_X$-modules
\[
\theta_{\s{F}}:\s{F}\to \s{F}(-1)\otimes_{\s{O}_X}\Omega_X^1(\log)
\]
called the geometric Sen operator of $\s{F}$, satisfying the following properties
\begin{itemize}

\item $\theta_{\s{F}}$ is a Higgs field, namely, $\theta_{\s{F}}\wedge \theta_{\s{F}}=0$.

\item The formation of $\theta_{\s{F}}$ is functorial in $\s{F}$ and compatible with pullbacks on log smooth adic spaces. 

\item Let $C/L$ be the completion of an algebraic closure of $L$, and let  $\nu:X_{C,\proket}\to X_{C,\ket}$ be the projection of sites. Then there is a natural isomorphism in cohomology 
\[
R^i\nu_{*} \s{F} = \nu_* H^i(\theta_{\s{F}},\s{F})
\]
where $H^*(\theta_{\s{F}},\s{F})$ is the cohomology of the Higgs field
\[
0\to \s{F}\to \s{F}(-1)\otimes_{\s{O}_X} \Omega^1_{X}(\log)\to \cdots \to \s{F}(-d)\otimes_{\s{O}_X} \Omega^d_X(\log)\to 0. 
\]
\item  $\theta_{\s{F}}=0$ if and only if $\nu_*\s{F}$ is an ON Banach sheaf locally in the Kummer-\'etale topology of $X$ and $\s{F}= \nu_*\s{F} \widehat{\otimes}_{\s{O}_X} \widehat{\s{O}}_X$.

\end{itemize}

\item  Moreover, suppose that $X_{\infty}\to X$ is a pro-Kummer-\'etale $G$-torsor with $G$ a compact $p$-adic Lie group. Then there is a map of $\widehat{\s{O}}_X$-modules
\[
\theta_{X_{\infty}}: \widehat{\s{O}}_X\otimes_{\bb{Q}_p} (\Lie G)^{\vee}_{\ket} \to \widehat{\s{O}}_X(-1) \otimes_{\s{O}_X} \Omega^1_X(\log),
\]
where $(\Lie G)_{\ket}$ is the pro-Kummer-\'etale local system over $X$ obtained by descending the adjoint representation along $X_{\infty}\to X$, satisfying the following properties 
\begin{itemize}

\item $\theta_{X_\infty}$ is a Higgs field, namely $\theta_{X_{\infty}}\wedge \theta_{X_{\infty}}=0$.

\item Let $V$ a Banach locally analytic representation of $G$,  consider the pro-Kummer-\'etale sheaf $V_{\ket}$ over $X$ obtained by descent from the $G$-torsor $X_{\infty}\to X$ and the $G$-equivariant pro-Kummer-\'etale sheaf $\underline{V}$ on $X_{\infty}$ (see Definition \ref{DefinitionPeriodSheaves} (2)). Then we have a commutative diagram 
\[
\begin{tikzcd}
V_{\ket}\widehat{\otimes}_{\bb{Q}_p} \widehat{\s{O}}_X \ar[r,"d_V\otimes \id_{\widehat{\s{O}}_X}"] \ar[dr, "\theta_V"']& (V_{\ket}\otimes_{\bb{Q}_p} (\Lie G)^{\vee})\widehat{\otimes}_{\bb{Q}_p} \widehat{\s{O}}_X \ar[d,"\id_V\otimes \theta_{X_{\infty}}"]  \\ 
&  V_{\ket}\widehat{\otimes}_{\bb{Q}_p} \widehat{\s{O}}_X \otimes_{\s{O}_X} \Omega^1_X(\log),
\end{tikzcd}
\]
where $\theta_V$ is the geometric Sen operator of $V_{\ket}\widehat{\otimes}_{\bb{Q}_p} \widehat{\s{O}}_X$,  $d_V:V\to V\otimes_{\bb{Q}_p} (\Lie G)^{\vee}$ is the adjoint of the derivation map, and the tensor products are $p$-completed.  

\end{itemize}

We call $\theta_{X_{\infty}}$ the geometric Sen operator of the tower $X_{\infty}\to X$.

\end{enumerate}
\end{theo}

\begin{remark}\label{RemarkSenOperatorTower}

Theorem \ref{TheoMainSenTheory} (2) implies that there is a geometric Sen operator for the tower $\pi_{K_p}:\Shan^{\tor}_{K^p,\infty,L}\to \Shan^{\tor}_{K^pK_p,L}$
\begin{equation}
\label{eqSenOperatorShimura}
\theta_{\Shan}: \widehat{\s{O}}_{\Shan}\otimes_{\bb{Q}_p} \widetilde{\f{g}}^{\vee}_{\ket} \to \widehat{\s{O}}_{\Shan}(-1)\otimes_{\s{O}_{\Shan}} \Omega^1_{\Shan}(\log),
\end{equation}
with $\widetilde{\f{g}}=\Lie \widetilde{K}_p$. Moreover, the theorem tells us that $\theta_{\Shan}$ acts by derivations on pro-Kummer-\'etale sheaves obtained from locally analytic representations of $\widetilde{K}_p$. 

\end{remark}

\subsubsection{Connected components of Shimura varieties}

Before we compute the geometric  Sen operator in terms of representation theory over the flag variety,  let us discuss a technicality about the kernel of the map $\widetilde{\f{g}}\to \f{g}^c$ and the geometric Sen operator.

\begin{lem}
\label{LemmaFactorizationSenOperator}
Let us write $\widetilde{\f{g}}=\f{g}^{\mathrm{der}}\oplus \widetilde{\f{z}}$  as a direct sum of the derived algebra and the center. Then the geometric Sen operator \eqref{eqSenOperatorShimura} factors through
\begin{equation}
\label{eqSenOperatorFactor}
\widehat{\s{O}}_{\Shan}\otimes_{\bb{Q}_p}\f{g}^{\mathrm{der},\vee}_{\ket} \to \widehat{\s{O}}_{\Shan}(-1)\otimes_{\s{O}_{\Shan}}\Omega^1_{\Shan}(\log). 
\end{equation}
\end{lem}
\begin{proof}
Let $C$ be the completion of an algebraic closure of $L$, and consider the $C$-linear base change of Shimura varieties 
\[
\Shan^{\tor}_{K^p,\infty,C}\to \Shan^{\tor}_{K^pK_p,C}. 
\]
 Fix a connected component $\Shan^{\tor,0}_{K^p,\infty,C}$ of $\Shan^{\tor}_{K^p,\infty,C}$, and let $\Shan^{\tor,0}_{K^pK_p,C}$ be its image in  $\Shan^{\tor}_{K^pK_p,C}$. Denote  $\pi_{K_p}^{\tor,0}:\Shan^{\tor,0}_{K^p,\infty,C}\to \Shan^{\tor,0}_{K^pK_p,C}$, and consider the restriction of the Hodge-Tate period map to the connected component  $\pi_{\HT}^{\tor,0}:\Shan^{\tor,0}_{K^p,\infty,C}\to \Fl_C$.  By \cite[Corollary 5.2.4]{DiaoLogarithmicHilbert2018} the Galois group of the pro-Kummer-\'etale cover  $\pi_{K_p}^{\tor,0}$ injects into $\bbf{G}^{c,\mathrm{der}}(\bb{Q}_p)$. This shows that, at any connected component of $\Shan^{\tor}_{K^p, \infty,C}$, the geometric Sen operator \eqref{eqSenOperatorShimura} factors through the $\widehat{\s{O}}_{\Shan}$-extension of scalars of  $\f{g}^{c,\mathrm{der},\vee}_{\ket}= \f{g}^{\mathrm{der},\vee}_{\ket}$. Since \eqref{eqSenOperatorShimura} is a map of $\widehat{\s{O}}_{\Shan}$-vector bundles, it must factor through  \eqref{eqSenOperatorFactor} over the whole Shimura variety.  Indeed, the factorization of the geometric Sen operator
 \[
 \widehat{\s{O}}_{\Shan}\otimes_{\bb{Q}_p}\widetilde{\f{g}}^{\vee}_{\ket}\to  \widehat{\s{O}}_{\Shan}\otimes_{\bb{Q}_p}\f{g}^{\mathrm{der},\vee}_{\ket}\to \widehat{\s{O}}_{\Shan}(-1)\otimes_{\s{O}_{\Shan}}\Omega^1_{\Shan}(\log)
 \]
 can be checked after pullback to the pro-Kummer-\'etale site of the  points of $\Shan^{\tor}_{K^p,\infty,C}$, in particular at connected components. 
\end{proof}

\subsubsection{Computation of the geometric Sen operator}
\label{SubsubComputationSen}

By Lemma \ref{LemmaFactorizationSenOperator} the geometric Sen operator factors through the derived Lie algebra of $\widetilde{\f{g}}$, thus it suffices to consider the restriction to $\f{g}^c$
\begin{equation}
\label{eqSenOpeRestrictiongc}
\theta_{\Shan}: \widehat{\s{O}}_{\Shan}\otimes_{\bb{Q}_p} \f{g}^{c,\vee}_{\ket}\to \widehat{\s{O}}_{\Shan}(-1)\otimes_{\s{O}_X} \Omega^1_{\Shan}(\log).
\end{equation}

Before stating the main theorem concerning the computation of the geometric Sen operator, let us recall the hypothesis (BUN) of  \cite[Theorem 3.4.5]{RCGeoSenTheory}.

\begin{condition}[BUN]\label{ConditionBUN}
Let $X$ be a log-smooth adic space over an algebraically closed field $C$ with reduced normal crossing divisors, let $G$ be a compact $p$-adic Lie group and let $\pi:\widetilde{X}\to X$ be a pro-Kummer-\'etale $G$-torsor. We say that $\pi$ satisfies the condition (BUN) if the geometric Sen operator  
\[
\widehat{\s{O}}_X\otimes_{\bb{Q}_p}(\Lie G)_{\ket}^{\vee}\to \widehat{\s{O}}_{X}(-1)\otimes_{\s{O}_X}\Omega^1_X(\log)
\]
is a surjection of $\widehat{\s{O}}_X$-modules.  
\end{condition}

\begin{theo}
\label{TheoComputationSenOperator}
The geometric Sen operator \eqref{eqSenOpeRestrictiongc} is isomorphic to the pullback along $\pi^{\tor}_{\HT}$ of the quotient map of $\bbf{G}^c$-equivariant vector bundles on $\Fl$
\begin{equation}\label{eqSenOpeOverFlag}
\f{g}^{c,0,\vee}\to \f{n}_{\mu}^{c,0,\vee}
\end{equation}
via the isomorphism $\f{n}^{c,\vee}_{\mu}\cong \f{g}^c/\f{p}^{c}_{\mu}$  arising from the Killing form of $\f{g}^{\mathrm{der}}$, and  the Kodaira-Spencer map $\KS:\pi_{\HT}^{\tor,*}(\f{g}^{c,0}/\f{p}^{c,0}_{\mu}) \cong \pi_{K_p}^*(\Omega^1_{\Shan}(\log))(-1)$ of Proposition \ref{PropKodairaSpencer}. In particular, the torsor $\Shan^{\tor}_{K^p,\infty,L}\to \Shan^{\tor}_{K^pK_p,L}$ satisfies Condition \ref{ConditionBUN}.
\end{theo}
\begin{proof}
To compute \eqref{eqSenOpeRestrictiongc}, it suffices to compute the geometric Sen operator of the local system attached to  a faithful finite dimensional representation  $V$ of $\bbf{G}^c$.   By construction of $\pi_{\HT}^{\tor}$, we have an isomorphism of $K_p$-equivariant sheaves over $\Shan^{\tor}_{K^p,\infty,L}$
\[
V_{\ket}\otimes_{\bb{Q}_p} \widehat{\s{O}}_{\Shan}= \pi_{\HT}^{\tor,*}(\n{W}(V)).
\]
Since the localization $\n{W}(V)$ of $V$ as a $\bbf{G}^c$-equivariant vector bundle on $\Fl$ only depends on its restriction to $\bbf{P}^c_{\mu}$, it is endowed with a natural increasing filtration $\Fil_{i} \n{W}(V)$ whose pullback via $\pi_{\HT}^{\tor}$ is nothing but the Hodge-Tate filtration of $V_{\ket}$. 

On the other hand, by Theorem \ref{TheoMainSenTheory} (1), any $\widehat{\s{O}}_{\Shan}$-finite free module has attached a natural geometric Sen operator.  Therefore, to describe the geometric Sen operator of $V_{\ket}$ it suffices to describe the geometric Sen operator of the $\widehat{\s{O}}_{\Shan}$-finite free sheaves 
\[
\Fil_{i} (\widehat{\s{O}}_{\Shan}\otimes_{\bb{Q}_p} V_{\ket})=\pi_{\HT}^{\tor,*}(\Fil_{i} \n{W}(V)). 
\]

 Hence, it suffices to describe the geometric Sen operator of $\pi_{\HT}^{\tor}(\n{W}(W))$ for all $\bbf{P}^c_{\mu}$-representation $W$. Since any algebraic representation injects into the regular representation, it also suffices to consider the case where $W=\s{O}(\bbf{P}^c_{\mu})$.

By  the discussion of Section \ref{Subsec:RegularRepP} we can write $\s{O}(\bbf{P}^c_{\mu})\cong \s{O}(\bbf{N}^c_{\mu})\otimes_{L} \s{O}(\bbf{M}^c_{\mu})$ as $\bbf{P}^c_{\mu}$-representations. The action of $\bbf{P}^c_{\mu}$ on $W=\s{O}(\bbf{M}^c_{\mu})$ factors through $\bbf{M}^c_{\mu}$, and  Theorem \ref{TheoHodgeTatePeriod} gives an isomorphism of $K_{p}$-equivariant sheaves
\[
\pi_{\HT}^{\tor,*}(\n{W}(W))\cong \bigoplus_{\rho} \pi_{K_p}^*(W_{\Hd}[\rho])\otimes_{\widehat{\s{O}}_{\Shan}}\widehat{\s{O}}_{\Shan}(-\mu(\rho)),
\]
where $\rho$ runs over all irreducible representations of $\bbf{M}^c_{\mu}$,  $\mu(\rho)\in \bb{Z}$ is the $\mu$-weight of $\rho$, and $W_{\Hd}$ is as in Definition \ref{DefinitionAutoVB}.   Theorem \ref{TheoMainSenTheory} (1) implies that $\pi^{\tor,*}_{\HT}(\n{W}(W))$ has trivial geometric Sen operator. Then, we only need to compute the geometric Sen operator of the sheaf associated to $\s{O}(\bbf{N}^c_{\mu})$. By Theorem \ref{TheoSenOperator} we have a natural isomorphism 
\[
\pi^{\tor,*}_{\HT}(\n{W}(\s{O}(\bbf{N}^c_{\mu})))\cong \OCC{\Shan}.
\] 
On the other hand, \cite[Proposition 3.5.2]{RCGeoSenTheory} says that the geometric Sen operator of $\OCC{\Shan}$ is given by $-\overline{\nabla}$, but by Corollary \ref{CoroConnectionNablaPiHT} this operator is attached to the natural connection  
\[
\nabla_{\bbf{N}^c_{\mu}}:\s{O}(\bbf{N}^c_{\mu})\to \s{O}(\bbf{N}^c_{\mu}) \otimes_L \f{n}^{c,\vee}_{\mu}
\]
after taking equivariant sheaves on $\Fl$ and pullbacks along $\pi_{\HT}^{\tor}$. This produces an isomorphism between the geometric Sen operator $\theta_{\Shan}$ and the pullback along $\pi_{\HT}$ of \eqref{eqSenOpeOverFlag}. The fact that the isomorphism $\pi_{\HT}^{\tor,*}(\f{n}^{c,0,\vee}_{\mu})\cong \pi_{K_p}^*(\Omega^1_{\Shan}(\log))(-1)$ in the theorem is the Kodaira-Spencer map $\KS$ follows from the computation of the geometric Sen operator of $\OCC{\Shan}$ being equal to $-\overline{\nabla}$, and the $-1$ factor appearing in the isomorphism of extensions in Theorem \ref{TheoSenOperator}. 
\end{proof}

\section{Locally analytic completed cohomology}
\label{Section:CompletedCohomology}

In this last section we prove that the rational completed cohomology of Shimura varieties vanishes above middle degree, proving a weaker version of the Calegari-Emerton conjectures \cite{CalegariEmerton} for arbitrary Shimura varieties.  We first recall the definition of completed cohomology of \cite{EmertonInterpolation}, we then relate it with the pro-Kummer-\'etale cohomology of infinite level Shimura varieties, and finally prove the vanishing result using the theory of locally analytic representations and the geometric Sen operator.

\subsection{Completed cohomology}
\label{Subsec:CompletedCohomology}

 We let $j_{K_p}:  \Shan_{K^pK_p,L}\hookrightarrow \Shan_{K^pK_p,L}^{\tor}$ denote the open immersion of Shimura varieties, and let $C= \bb{C}_p$ be a completion of an algebraic closure of $L$.  

\begin{definition}
\label{DefinitionCompletedCohomology}
Let $i\in \bb{Z}$,  the $i$-th completed  cohomology  group at level $K^p$ is the space
\[
\widetilde{H}^i(K^p, \bb{Z}_p):= \varprojlim_{s} \varinjlim_{K_p} H^i_{\et}(\Shan_{K^pK_p,C}, \bb{Z}/p^s)
\]
where $K_p$ runs over the compact open subgroups of $\bbf{G}(\bb{Q}_p)$. Similarly, the $i$-th completed cohomology group at level $K^p$ with compact support is the group 
\[
\widetilde{H}^i_{c}(K^p, \bb{Z}_p):= \varprojlim_{s} \varinjlim_{K_p} H^i_{\et,c}(\Shan_{K^pK_p,C},  \bb{Z}/p^s),
\]
where the transition maps by pullbacks are well defined since the morphisms of Shimura varieties are finite \'etale.
\end{definition}

\begin{remark}
\label{RemarkComparisonCompletedEmerton}
Definition \ref{DefinitionCompletedCohomology}  of completed cohomology is slightly different from Definition \ref{DefinitionCompletedCohoIntro}, which is  the one introduced by  Emerton in \cite{EmertonInterpolation}.  The isomorphism between these two definitions follows from the  comparison  of \'etale cohomology for  adic spaces  vs algebraic varieties over $\bb{C}_p$ \cite{HuberEtaleCohomology}, and the Artin comparison between \'etale cohomology of algebraic varieties  (after fixing an isomorphism $\bb{C}_p\simeq \bb{C}$)  and the Betti cohomology of its underlying complex analytic space \cite{ArtinEtale}.
\end{remark}

We shall need a different expression for completed cohomology in terms of the Kummer-\'etale cohomology of toroidal compactifications of the Shimura varieties. Let $j_{K_p,\ket}: \Shan_{K^pK_p,L,\et}\to \Shan^{\tor}_{K^pK_p,L,\ket}$ and $j_{K_p,\et}: \Shan_{K^pK_p,L,\et}\to \Shan^{\tor}_{K^pK_p,L,\et}$ be the natural morphism of sites. 

\begin{lem}\label{lem:CompletedCohoAsKummerEtale}
There are natural $\bbf{G}(\bb{Q}_p)\times\Gal_{L}$-equivariant   isomorphisms 
\[
\widetilde{H}^i(K^p,\bb{Z}_p)= \varprojlim_{s} \varinjlim_{K_p} H^i_{\ket}(\Shan^{\tor}_{K^pK_p}, \bb{Z}/p^s)
\]
and 
\[
\widetilde{H}^i_c(K^p,\bb{Z}_p)= \varprojlim_{s} \varinjlim_{K_p} H^i_{\ket}(\Shan^{\tor}_{K^pK_p}, j_{K_p,\ket, !}\bb{Z}/p^s),
\]
where $K_p$ runs over a suitable family of compact open subgroups converging to $1$ as in Section \ref{Subsec:SetUp}.
\end{lem}
\begin{proof}
It suffices to show that we have natural isomorphisms 
\[
H^{i}_{\et}(\Shan_{K^pK_p,C},\bb{Z}/p^s)\cong H^i_{\ket}(\Shan^{\tor}_{K^pK_p,C}, \bb{Z}/p^s)
\]
and
\[
H_{\et,c}^i(\Shan_{K^pK_p,C},\bb{Z}/p^s)\cong H^i_{\ket}(\Shan^{\tor}_{K^pK_p,C}, j_{K_p,\ket, !}\bb{Z}/p^s). 
\]
The first equality follows from the natural isomorphism 
\[
Rj_{K_p,\ket,*}\bb{Z}/p^s=\bb{Z}/p^s
\] of \cite[Theorem 4.6.1]{DiaoLogarithmic2019} after taking global sections. For the second equality, consider the  morphism of sites $\eta: \Shan^{\tor}_{K^pK_p,L,\ket}\to \Shan^{\tor}_{K^pK_p,L,\et}$, then we claim that 
\[
R\eta_{*} j_{K_p,\ket,!}\bb{Z}/p^s=j_{K_p,\et,!}\bb{Z}/p^s,
\]
the lemma then follows by taking global sections. Indeed, let $\widetilde{x}\in \Shan^{\tor}_{K^pK_p,L}$ be a log geometric point of the boundary. We have to show that  $(R\eta_{*} j_{K_p,\ket,!}\bb{Z}/p^s)|_{\widetilde{x}}=0$, but we have 
\[
(R\eta_{*} j_{K_p,\ket,!}\bb{Z}/p^s)|_{\widetilde{x}}\cong R\Gamma_{\ket}(\widetilde{x},j_{K_p,\ket,!}\bb{Z}/p^s )=0
\] 
since $(j_{K_p,\ket,!}\bb{Z}/p^s) |_{\widetilde{x}}=0$. 
\end{proof}

\begin{definition}\label{DefinitionExtensionbyZero}
Let $j:\Shan_{K^p,\infty,L}\to \Shan^{\tor}_{K^p,\infty,L}$ be the open immersion of infinite level Shimura varieties, viewed as objects in the pro-Kummer-\'etale site of $\Shan^{\tor}_{K^pK_p,L}$.  For $\Lambda$ a $p$-adically complete ring we let $j_{!} \Lambda:= R\varprojlim_{s} j_{!} (\Lambda/p^s)$ be the (derived) $p$-adic completion of the extension by zero of the \'etale constant sheaves $\Lambda/p^s$ over $\Shan_{K^p,\infty,L}$. 
\end{definition}

\begin{remark}
Let us briefly mention why $j_! \Lambda$ sits in degree $0$ and therefore it is isomorphic to $\varprojlim_s j_! \bb{Z}/p^s$. Let $j_{K_p}:\Shan_{K^pK_p,L,\et}\to \Shan^{\tor}_{K^pK_p,L,\ket} $ be the natural map of sites. Let $D_{K_p}$ be the boundary normal crossings divisor of $\Shan^{\tor}_{K^pK_p,L}$ endowed with the induced log structure  and let $\iota_{K_p}: D_{K_p,\ket}\to \Shan^{\tor}_{K^pK_p,C,\ket}$ be the  map of Kummer-\'etale sites. Then for all $n\geq 1$ we have a short exact sequence  of Kummer-\'etale sheaves on $\Shan^{\tor}_{K^pK_p,L}$
\[
0\to j_{K_p,!} \Lambda/p^n \to \Lambda /p^n\to \iota_{D_{K_p},*} \Lambda /p^n\to 0
\]
that we can see as pro-Kummer-\'etale sheaves by \cite[Proposition 5.1.7]{DiaoLogarithmic2019}. 
Let $D=\varprojlim_{K_p'\subset K_p} D_{K_p'}$ be the limit in the pro-Kummer-\'etale site of the log  adic space $D_{K_p}$ and let $\iota: D_{\proket}\to  \Shan^{\tor}_{K^p,\infty,L,\proket}$ be the natural map. Taking colimits when $K_p\to \infty$ we have a short exact sequence of pro-Kummer-\'etale sheaves
\[
0\to j_{!} \Lambda/p^n \to \Lambda/p^n\to \iota_{D,*} \Lambda /p^n\to 0.
\]
Finally, taking derived limits as $n\to \infty$ we have an exact triangle 
\[
j_! \Lambda\to \Lambda\to \iota_{D_*} \Lambda\xrightarrow{+},
\]
but the map $\Lambda\to \iota_{D_*} \Lambda$ is surjective (being a countable limit of surjections along countable surjective maps in a replete topos \cite[Proposition 3.1.10]{bhatt2014proetale}), so $j_! \Lambda$ is concentrated in degree zero and we  actually have a short exact sequence 
\[
0\to j_{!} \Lambda \to \Lambda \to \iota_{D,*} \Lambda \to 0.
\] 
\end{remark}

   We want to relate the completed cohomology groups with the cohomology of infinite level Shimura varieties, for this  we need  the following theorem of Emerton. 

\begin{theo}[Emerton]
\label{TheoCompletedcohomologyAdmissible}
Let $\Lambda$ denote $\bb{Z}_p$ or $\bb{Z}/p^s$.  The cohomologies
\begin{equation}
\label{eqCompletedCohomologiesTheoremEmerton}
R\Gamma_{\proket}(\Shan^{\tor}_{K^p,\infty,C},  \Lambda) \mbox{ and } R\Gamma_{\proket}(\Shan^{\tor}_{K^p,\infty,C},  j_! \Lambda) 
\end{equation}
are represented by  bounded complexes of admissible $\widetilde{K}_p$-representations with terms isomorphic to finitely many copies of $C(\widetilde{K}_p, \Lambda)$, the space of continuous functions from $\widetilde{K}_p$ to $\Lambda$ endowed with the left regular action.  In particular,  $H^i_{\proket}(\Shan_{K^p,C}^{\tor}, \Lambda)$ and $H^i_{\proket}(\Shan_{K^p,C}^{\tor}, j_! \Lambda)$ are admissible $\widetilde{K}_p$-representations over $\Lambda$.  
\end{theo}

\begin{proof}

Let us first argue for torsion coefficients and $\Lambda=\bb{Z}/p^s$. By \cite[Proposition 5.1.6]{DiaoLogarithmic2019} we have 
\[
R\Gamma_{\proket}(\Shan^{\tor}_{K^p,\infty,C}, \Lambda)\cong \varinjlim_{K_p'\subset K_p} R\Gamma_{\ket}(\Shan^{\tor}_{K^pK_p',C},\Lambda),
\]
where $K_p'$ runs over all the open normal subgroups of $K_p$ (a similar formula holds for $j_!\Lambda$).  Since $\Shan^{\tor}_{K^pK_p',C}\to \Shan^{\tor}_{K^pK_p,C}$ is a $\widetilde{K}_p/ \widetilde{K}_p'$-torsor, we get 
\begin{equation}
\label{eqShaphiroVarietyetale}
\begin{aligned}
R\Gamma_{\ket}(\Shan^{\tor}_{K^pK_p,C}, C(\widetilde{K}_p/\widetilde{K}_p',\Lambda)) & \cong R\Gamma(\widetilde{K}_p/\widetilde{K}_p', R\Gamma_{\ket}(\Shan^{\tor}_{K^pK_p',C}, C(\widetilde{K}_p/\widetilde{K}_p', \Lambda)))\\
		& \cong R\Gamma(\widetilde{K}_p/\widetilde{K}_p', C(\widetilde{K}_p/\widetilde{K}_p',\Lambda)\otimes^L_{\Lambda} R\Gamma_{\ket}(\Shan^{\tor}_{K^pK_p',C},\Lambda))\\
		& \cong R\Gamma_{\ket}(\Shan^{\tor}_{K^pK_p',C}, \Lambda),
\end{aligned}
\end{equation}
where  the first equivalence is the Hochschild-Serre spectral sequence arising from the $\widetilde{K}_p/\widetilde{K}_p'$-torsor $\Shan^{\tor}_{K^pK_p',L}\to \Shan^{\tor}_{K^pK_p,L}$, the second follows from the fact that $\Shan^{\tor}_{K^pK_p',C}$ is qcqs so that cohomology commutes with filtered colimits, and the third equivalence is Shapiro's lemma for finite groups. The group cohomology of \eqref{eqShaphiroVarietyetale} is the usual group cohomology of smooth representations. 

 We deduce a natural quasi-isomorphism 
\[
\begin{aligned}
R\Gamma_{\proket}(\Shan^{\tor}_{K^p,\infty,C},\Lambda) & \cong \varinjlim_{K_p'} R\Gamma_{\ket}(\Shan^{\tor}_{K^pK_p,C}, C(\widetilde{K}_p/\widetilde{K}_p',\Lambda)) \\ 
			& \cong R\Gamma_{\ket}(\Shan^{\tor}_{K^pK_p, C}, C(\widetilde{K}_p,\Lambda)).
\end{aligned}
\]  
A similar argument also shows that 
\[
R\Gamma_{\proket}(\Shan^{\tor}_{K^p,\infty,C},j_! \Lambda) \cong R\Gamma_{\ket}(\Shan^{\tor}_{K^pK_p, C}, j_{K_p,!}C(\widetilde{K}_p,\Lambda)).
\]

By fixing an isomorphism of fields $C\simeq \bb{C}$,  the discussion in Remark \ref{RemarkComparisonCompletedEmerton} and Lemma \ref{lem:CompletedCohoAsKummerEtale} allow us to compare Kummer-\'etale and Betti cohomology obtaining $\bbf{G}(\bb{Q}_p)$-equivariant quasi-isomorphisms
\[
R\Gamma_{\proket}(\Shan^{\tor}_{K^p,\infty,C},\Lambda) \cong R\Gamma_{\Betti}(\Sh_{K^pK_p,E}(\bb{C}), C(\widetilde{K}_p,\Lambda))
\]
and 
\[
R\Gamma_{\proket}(\Shan^{\tor}_{K^p,\infty,C},j_!\Lambda) \cong R\Gamma_{\Betti,c}(\Sh_{K^pK_p,E}(\bb{C}), C(\widetilde{K}_p,\Lambda)),
\]
where the $\bb{C}$-points are taken with respect to the embedding $E\to C\simeq \bb{C}$.

Let $\Sh_{K^pK_p,E}(\bb{C})^{\BS}$ be a Borel-Serre compactification of $\Sh_{K^pK_p,E}(\bb{C})$ (cf. \cite{BorelSerre1973}).  It is a compact CW complex which is homotopically equivalent to $\Sh_{K^pK_p,E}(\bb{C})$.   Let $S_{\bullet}$ be a finite simplicial resolution of $\Sh_{K^pK_p,E}(\bb{C})^{\BS}$,  then 
\begin{align*}
R\Gamma_{\Betti}(\Sh_{K^pK_p,E}(\bb{C}),  C(\widetilde{K}_p, \Lambda))  & \cong R\Gamma_{\Betti}(\Sh_{K^pK_p,E}(\bb{C})^{\BS},  C(\widetilde{K}_p, \Lambda))  \\ 
& \cong \Hom^{\bullet}_{\Lambda}( \Lambda[S_\bullet],  C(\widetilde{K}_p, \Lambda))
\end{align*}
is quasi-isomorphic to a bounded complex whose terms are finite direct sums of $C(\widetilde{K}_p, \Lambda)$,  in particular a bounded  complex of $\Lambda$-linear admissible  $\widetilde{K}_p$-representations.

Let $\Lambda[[\widetilde{K}_p]]= \Hom_{\Lambda}(C(\widetilde{K}_p,\Lambda),\Lambda)$ be the $\Lambda$-linear Iwasawa algebra of $\widetilde{K}_p$, equivalently, $\Lambda[[\widetilde{K}_p]]=\varprojlim_{K_p'\subset K_p}\Lambda[\widetilde{K}_p/\widetilde{K}_p']$ where $K_p'$ runs over the compact open subgroups of $K_p$. By Poincar\'e duality,  for a finite free $\Lambda$-module $\s{F}$, the derived $\Lambda$-dual of $R\Gamma_{\Betti,c}(\Sh_{K^pK_p,E}(\bb{C}), \s{F})$ is quasi-isomorphic to $R\Gamma_{\Betti}(\Sh_{K^pK_p,E}(\bb{C}), \s{F}^{\vee}[2d])$ where $d$ is the complex dimension of the Shimura variety. Applying this to $\s{F}= C(\widetilde{K}_p/K_p',\Lambda)$ one deduces that 
\[
\begin{aligned}
R\Gamma_{\Betti,c}(\Sh_{K^pK_p,E}(\bb{C}),  C(\widetilde{K}_p/\widetilde{K}_{p}',\Lambda)) & \cong R\Hom_{\Lambda}( R\Gamma_{\Betti}(\Sh_{K^pK_p,E}(\bb{C}), \Lambda[[\widetilde{K}_p/\widetilde{K}_p']][2d]),\Lambda) \\
& \cong R\Hom_{\Lambda}(  \Hom^{\bullet}_{\Lambda}(\Lambda[S_{\bullet}],  \Lambda[[\widetilde{K}_p/\widetilde{K}_p']] )  ,  \Lambda  )[-2d] \\
 & =  \Lambda[S_\bullet]\otimes^L_{\Lambda} C(\widetilde{K}_p/\widetilde{K}_p',  \Lambda)  [-2d],
\end{aligned}
\]
where in the last equality we use that  the chain complex $\Lambda[S_{\bullet}]$ is  a perfect complex so quasi-isomorphic to its double dual. Taking colimits as $K_p'\to 1$  we get the quasi-isomorphism
\[
R\Gamma_{\Betti,c}(\Sh_{K^pK_p,E}(\bb{C}),  C(\widetilde{K}_p,\Lambda))  \cong   \Lambda[S_\bullet]\otimes^L_{\Lambda} C(\widetilde{K}_p,  \Lambda)  [-2d]
\]  deducing that the completed cohomology with compact support is represented by a bounded complex of admissible $\widetilde{K}_p$-representations.  

In summary, we have quasi-isomorphisms
\begin{equation}
\label{eqNormalCompleted}
R\Gamma_{\proket}(\Shan^{\tor}_{K^p,\infty,C},  \Lambda) \cong \Hom^{\bullet}_{\Lambda}( \Lambda[S_\bullet],  C(\widetilde{K}_p, \Lambda))
\end{equation}
and 
\begin{equation}
\label{EqCompactCompleted}
R\Gamma_{\proket}(\Shan^{\tor}_{K^p,\infty,C},  j_! \Lambda)  \cong  \Lambda[S_\bullet]\otimes^L_{\Lambda} C(\widetilde{K}_p,  \Lambda) [-2d].
\end{equation}

Finally, to show the theorem for $\Lambda=\bb{Z}_p$, note that the complexes \eqref{eqCompletedCohomologiesTheoremEmerton} are derived $p$-complete being the cohomology complexes of derived $p$-complete sheaves, and that the quasi-isomorphisms \eqref{eqNormalCompleted} and \eqref{EqCompactCompleted} are compatible for $\Lambda$ a torsion ring, the statement follows by taking derived limits  of \eqref{eqNormalCompleted} and \eqref{EqCompactCompleted} with coefficients $\bb{Z}/p^s$. 
\end{proof}

\begin{cor}
\label{CoroCompletedCohoAsProkummerCoho}
We have natural $\widetilde{K}_p$-equivariant isomorphisms 
\begin{eqnarray*}
H^i_{\proket}(\Shan^{\tor}_{K^p,\infty,C}, \bb{Z}_p)& \cong \widetilde{H}^i(K^p, \bb{Z}_p)  \\
H^i_{\proket}(\Shan^{\tor}_{K^p,\infty,C}, j_! \bb{Z}_p)& \cong \widetilde{H}^i_{c}(K^p, \bb{Z}_p).
\end{eqnarray*}
\end{cor}
\begin{proof}

Since $\bb{Z}_p= R\varprojlim_{s} \bb{Z}/p^s$ and $j_!\bb{Z}_p= R\varprojlim_{s} j_! \bb{Z}/p^s$ as pro-Kummer-\'etale sheaves,  one has a short exact sequence at the level of cohomology 
\[
0 \to  R^1\varprojlim_s(H^{i-1}_{\proket}(\Shan^{\tor}_{K^p,\infty,C},  \bb{Z}/p^s)) \to H^i_{\proket}(\Shan_{K^p,\infty,C}^{\tor}, \bb{Z}_p) \to  \varprojlim_s H^{i}_{\proket}(\Shan^{\tor}_{K^p,\infty,C},  \bb{Z}/p^s) \to 0,
\]
(resp.  for cohomology  with compact support). On the other hand, \cite[Proposition 5.1.6]{DiaoLogarithmic2019} implies that 
\[
 H^{i}_{\proket}(\Shan^{\tor}_{K^p,\infty,C},  \bb{Z}/p^s) =\varinjlim_{K_p} H^i_{\ket}(\Shan^{\tor}_{K^pK_p,C}, \bb{Z}/p^s).
\]
  Therefore,  we only need to show that the $R^1\varprojlim_s$ term appearing above vanishes.  This is a consequence of Proposition  1.2.12 of \cite{EmertonInterpolation} knowing that the cohomologies \eqref{eqCompletedCohomologiesTheoremEmerton} are represented by bounded complexes of admissible $\widetilde{K}_p$-representations (Theorem \ref{TheoCompletedcohomologyAdmissible}).  
\end{proof}

\subsection{Locally analytic completed cohomology}
\label{Subsec:CalegariEmerton}

Let $C$ be the completion of an algebraic closure of $L$. In this section we study the locally analytic vectors of completed cohomology and relate them with the analytic cohomology of a sheaf $\s{O}^{la}_{\Shan}$ of locally analytic functions on the infinite level Shimura variety $\Shan^{\tor}_{K^p,\infty,C}$. As a corollary, we shall obtain a vanishing result for completed  cohomology, proving a rational version of a conjecture of Calegari and Emerton.

 \subsubsection{The sheaf $\s{O}^{la}_{\Shan}$}  We recall some definitions from   \cite[Section 3.4]{RCGeoSenTheory}.

\begin{definition}
\label{DefinitionOla}
Consider $\Shan^{\tor}_{K^p,\infty,C}$ the infinite level Shimura variety over $C$.

\begin{enumerate}

\item The analytic site of $\Shan^{\tor}_{K^p,\infty,C}$ is given by the site of open subspaces of the underlying topological space $|\Shan^{\tor}_{K^p,\infty, C}|=\varprojlim_{K_p} |\Shan^{\tor}_{K^pK_p,C}|$. An open subspace $U\subset |\Shan^{\tor}_{K^p,\infty,C}|$ is a rational subspace if it is the pullback of a rational subspace at finite level.

\item The sheaf $\s{O}^{la}_{\Shan}$ of locally analytic functions of $\Shan^{\tor}_{K^p,\infty,C}$  is the ind-Banach sheaf on the topological space $|\Shan^{\tor}_{K^p,\infty,C}|$ of locally analytic sections of $\widehat{\s{O}}_{\Shan}$. More precisely, it is the sheaf mapping a rational subspace $U\subset |\Shan^{\tor}_{K^p,\infty,C}|$ with stabilizer $K_{p,U}\subset K_p$ to the ind-Banach space 
\[
\s{O}^{la}_{\Shan}(U)= \widehat{\s{O}}_{\Shan}(U)^{K_{p,U}-la}
\]
of $K_{p,U}$-locally analytic vectors, see \cite[Definition 3.4.2 and Lemma 3.4.3]{RCGeoSenTheory}.

\end{enumerate}

\end{definition}

For completeness of the paper we show that  $\s{O}^{la}_{\Shan}$ is  a sheaf.

\begin{lem}\label{LemmaSheafOla}
The subpresheaf $\s{O}^{la}_{\Shan}\subset \widehat{\s{O}}_{\Shan}$ on the topological space  $|\Shan^{\tor}_{K^p,\infty,C}|$ is a sheaf. 
\end{lem}
\begin{proof}
Let $U\subset |\Shan^{\tor}_{K^p,\infty,C}|$ be a qcqs subspace and let $\{U_i\}_{i=1}^n$ be a finite open cover of $U$ by qcqs subspaces. Since $\widehat{\s{O}}_{\Shan}$ is a sheaf, we have a left exact sequence 
\[
0\to \widehat{\s{O}}_{\Shan}(U)\to \prod_{i=1}^n\widehat{\s{O}}_{\Shan}(U_i)\to \prod_{1 \leq i\leq j\leq n} \widehat{\s{O}}_{\Shan}(U_i\cap U_j)).
\]
 Let $K_p\subset \bbf{G}(\bb{Q}_p)$ be such that all the  finite intersections of the $U_i$ are $K_p$-stable.  Then, since taking locally analytic vectors is left exact (eg. by \cite[Definition 3.2.3 (3)]{RRLocAnII}), we have a left exact sequence after taking locally analytic vectors
 \[
0\to  \s{O}^{la}_{\Shan}(U)\to \prod_{i=1}^n\s{O}^{la}_{\Shan}(U_i)\to \prod_{1 \leq i\leq j\leq n} \s{O}^{la}_{\Shan}(U_i\cap U_j))
\]
proving what we wanted. 
\end{proof}

\subsubsection{Locally analytic vectors of the boundary}

To introduce the relevant sheaf that will compute the compactly supported  locally analytic completed cohomology, we need to discuss the locally analytic vectors of the boundary divisors of the toroidal compactifications of the Shimura varieties. 

Let $D_{K_p}$ be the boundary of $\Shan^{\tor}_{K^pK_p,C}$ and let us write  $D_{K_p}=\bigcup_{a\in I} D_{K_p,a}$  as a union of irreducible smooth divisors with smooth finite intersections (we can always arrange this thanks to \cite[Proposition 9.20]{Pink1990ArithmeticalCO}). For $J\subset I$ a finite set we denote $D_{K_p,J}=\bigcap_{a\in J} D_{K_p,a}$, and for a level $K_p'\subset K_p$ we let $D_{K_p',J}$ be the reduced pullback of $D_{K_p,J}$ to $\Shan^{\tor}_{K^pK_p',C}$.   Let $\widehat{\s{O}}_{J}$ (resp. $\widehat{\s{O}}_{J}^+$) be  the pro-Kummer-\'etale sheaf   of completed (bounded) functions of  $\iota_{K_p,J}: D_{K_p,J}\subset \Shan^{\tor}_{K^pK_p,C}$ when endowed with the induced log structure. Finally, we define $\widehat{\n{I}}_{\Shan}^{+}$ to be the kernel of the map of pro-Kummer-\'etale shaves
\[
\widehat{\n{I}}_{\Shan}^{+} \colon= \ker(\widehat{\s{O}}_{\Shan}^{+}\to \bigoplus_{a\in I} \iota_{K_p,a,*} \widehat{\s{O}}_{a}^{+})
\]
and set $\widehat{\n{I}}_{\Shan}=\widehat{\n{I}}_{\Shan}^{+}[\frac{1}{p}]$.

\begin{lem}\label{LemmaLongExactSequenceOBoundary}
We have a long exact sequence of  almost pro-Kummer-\'etale $\widehat{\s{O}}^+_{\Shan}$-modules
\begin{equation}\label{eqExactSequenceBoundaryO}
0\to \widehat{\n{I}}_{\Shan}^{+}\to \widehat{\s{O}}_{\Shan}^{+}\to \bigoplus_{a\in  I} \iota_{K_p,a,*} \widehat{\s{O}}_a^{+} \to \cdots \to  \bigoplus_{\substack{J\subset  I \\ 
|J|=k}} \iota_{K_p,J,*} \widehat{\s{O}}_a^{+}  \to \cdots \to \iota_{K_p,I,*}\widehat{\s{O}}_I^{+}\to 0.
\end{equation}
Moreover, let $j_{K_p}:\Shan_{K^pK_p,C}\to \Shan_{K^pK_p,C}^{\tor}$ be the open immersion. The sequence \eqref{eqExactSequenceBoundaryO} is the $p$-completed $\widehat{\s{O}}_{\Shan}^+$-base change of the long exact sequence
\begin{equation}\label{eqExactSequenceBoundary}
\begin{aligned}
0\to j_{K_p,!} \bb{Z}_p \to \bb{Z}_p \to \bigoplus_{a\in I} \iota_{K_p,a,*} \bb{Z}_p \to  \cdots  \to \\  \cdots \to \bigoplus_{\substack{J\subset  I \\ 
|J|=k}} \iota_{K_p,J,*} \bb{Z}_p   \to \cdots   \to  \iota_{K_p,I,*}\bb{Z}_p \to 0.
\end{aligned}
\end{equation}
\end{lem}
\begin{proof}

By \cite[Lemma 2.1.5]{LanLiuZhuRhamComparison2019} for all $s\in \bb{N}$ we have a long exact sequence  of Kummer-\'etale sheaves on $\Shan^{\tor}_{K^pK_p,C}$
\[
0\to j_{K_p,!} \bb{Z}/p^s \to \bb{Z}/p^s \to \bigoplus_{a\in I} \iota_{K_p,a,*} \bb{Z}/p^s \to \cdots \to \iota_{K_p,I,*}\bb{Z}/p^s \to 0.
\]
Taking derived limits we get a long exact sequence as in \eqref{eqExactSequenceBoundary}. On the other hand, by \cite[Lemma 4.5.7]{DiaoLogarithmic2019},  tensoring  \eqref{eqExactSequenceBoundary} with $\s{O}^+_{\Shan}/p^s$ produces a long exact sequence of almost Kummer-\'etale $\s{O}^+_{\Shan}/p^s$-modules
\[
\begin{gathered}
0\to j_{K_p,!} \bb{Z}_p\otimes_{\bb{Z}_p} \s{O}^+_{\Shan}/p^s\to  \s{O}_{\Shan}^+/p^s\to \bigoplus_{a\in  I} \iota_{K_p,a,*}  \s{O}_a^+/p^s \to \cdots  \\ \cdots \to \iota_{K_p,I,*} \s{O}_I^+/p^s\to 0.
\end{gathered}
\]
Taking inverse limits  we obtain  the long exact sequence \eqref{eqExactSequenceBoundaryO}.
\end{proof}

\begin{definition}\label{DefinitionOlaD}
\begin{enumerate}

\item  For each $J\subset I$ we let $D_{\infty,J}=\varprojlim_{K_p'\subset K_p} D_{K_p',J}$ be the limit in the pro-Kummer-\'etale site of $D_{K_{p},J}$. The analytic site of $D_{\infty, J}$ is defined as  in  Definition \ref{DefinitionOla} (1), similarly for rational subsapces. 

\item The sheaf $\s{O}^{la}_{J}$ of locally analytic functions of $D_{\infty,J}$ is the ind-Banach sheaf on the topological space $|D_{\infty,J}|$ of locally analytic sections of $\widehat{\s{O}}_{J}$, as in  Definition \ref{DefinitionOla} (2). More precisely, it is the the sheaf mapping a rational subspace $U\subset D_{\infty,J}$  with stabilizer $K_{p,U}$ to the locally analytic functions 
\[
\s{O}^{la}_J(U)= \widehat{\s{O}}_{J}(U)^{K_{p,U}-la}.
\]

\item Finally, we define the sheaf $\n{I}^{la}_{\Shan}$ to be the ind-Banach sheaf on $|\Shan^{\tor}_{K^p,\infty,C}|$ of locally analytic sections of $\widehat{\n{I}}_{\Shan}$ as  in  Definition \ref{DefinitionOla} (2). It is the the sheaf mapping a rational subspace $U\subset \Shan_{K^p,\infty,C}$  with stabilizer $K_{p,U}$ to the locally analytic functions 
\[
\n{I}^{la}_{\Shan}(U)= \widehat{\n{I}}_{\Shan}(U)^{K_{p,U}-la}.
\]

\end{enumerate}
\end{definition}

\begin{remark}
The objects $\s{O}_{J}^{la}$ and $\n{I}^{la}_{\Shan}$ are sheaves on $|D_{\infty,J}|$ and $|\Shan_{K^p,\infty,C}|$ respectively, for example, by the same proof of Lemma \ref{LemmaSheafOla}.
\end{remark}

\subsubsection{Main comparison theorem}

In the following we prove the main comparison theorem between locally analytic completed cohomology and the sheaf $\s{O}^{la}_{\Shan}$. We denote by $\pi_{K_p}:\Shan^{\tor}_{K^p,\infty,C}\to \Shan^{\tor}_{K^pK_p,C}$ the natural projection.

\begin{theo}
\label{TheoMainVanishingOla}
There are natural $K_p$-equivariant quasi-isomorphisms of derived solid $C$-vector spaces
\begin{equation}
\label{eqIso1LocAn}
(R\Gamma_{\proket}(\Shan^{\tor}_{K^p,\infty,C}, \bb{Z}_p)\widehat{\otimes}^L_{\bb{Z}_p} C)^{Rla}=R\Gamma_{\an}(\Shan^{\tor}_{K^p,\infty,C}, \s{O}^{la}_{\Shan})
\end{equation}
and
\begin{equation}
\label{eqIso2LocAn}
(R\Gamma_{\proket}(\Shan^{\tor}_{K^p,\infty,C},j_! \bb{Z}_p)\widehat{\otimes}^L_{\bb{Z}_p} C)^{Rla}=R\Gamma_{\an}(\Shan^{\tor}_{K^p,\infty,C},\n{I}^{la}_{\Shan}),
\end{equation}
where the  derived locally analytic vectors are taken with respect to the group $\widetilde{K}_p$ (see Section \ref{Subsec:LocallyAnalyticRepresentations}), and the  $\widehat{\otimes}$-tensor products are $p$-completed. Furthermore, we have natural isomorphisms of cohomology groups
\[
(\widetilde{H}^i(K^p,\bb{Z}_p)\widehat{\otimes}_{\bb{Z}_p} C)^{la}= H^i_{\an}(\Shan^{\tor}_{K^p,\infty,C}, \s{O}^{la}_{\Shan})
\]
and 
\[
(\widetilde{H}^i_{c}(K^p,\bb{Z}_p))\widehat{\otimes}_{\bb{Z}_p} C)^{la}= H^i_{\an}(\Shan^{\tor}_{K^p,\infty,C}, \n{I}^{la}_{\Shan}).
\]
\end{theo}

\begin{remark}
Before giving the proof of the theorem let us explain how we see the objects involved in the equivalences \eqref{eqIso1LocAn} and \eqref{eqIso2LocAn} as solid $C$-linear $K_p$-equivariant representations, cf. \cite[Section 4.2]{RRLocallyAnalytic}. Let us just explain the first case, the second being analogous.  The completed cohomology $R\Gamma_{\proket}(\Shan^{\tor}_{K^p,\infty,C}, \bb{Z}_p)$ is a derived $p$-adically complete object with discrete mod $p$ fiber, then it naturally defines an object in the derived category of $K_p$-equivariant solid $\bb{Z}_p$-modules. The $p$-adically complete tensor product with $C$ is the same as the solid tensor product over $\bb{Z}_p$ (see \cite[Lemma 2.12.9]{mann2022padic}), and so it gives rise to a  $K_p$-equivariant solid $C$-vector space. Finally, the functor of locally analytic vectors is an endofunctor of the  derived category of $K_p$-equivariant solid $C$-vector spaces, see \cite[Definition 4.40]{RRLocallyAnalytic} and \cite[Definition 3.2.3]{RRLocAnII}. On the other hand, the right hand side term $R\Gamma_{\an}(\Shan^{\tor}_{K^p,C,\infty}, \s{O}^{la}_{\Shan})$ can be computed as a colimit of   \v{C}ech complexes of a rational hypercovers.  The values at rational subspaces of $\s{O}^{la}_{\Shan}$ are ind-Banach $C$-vector spaces, which are naturally solid $C$-vector spaces. In  Proposition \ref{LemmaLongExactSequenceOlaSheaves} we will even show that $\s{O}^{la}_{\Shan}$ is acyclic in a suitable  basis of rational open subspaces of $\Shan^{\tor}_{K^p,\infty,C}$.
\end{remark}

As a first key ingredient to show Theorem \ref{TheoMainVanishingOla}, we need a pro-Kummer-\'etale cohomological computation. In the following we let $\bb{T}_C=\Spa(C\langle T^{\pm 1}\rangle, \n{O}_C\langle T^{\pm 1}\rangle)$  denote the affinoid torus over $C$ with trivial log structure, and let  $\bb{D}_{C}=\Spa(C\langle S \rangle,\n{O}_C\langle S \rangle)$ denote the affinoid closed unit disc endowed with the log structure given by the divisor $S=0$. Let $V\subset \Fl$ be an open affinoid such that the sheaves $\f{n}^0_{\mu}$ and $\f{g}^{c,0}/\f{n}^{0}_{\mu}$ are finite free over $V$. Let  $K_p\subset \bbf{G}(\bb{Q}_p)$ be a compact open subgroup  stabilizing $V$ and $U\subset \Shan^{\tor}_{K^pK_p,C}$ open affinoid  such that $\pi_{K_p}^{-1}(U)\subset \pi^{\tor}_{\HT}(V)$.  Finally, we assume that $U$ admits a toric chart, namely, that there is a map $\psi: U\to \bb{T}^e_{C}\times \bb{D}^{d-e}_C $, for some $0\leq e\leq d$, that factors as a composite of finite \'etale maps and rational localizations, and such that $U$ has the log structure obtained from  $ \bb{T}^e_{C}\times \bb{D}^{d-e}_C$ by pullback.

\begin{prop}\label{LemmaLongExactSequenceOlaSheaves} Let $U\subset \Shan^{\tor}_{K^pK_p,C}$ be as above, and let $C^{la}(\widetilde{K}_p,\bb{Q}_p)_{\ket}$ be the pro-Kummer-\'etale sheaf over $\Shan^{\tor}_{K^pK_p,L}$ associated to the left regular representation of the locally analytic functions seen as a colimit of $p$-complete sheaves.  The following hold
 
 \begin{enumerate} 
 
 \item  For all $J\subset I$ there are natural quasi-isomorphisms 

\[
R\Gamma_{\proket}(U, \iota_{K_p,J.*} \widehat{\s{O}}_J \widehat{\otimes}_{\bb{Q}_p} C^{la}(\widetilde{K}_p, \bb{Q}_p)_{\ket}) \cong  \s{O}_J^{la}(\pi_{K_p}^{-1}(U)\cap D_{\infty,J}).
\]

\item We have a long exact sequence
\begin{equation}\label{eqLongExactSequneceIla}
0\to \n{I}^{la}_{\Shan}(\pi_{K_p}^{-1}(U))\to \s{O}^{la}_{\Shan}(\pi_{K_p}^{-1}(U))\to \cdots \to \s{O}^{la}_{I}(\pi_{K_p}^{-1}(U)\cap D_{\infty,I})\to 0.
\end{equation}
In particular, we also have that 
\begin{equation}\label{equationVanishingIlaProp}
R\Gamma_{\proket}(U, \widehat{\n{I}}_{\Shan} \widehat{\otimes}_{\bb{Q}_p} C^{la}(\widetilde{K}_p, \bb{Q}_p)_{\ket}) \cong  \n{I}_{\Shan}^{la}(\pi_{K_p}^{-1}(U)).
\end{equation}

\end{enumerate}
\end{prop}
\begin{proof}

The vanishing of the higher cohomology groups in part (1) is \cite[Proposition 3.2.7]{RCGeoSenTheory} for $J=\emptyset$ and \cite[Theorem 3.4.5]{RCGeoSenTheory} for general $J$. The exactness of the long exact sequence of part (2) is   \cite[Theorem 3.4.5]{RCGeoSenTheory} (more precisely, a direct consequence of its proof). For completeness of the paper,  we add the argument of part (1) which is the most interesting between the two. Indeed, part (2) is a by-product of the proof of part (1) after a more careful bookkeeping of the decompletions provided by  geometric Sen theory. 

\begin{itemize}

\item Let $T_1,\ldots, T_e$ be the coordinates of $\bb{T}^e_C$, similarly we let $S_{e+1},\ldots, S_d$ be the coordinates of $\bb{D}^{d-e}_{C}$.

\item  For $n\in \bb{N}$ we let $\bb{T}^{e}_{n,C}$ be the $e$-dimensional torus of coordinates $T_{1}^{1/n}, \ldots, T^{1/n}_{e}$, similarly we let $\bb{D}^{d-e}_{n,C}$ be the $(d-e)$-dimensional polydisc of coordinates $S_{e+1}^{1/n}, \ldots, S^{1/n}_{d}$.  

\item  We let $\bb{T}^{e}_{\infty,C}= \varprojlim_{n} \bb{T}^{e}_{n,C}$ and $\bb{D}^{d-e}_{\infty,C}=\varprojlim_{n} \bb{D}_{n,C}^{d-e}$ be the perfectoid torus and polydisc respectively. 

\item We let $\Gamma$ be the pro-Kummer-\'etale Galois group of $\bb{T}^e_{\infty,C}\times \bb{D}^{d-e}_{\infty,C}\to \bb{T}^e_{C}\times \bb{D}^{d-e}_C$. After fixing a sequence of  power roots of unit we have an isomorphism $\Gamma\cong \widehat{\bb{Z}}^d$ given by the action on coordinates. 

\item Finally, given $J\subset \{e+1,\ldots, d\}$ a finite subset we write $\Gamma=\Gamma_{J^c}\times \Gamma_{J}$, where $\Gamma_{J}$ is the Galois group associated to the coordinates $\{S_j\colon j\in J\}$.

\end{itemize}

Next we define the following pro-Kummer-\'etale objects over $U$

\begin{itemize}

\item  The pro-Kummer-\'etale $\widetilde{K}_p$-torsor $\widetilde{U}=\pi_{K_p}^{-1}(U)\to U$. For $K_p'\subset K_p$ we let $U_{K_p'}\to U$ be the quotient $U_{K_p'}=\widetilde{U}/K_p'$ in the pro-Kummer-\'etale site of $U$.

\item The finite Kummer-\'etale covers $U_n= U\times_{(\bb{T}^e_C \times \bb{D}^{d-e}_C)}(\bb{T}^e_{n,C} \times \bb{D}^{d-e}_{n,C})$.

\item The pro-Kummer-\'etale $\Gamma$-torsor $U_{\infty}=\varprojlim_n U_n$.

\item For $J\subset \{e+1,\ldots, d\}$ we write $\Gamma_{J}=\Gamma_J^p\times \Gamma_{J,p}$ as a product of its prime-to-$p$ part and its pro-$p$-Sylow subgroup. 

\item  The pro-Kummer-\'etale $\widetilde{K}_p\times \Gamma$-torsor $\widetilde{U}_{\infty}= \widetilde{U}\times_U U_{\infty} \to U$, resp. the $\widetilde{K}_p\times \Gamma_p$-torsor $\widetilde{U}_{p^\infty}= \widetilde{U}\times_U U_{p^\infty} \to U$. We also write $U_{K_p',p^n}= \widetilde{U}_{p^{\infty}}/ (\widetilde{K}_p'\times \Gamma_p^{p^n})$ as quotients in the pro-Kummer-\'etale site of $U$. 

\end{itemize}
Note that the objects $\widetilde{U}_{\infty}$ and $U_{\infty}$ are log affinoid perfectoid in the sense of  \cite[Definition 5.3.1]{DiaoLogarithmic2019}.

Finally, we introduce the last notation that takes care of the boundary.  For all $J\subset \{e+1,\ldots, d\}$ we let $U_{J}\to U$ be the Zariski closed subspace with vanishing locus $\{S_j=0 \colon j\in J\}$ endowed with the induced log structure. For $V\in U_{\proket}$ an object in the pro-Kummer-\'etale site of $U$, we let $V_{J}=U_J\times_{U} V$ be the pullback to an object in the pro-Kummer-\'etale site of $U_J$.

\textbf{Proof of part (1)}  Since $U$ has a toric chart $\psi$, we can assume that $I=\{e+1,\ldots, d\}$ indexes the variables of the polydisc. Let  $J\subset I$. We want to prove that the natural map 
\begin{equation}\label{eqComputationCohoCla}
\s{O}^{la}_J(\widetilde{U}_J)\to R\Gamma_{\proket}(U_J, \widehat{\s{O}}_J \widehat{\otimes}_{\bb{Q}_p} C^{la}(\widetilde{K}_p, \bb{Q}_p)_{\ket})
\end{equation}
is a quasi-isomorphism.  We proceed in different steps.

\textit{Step 1.} We rewrite the right hand side term of \eqref{eqComputationCohoCla} in terms of continuous group cohomology of an ind-Banach representation.  We claim that 
\begin{equation}\label{eqacyclicitylogPerfectoid}
R\Gamma_{\proket}(\widetilde{U}_{\infty, J},  \widehat{\s{O}}_J \widehat{\otimes}_{\bb{Q}_p} C^{la}(\widetilde{K}_p, \bb{Q}_p)_{\ket})= \widehat{\s{O}}_J(\widetilde{U}_{\infty, J})\widehat{\otimes}_{\bb{Q}_p} C^{la}(\widetilde{K}_p,\bb{Q}_p)
\end{equation}
as $\widetilde{K}_p\times \Gamma$-representation, where $\widetilde{K}_p$ acts diagonally  via the left regular action and $\Gamma$ only on the left term of the tensor.  Indeed, we can write $C^{la}(\widetilde{K}_p,\bb{Q}_p)=\varinjlim_{h} C(\bb{G}^{(h)}, \bb{Q}_p)$ as a colimit of analytic functions, where $\bb{G}^{(h)}\subset \bb{G}$ are rigid analytic groups whose intersection is $\widetilde{K}_p$, see \cite[Definition 2.1.4]{RRLocAnII}. Therefore, the sheaf $\s{G}_{J,K_p}= \widehat{\s{O}}_J \widehat{\otimes}_{\bb{Q}_p} C^{la}(\widetilde{K}_p, \bb{Q}_p)_{\ket}$  is a colimit of ON $\widehat{\s{O}}_J$-Banach sheaves $\s{G}_{J,K_p}=\varinjlim_{h} \s{G}_{J,K_p,h}$ whose restrictions to $\widetilde{U}_{\infty, J}$ are isomorphic to $\widehat{\bigoplus}_{\bb{N}} \widehat{\s{O}}_J$, and therefore relative locally analytic as in \cite[Definition 3.2.1]{RCGeoSenTheory}. 

By \cite[Lemma 5.3.8]{DiaoLogarithmic2019} there is an equivalence of topoi $\widetilde{U}_{\infty,J,\proet}^{\sim}\cong  \widetilde{U}_{\infty,J,\proket}^{\sim}$. Therefore, the almost acyclicity of $\s{O}^+_{J}/p$ on $\widetilde{U}_{\infty,J,\proet}^{\sim}$  \cite[Lemma 4.12]{ScholzeHodgeTheory2013} yields  the quasi-isomorphism \eqref{eqacyclicitylogPerfectoid}. 

Since $\widetilde{U}_{\infty,J}\to U_J$ is a $\widetilde{K}_p\times \Gamma$-torsor, there is a quasi-isomorphism 
\begin{equation}\label{eqwpjwepfoqw}
 R\Gamma_{\proket}(U_J, \s{G}_{J,K_p})\cong R\Gamma(\widetilde{K}_p\times \Gamma, \widehat{\s{O}}_J(\widetilde{U}_{\infty, J})\widehat{\otimes}_{\bb{Q}_p} C^{la}(\widetilde{K}_p,\bb{Q}_p))
\end{equation}
where the right hand side term is continuous group cohomology. It is clear that the $H^0$-cohomology group of the right hand side is $\s{O}^{la}_J(\widetilde{U}_J)$ (by first taking $\Gamma$-invariants and then $\widetilde{K}_p$-invariants). Therefore, we  only need to show the vanishing of the higher cohomology groups of \eqref{eqwpjwepfoqw}.

The following remark will be used later in the proof of Theorem \ref{TheoArithSenOla}.

\begin{remark}\label{RemaSimplification}
One can simplify the cohomology $ R\Gamma(\widetilde{K}_p\times \Gamma, \s{G}_{J,K_p}(\widetilde{U}_{\infty,J}))$ as  follows: since $U_{\infty,J}$ is log affinoid perfectoid, the almost acyclicity of $\s{O}^+/p$ yields 
\[
R\Gamma(\widetilde{K}_p\times \Gamma, \s{G}_{J,K_p}(\widetilde{U}_{\infty,J}))= R\Gamma(\Gamma, \s{G}_{J,K_p}(U_{\infty,J})).
\]
On the other hand, writing $\Gamma=\Gamma^p\times \Gamma_p$, with $\Gamma^p\cong \widehat{\bb{Z}}^{(p),d}$  having no pro-$p$-Sylow subgroups, the existence of a $p$-adic Haar measure of $\Gamma^p$   implies that $\Gamma^p$-invariants is exact in $p$-adic representations. A  Hochschild-Serre spectral sequence yields 
\[
R\Gamma(\Gamma, \s{G}_{J,K_p}(U_{\infty,J}))= R\Gamma(\Gamma_{p},  \s{G}_{J,K_p}(U_{p^{\infty},J})). 
\]
Similarly, one can also write 
\[
R\Gamma(\widetilde{K}_p\times \Gamma, \s{G}_{J,K_p}(\widetilde{U}_{\infty,J}))=R\Gamma(\widetilde{K}_p\times \Gamma_p, \s{G}_{J,K_p}(\widetilde{U}_{p^{\infty},J})). 
\]
\end{remark}

\textit{Step 2.} Now we use geometric Sen theory.  By \cite[Proposition 3.4.1]{RCGeoSenTheory} (and more precisely, its proof), we have that 
\begin{equation}\label{eqpwepgnqwofq}
H^i_{\proket}(U_J, \s{G}_{J,K_p})=H^0_{\proket}( H^i(\theta_{\s{G}}, \s{G}_{J,K_p} ) )
\end{equation}
where $\theta_{\s{G}}$ is the  geometric Sen operator of $\s{G}_{K_p}= \widehat{\s{O}}_{\Shan} \widehat{\otimes}_{\bb{Q}_p} C^{la}(\widetilde{K}_p, \bb{Q}_p)_{\ket}$.   Thus, it suffices to show that $ H^i(\theta_{\s{G}}, \s{G}_{J,K_p} ) =0$ for $i\geq 1$, and it would be enough to prove the vanishing after evaluating at $\widetilde{U}_{\infty,J}$. This  follows essentially  from  the Poincar\'e-Birkhoff-Witt theorem, after finding a complementary basis to the geometric Sen operator.  To prove it, it is more convenient to rewrite  the cohomology group $H^i_{\proket}(U_J, \s{G}_{J,K_p})$ as a colimit of group cohomologies as $K_p\to 1$. Namely, by \eqref{eqwpjwepfoqw} and \eqref{eqpwepgnqwofq} we have that 
\[
H^i_{\proket}(U_J, \s{G}_{J,K_p})= \varinjlim_{K_p'\subset K_p} H^i(\widetilde{K}_p'\times \Gamma,  \s{G}_{J,K_p'}(\widetilde{U}_{\infty,J}))= \varinjlim_{K_p'\subset K_p} H^i(\theta_{\s{G}},  \s{G}_{J,K_p'}(\widetilde{U}_{\infty,J}))^{\widetilde{K}_p'\times \Gamma}
\]
where in the first equality we use Shapiro's lemma since $\mathrm{Ind}_{K_p'}^{K_p}\big(\s{G}_{J,K_p'}(\widetilde{U}_{\infty})\big)=\s{G}_{J,K_p}(\widetilde{U}_{\infty,J})$ by construction.  So, denoting $\s{G}_{J,\f{g}}(\widetilde{U}_{\infty,J}):=\varinjlim_{K_p'\subset K_p} \s{G}_{J,K_p'}(\widetilde{U}_{\infty,J})$, we have to show that $H^i(\theta_{\s{G}}, \s{G}_{J,\f{g}}(\widetilde{U}_{\infty,J}))=0$ for $i\geq 1$. By the definition of $\s{G}_{J,K_p}$, we have the presentation 
 \[
 \s{G}_{J,\f{g}}(\widetilde{U}_{\infty,J})=\widehat{\s{O}}_J(\widetilde{U}_{J,\infty})\widehat{\otimes}_{\bb{Q}_p} C^{la}(\widetilde{\f{g}}, \bb{Q}_p),
 \]
 where $C^{la}(\widetilde{\f{g}}, \bb{Q}_p)=\varinjlim_{K_p'\subset K_p} C^{la}(\widetilde{K}_p', \bb{Q}_p)$ is the space of germs of locally analytic functors at $1$ of $\widetilde{K}_p$.  Summarizing, we need to prove that 
 \begin{equation}\label{eqpkwpemfpqwd}
 H^{i}(\theta_{\s{G}},\widehat{\s{O}}_J(\widetilde{U}_{\infty,J})\widehat{\otimes}_{\bb{Q}_p} C^{la}(\widetilde{\f{g}}, \bb{Q}_p) )= 0 \mbox{ for } i\geq 1.
 \end{equation}

\textit{Step 3.} In order to prove \eqref{eqpkwpemfpqwd}, we have to  consider  some rigid analytic varieties over $V\subset \Fl$.  The following constructions are in the same spirit as those of \cite[The\'or\`eme 6.1]{BC2}.

Let $g=\dim_{\bb{Q}_p} \widetilde{\f{g}}$ and  let $\f{H}_1,\ldots, \f{H}_g$ be a basis of $\widetilde{\f{g}}$ over $\bb{Q}_p$. Then, for $h\gg 0$ the exponential of the $\bb{Z}_p$-lattice $\n{K}_h$ generated by $\{p^h\f{H}_k\}_k$ defines an  affinoid group $\bb{G}_h$ isomorphic to a closed polydisc of dimension $g=\dim_{\bb{Q}_p} \widetilde{K}_p$, see \cite[Definition 2.1.4]{RRLocAnII}.  The groups $\{\bb{G}_h\}_{h}$ form a decreasing sequence of affinoid groups, for $h\gg 0$ we can also define open Stein groups $\mathring{\bb{G}}_{h}=\bigcup_{h'>h} \bb{G}_{h'}$ by taking the union of strictly smaller affinoid groups.  Let $C(\bb{G}_h,\bb{Q}_p)$ and $C(\mathring{\bb{G}}_h,\bb{Q}_p)$ be the spaces of functions of $\bb{G}_h$ and $\mathring{\bb{G}}_h$ respectively.  We have that 
\[
C^{la}(\widetilde{\f{g}}, \bb{Q}_p)=\varinjlim_{h\to \infty} C(\bb{G}_h,\bb{Q}_p)= \varinjlim_{h\to \infty} C(\mathring{\bb{G}}_h,\bb{Q}_p). 
\] 

Hence, for $h'>h\gg 0$ we have rigid analytic varieties over $V$:
\[
V\times \mathring{\bb{G}}_{h'}\subset V\times \bb{G}_{h'}\subset V\times \mathring{\bb{G}}_{h}\subset V\times \bb{G}_{h}. 
\]
Note that for a fixed $h>0$ there is some compact open subgroup $K_p'(h)\subset K_p$ (depending on $h$) acting on $V\times \bb{G}_{h}$ (resp. $V\times \mathring{\bb{G}}_{h}$) via the natural action on $V$, and the left multiplication on the group. Therefore, the spaces of functions 
\[
\s{O}(V\times \bb{G}_h)=\s{O}(V)\widehat{\otimes}_{\bb{Q}_p} C(\bb{G}_h, \bb{Q}_p)
\]
and 
\[
\s{O}(V\times \mathring{\bb{G}}_h)=\s{O}(V)\widehat{\otimes}_{\bb{Q}_p} C(\mathring{\bb{G}}_h, \bb{Q}_p)
\]
are naturally endowed with an action of $\widetilde{K}_p'(h)\subset \widetilde{K}'_{p}$ which is the natural action on $\s{O}(V)$ and the left regular action on the functions of the groups.

Now, let $\{\f{X}_1,\ldots, \f{X}_d\}$  and $\{\f{Y}_{d+1}, \cdots, \f{Y}_{g}\}$ be a basis of $\f{n}^0_{\mu} $  and a complement basis  in  $\widetilde{\f{g}}^0=\s{O}_{\Fl}\otimes \widetilde{\f{g}}$ over $V$ respectively. Since the Lie algebra $\f{n}_{\mu}\subset \f{g}$ is abelian, we can take the basis $\{\f{X}_s\}$ such that $[\f{X}_i,\f{X}_j]=0$ for all $1 \leq i,j\leq d$.

   For $r\gg 0$ the basis $\{p^r\f{X}_s, p^r\f{Y}_{k}\}_{s,k}$ of $\f{g}^0$ admits an exponential defining an affinoid rigid analytic  variety $\bb{X}_{r}$ over $V$ isomorphic to a polydisc of dimension $g$. Indeed, if $\n{L}_r$ is the $\s{O}^+(V)$-lattice generated by $\{p^r\f{X}_s, p^r\f{Y}_{k}\}_{s,k}$, for a given $h$ there is some $r$ big enough with $\n{L}_r\subset \s{O}^+(V)\otimes_{\bb{Z}_p} \n{K}_h$. Then, if $\bb{D}_{\bb{Q}_p}(\n{K}_h)$ is the closed polydisc over $\bb{Q}_p$ defined by the lattice $\n{K}_h$,  the exponential map 
\[
\exp: V\times \bb{D}_{\bb{Q}_p}(\n{K}_h) \xrightarrow{\sim } V\times \bb{G}_h
\]
identifies the polydisc   $\bb{D}_V(\n{L}_r)\subset  V\times \bb{D}_{\bb{Q}_p}(\n{K}_h)$ over $V$  with an  open subspace $\bb{X}_r\subset V\times \bb{G}_h$. Hence, we have a decreasing families of affinoid rigid spaces $\{\bb{X}_r\}_{r}$ which is final with respect to $\{V\times \bb{G}_h\}_h$.  We shall also consider Stein versions of the varieties $\mathring{\bb{X}}_{r}=\bigcup_{r'>r} \bb{X}_r$. Thus, for a fixed $r$ there is some compact open subgroup $\widetilde{K}_{p}(r)\subset \widetilde{K}_p$ acting on $\bb{X}_{r}$ and $\mathring{\bb{X}}_{r}$ by left multiplication. This endows $\s{O}(\bb{X}_{r})$ and $\s{O}(\mathring{\bb{X}}_{r})$ with a natural continuous locally analytic  action of $K_p(r)$.

 Finally, the basis $\{p^r\f{X}_s, p^r\f{Y}_k\}$ of  the lattice $\n{L}_r$ gives rise  decompositions as rigid spaces $\bb{X}_r=\bb{X}_r^1\times \bb{X}_r^2$ (resp. $\mathring{\bb{X}}_r= \mathring{\bb{X}}_{r}^1\times \mathring{\bb{X}}_r^2$) with each term isomorphic to an affinoid closed  polydisc (resp. an open polydisc) over $V$.

\textit{Step 4.} We now prove \eqref{eqpkwpemfpqwd}.  Set $A=\widehat{\s{O}}_J(\widetilde{U}_{\infty,J})$.    By Theorem \ref{TheoComputationSenOperator}, the action of the geometric Sen operator $\theta_{\s{G}}$ arise from the natural action by $A$-linear derivations  of $\f{n}^0_{\mu}$  via the map $\f{n}^0_{\mu} \to \widetilde{\f{g}}^0$. We want to compute the $\f{n}^0_{\mu}$-cohomology of $A\widehat{\otimes}_{\bb{Q}_p} C^{la}(\widetilde{\f{g}}, \bb{Q}_p)$, where $\widetilde{\f{g}}^0$ acts via  the $A$-linear extension of the derivation for the left regular action. By Step 3 we can write 
\begin{equation}\label{eqcolimitRigidX}
A\widehat{\otimes}_{\bb{Q}_p} C^{la}(\widetilde{\f{g}}, \bb{Q}_p)= \varinjlim_{r\to \infty} A\widehat{\otimes}_{\s{O}(V)}\s{O}(\bb{X}_r)=\varinjlim_{r\to \infty} A\widehat{\otimes}_{\s{O}(V)}\s{O}(\mathring{\bb{X}}_r) 
\end{equation}
where  the tensor product of the third term is a projective tensor product  of Fr\'echet spaces (isomorphic to  the global sections of a  closed polydisc relative to the perfectoid algebra $A$).

To finish the proof of part (1) it suffices to show the following lemma:

\begin{lem}\label{LemmaVanishingCohomologyStein}
Let $r\gg 0$, then there is a  quasi-isomorphism 
\begin{equation}\label{eqVanishingCohomologyOXring}
R\Gamma(\f{n}^0_{\mu}, \s{O}(\mathring{\bb{X}}_r))\cong \s{O}(\mathring{\bb{X}}_r^2). 
\end{equation}

\end{lem}
\begin{proof}
 Write $\mathring{\bb{X}}_r=\mathring{\bb{X}}_r^{1}\times \mathring{\bb{X}}^2_{r}$ as product of the Stein spaces obtained as the exponential of $\{p^r\f{X}_{s}\}_s$ and $\{p^r\f{Y}_k\}_{k}$  respectively. This allow us to write a point $x\in \bb{X}_r$ as 
\[
x(t_1,\ldots, t_g)= \exp(t_1 \f{X}_1)\cdots \exp(t_d \f{X}_d) \exp(t_{d+1} \f{Y}_{d+1})\cdots \exp(t_g\f{Y}_g)
\]
with $|t_i|<|p^r|$.

 By the choice of the basis $\{\f{X}_{s}\}_s$, we have that $[\f{X}_{i},\f{X}_j]=0$ for all $1\leq i,j\leq d$.  Therefore, the Lie algebra action of $\f{n}^0_{\mu}$ on $\s{O}(\mathring{\bb{X}}_r)\cong \s{O}(\mathring{\bb{X}}_r^1)\widehat{\otimes}_{\s{O}(V)}(\mathring{\bb{X}}_r^2)$ is nothing but the natural action by derivations with respect to the variables $\{t_1,\ldots, t_d\}$. Then, we obtain \eqref{eqVanishingCohomologyOXring} by the Poincar\'e lemma for open polydiscs \cite[Lemma 26]{Tamme}.  
\end{proof}

By Step (3) and \eqref{eqcolimitRigidX} we have
\[
H^i(\theta_{\Shan}, A\widehat{\otimes}_{\bb{Q}_p} C^{la}(\widetilde{\f{g}}, \bb{Q}_p))\cong \varinjlim_{r\to \infty} H^i(\f{n}^0_{\mu}, A\widehat{\otimes}_{\s{O}(V)} \s{O}(\mathring{\bb{X}}_r) ).
\]
Lemma \ref{LemmaVanishingCohomologyStein} shows that 
\[
R\Gamma(\f{n}^0_{\mu}, A\widehat{\otimes}_{\s{O}(V)} \s{O}(\mathring{\bb{X}}_r) )\cong A\widehat{\otimes}_{\s{O}(V)} \s{O}(\mathring{\bb{X}}_r^2),
\]
in particular that the cohomology groups of \eqref{eqpkwpemfpqwd}  vanish for $i>0$, proving what we wanted. 
\end{proof}

\begin{proof}[Proof of Theorem \ref{TheoMainVanishingOla}]
 In the following proof we consider pro-Kummer-\'etale cohomologies as objects in the derived category of solid abelian groups thanks to Lemma \ref{LemmaSolidEnhancement}. The $\widehat{\otimes}$-tensor products  between Banach or $p$-adically complete sheaves will be $p$-adically completed (these are the same as the solid tensor products by \cite[Lemma 3.13]{RRLocallyAnalytic} and \cite[Proposition 2.12.10]{mann2022padic}).  We  consider almost mathematics with respect to the maximal ideal of $\n{O}_C$.

\textit{Step 1.} Let us first rewrite the completed cohomologies of the left hand side terms of \eqref{eqIso1LocAn} and \eqref{eqIso2LocAn}. Consider the resolution 
\begin{equation}
\label{eqResolutionlowerShierk}
0\to j_{K_p,!}\bb{Z}_p \to \bb{Z}_p \to \prod_{a\in I} \iota_{K_p,a,*} \bb{Z}_p \to \cdots \to \prod_{\substack{J\subset I  \\  |J|=k} } \iota_{K_p,J,*} \bb{Z}_p \to \cdots \to \iota_{K_p,I,*} \bb{Z}_p\to 0,
\end{equation}
where $\iota_{J}: D_{K_p, J}\to \Shan^{\tor}_{K^pK_p,C}$ and $j_{K_p}:\Shan_{K^pK_p,C}\to \Shan^{\tor}_{K^pK_p,C}$.  By Lemma \ref{LemmaLongExactSequenceOBoundary} we have a long exact sequence
\begin{equation}
\label{eqResolutionOhatLowerShrieck}
0\to \widehat{\n{I}}^+_{\Shan} \to \widehat{\s{O}}^+_{\Shan} \to \cdots \to \prod_{\substack{J\subset I \\ |J|=k} } \iota_{K_p,J,*}\widehat{\s{O}}^+_{J} \to \cdots \to \iota_{K_p,I,*} \widehat{\s{O}}^+_{I}\to 0.
\end{equation}
  By the primitive comparison theorem \cite[Theorem 6.2.1]{DiaoLogarithmic2019} and \cite[Theorem 2.2.1]{LanLiuZhuRhamComparison2019} we have almost quasi-isomorphisms for $J\subset I$ finite and $s\in \bb{N}$
\[
R\Gamma_{\ket}(\Shan^{\tor}_{K^pK_p,C} ,\iota_{K_p,J,*}\bb{Z}/p^s)\otimes^L_{\bb{Z}/p^s} \n{O}_C/p^s \cong^{ae} R\Gamma_{\ket}(\Shan^{\tor}_{K^pK_p,C} ,\iota_{K_p,J,*}\widehat{\s{O}}^+_{J}/p^s).
\]
By \eqref{eqResolutionOhatLowerShrieck} we get an almost  quasi-isomorphism
\[
R\Gamma_{\ket}(\Shan^{\tor}_{K^pK_p,C} ,j_{K_p,!}\bb{Z}/p^s)\otimes^L_{\bb{Z}/p^s} \n{O}_C/p^s \cong^{ae} R\Gamma_{\ket}(\Shan^{\tor}_{K^pK_p,C} ,\widehat{\n{I}}^+_{\Shan}/p^s).
\]
Taking colimits as $K_p\to 1$, and derived limits as $s\to \infty$, we get  natural almost quasi-isomorphisms 
\begin{equation}
\label{eqPrimitive1}
R\Gamma_{\proket}(\Shan^{\tor}_{K^p,\infty,C} ,  \iota_{J,*}\bb{Z}_p) \widehat{\otimes}_{\bb{Z}_p}^L \n{O}_{C} \cong^{ae} R\Gamma_{\proket}(\Shan^{\tor}_{K^p,\infty,C},\iota_{J,*} \widehat{\s{O}}^+_{J})
\end{equation}
and 
\begin{equation}
\label{eqPrimitive2}
R\Gamma_{\proket}(\Shan^{\tor}_{K^p,\infty,C} , j_! \bb{Z}_p) \widehat{\otimes}_{\bb{Z}_p}^L \n{O}_{C} \cong^{ae} R\Gamma_{\proket}(\Shan^{\tor}_{K^p,\infty,C}, \widehat{\n{I}}^+_{\Shan}).
\end{equation}

\textit{Step 2.} We now rewrite the LHS of \eqref{eqIso1LocAn} and \eqref{eqIso2LocAn} in terms of pro-Kummer-\'etale cohomology.   We can assume without loss of generality that $\widetilde{K}_p$ is a uniform $p$-adic Lie group. We can then embed $\widetilde{K}_p \subset \bb{G}$ into a rigid analytic group, and write  the space of locally analytic functions as a colimit of functions in affinoid rigid analytic subgroups $\bb{G}^{(h)}\subset \bb{G}$ whose intersection is $\widetilde{K}_p$:
\[
C^{la}(\widetilde{K}_p,\bb{Q}_p)= \varinjlim_h C(\bb{G}^{(h)},\bb{Q}_p), 
\]  
 see \cite[Definition 2.1.4]{RRLocAnII}. By \eqref{eqPrimitive1} and Lemma \ref{LemmaTensorInsideCohomology}  below we have  natural equivalences
\begin{equation}
\label{eqClaIntoRGamma}
\begin{aligned}
(R\Gamma_{\proket}(\Shan^{\tor}_{K^p,\infty,C}, \iota_{J,*} \bb{Z}_p)\widehat{\otimes}^L_{\bb{Z}_p} C)^{Rla} & = 
R\Gamma(\widetilde{K}_p,R\Gamma_{\proket}(\Shan^{\tor}_{K^p,\infty,C}, \iota_{J,*} \widehat{\s{O}}_{J})\widehat{\otimes}^L_{\bb{Z}_p} C^{la}(\widetilde{K}_p,\bb{Q}_p)) \\ &  \cong R\Gamma_{\proket}(\Shan^{\tor}_{K^pK^p,C}, \iota_{J,*}\widehat{\s{O}}_{J}\widehat{\otimes}^L_{\bb{Q}_p} C^{la}(\widetilde{K}_p,\bb{Q}_p)_{\ket}).
\end{aligned}
\end{equation}
By  \eqref{eqPrimitive2} and the resolution \eqref{eqResolutionOhatLowerShrieck}  we also have an analogue of \eqref{eqClaIntoRGamma} for the cohomology with compact support:
\[
(R\Gamma_{\proket}(\Shan^{\tor}_{K^p,\infty,C},j_! \bb{Z}_p)\widehat{\otimes}^L_{\bb{Z}_p} C)^{Rla}\cong R\Gamma_{\proket}(\Shan^{\tor}_{K^p,\infty,C}, \widehat{\n{I}}_{\Shan}\widehat{\otimes}_{\bb{Q}_p} C^{la}(\widetilde{K}_p,\bb{Q}_p)_{\ket}). 
\]

\textit{Step 3.} Now we prove \eqref{eqIso1LocAn} and \eqref{eqIso2LocAn}.   Let $\eta_{K_p}: \Shan^{\tor}_{K^pK_p,C,\proket}\to \Shan^{\tor}_{K^pK_p,C,\an}$ be the projection of sites. By Proposition \ref{LemmaLongExactSequenceOlaSheaves}  we have a natural quasi-isomorphisms of  sheaves 
\begin{equation}\label{eqHigherVanishingIla}
\begin{aligned}
R\eta_{K_p,*} ( \widehat{\s{O}}_{\Shan} \widehat{\otimes}_{\bb{Q}_p} C^{la}(\widetilde{K}_p, \bb{Q}_p)_{\ket}) & \cong & \pi_{K_p,*}( \s{O}^{la}_{\Shan})\\
R\eta_{K_p,*} ( \widehat{\n{I}}_{\Shan} \widehat{\otimes}_{\bb{Q}_p} C^{la}(\widetilde{K}_p, \bb{Q}_p)_{\ket}) & \cong & \pi_{K_p,*}( \n{I}^{la}_{\Shan}). 
\end{aligned}
\end{equation}
Then the quasi-isomorphisms   \eqref{eqIso1LocAn} and \eqref{eqIso2LocAn} follow by Step 2 after  taking global sections on 
 \eqref{eqHigherVanishingIla}.

\textit{Step 4.} Finally, the isomorphism at the level of cohomology groups follows from the vanishing of higher locally analytic vectors of admissible representations of Proposition \ref{PropAdmissiblenoHigher}, and the spectral sequence of \cite[Theorem 1.5]{RRLocallyAnalytic}. Indeed, by the projection formula of locally analytic vectors \cite[Corollary 3.2.14 (3)]{RRLocAnII} we have that 
\[
(R\Gamma_{\proket}(\Shan^{\tor}_{K^p,\infty,C}, \bb{Z}_p)\widehat{\otimes}^L_{\bb{Z}_p} C)^{Rla}= (R\Gamma_{\proket}(\Shan^{\tor}_{K^p,\infty,C}, \bb{Q}_p))^{Rla}\widehat{\otimes}^L_{\bb{Q}_p} C,
\]
and $R\Gamma_{\proket}(\Shan^{\tor}_{K^p,\infty,C}, \bb{Q}_p)$ is quasi-isomorphic to a complex of admissible representations of $\widetilde{K}_p$ over $\bb{Q}_p$ by  Theorem \ref{TheoCompletedcohomologyAdmissible}. 
\end{proof}

The following lemma was used in the proof of Theorem \ref{TheoMainVanishingOla}

\begin{lem}\label{LemmaTensorInsideCohomology}
Let $V$ be a filtered colimit of Banach  $\bb{Q}_p$-linear representations of $\widetilde{K}_p$. Then there is a natural equivalence
\[
R\Gamma(\widetilde{K}_p, R\Gamma_{\proket}(\Shan^{\tor}_{K^p,\infty,C}, \iota_{J,*} \widehat{\s{O}}_{J})\widehat{\otimes}_{\bb{Q}_p}^L V) \cong R\Gamma_{\proket}(\Shan^{\tor}_{K^pK_p,C}, \iota_{J,*} \widehat{\s{O}}_{J} \widehat{\otimes}_{\bb{Q}_p}^L V_{\ket}). 
\]
\end{lem}
\begin{proof}
Since both terms commute with filetered colimits it suffices to construct a natural isomorphism when $V$ is a Banach representation. Moreover, both terms are the $\bb{Q}_p$-linear extensions of the analogue expressions when $\widehat{\s{O}}_{J}$ is replaced by $\widehat{\s{O}}_{J}^+$ and $V$ by a $p$-adically complete lattice $V^0$. Thus, it suffices to see that for a $p$-adically complete representation $V$ of $\widetilde{K}_p$ there is a natural quasi-isomorphism 
\begin{equation}\label{eqDevisageCohomologyInsideRGamma1}
R\Gamma(\widetilde{K_p}, R\Gamma_{\proket}(\Shan^{\tor}_{K^p,\infty,C}, \iota_{J,*} \widehat{\s{O}}_{J}^+)\widehat{\otimes}_{\bb{Z}_p}^L V) \cong R\Gamma_{\proket}(\Shan^{\tor}_{K^pK_p,C}, \iota_{J,*} \widehat{\s{O}}^+_{J} \widehat{\otimes}_{\bb{Z}_p}^L V_{\ket})
\end{equation}
where now the tensor products are derived $p$-complete tensor products.  On the other hand, the terms in \eqref{eqDevisageCohomologyInsideRGamma1} are derived $p$-complete, so it suffices to construct natural equivalences after taking reduction modulo $p^s$ for all $s\in \bb{N}$: 
\[
R\Gamma(\widetilde{K}_p, R\Gamma_{\ket}(\Shan^{\tor}_{K^p,\infty,C}, \iota_{J,*} \s{O}_{J}^+/p^s) \otimes_{\bb{Z}/p^s}^L V/p^s) \cong R\Gamma_{\ket}(\Shan^{\tor}_{K^pK_p,C}, \iota_{J,*}  \s{O}^+_{J}/p^s  \otimes_{\bb{Z}/p^s}^L (V/p^s)_{\ket}).
\]
But now the statement follows from the fact that $\Shan^{\tor}_{K^p,\infty,C}\to \Shan^{\tor}_{K^pK_p,C}$ is a pro-Kummer-\'etale $\widetilde{K}_p$-torsor of qcqs objects,  and from a Hochschild-Serre spectral sequence induced from the \v{C}ech nerve of $\Shan^{\tor}_{K^p,\infty,C}\to \Shan^{\tor}_{K^pK_p,C}$ in the pro-Kummer-\'etale site. 
\end{proof}

We deduce the rational vanishing of the Calegari-Emerton conjectures for Shimura varieties \cite[Conjecture 1.5]{CalegariEmerton}.

\begin{cor}
\label{CoroCalegari}
Let $d=\dim \Shan$ be the dimension of the Shimura variety, then for $i>d$ the  rational completed cohomology groups at level $K^p$ vanish
\[
\widetilde{H}^i(K^p,\bb{Z}_p)[\frac{1}{p}]=\widetilde{H}^i_{c}(K^p,\bb{Z}_p)[\frac{1}{p}]=0. 
\]
\end{cor}
\begin{proof}
We can prove the vanishing after extending scalars to $C$.  By \cite[Theorem 7.1]{SchTeitDist} (see also \cite[Corollary 4.49]{RRLocallyAnalytic}) it suffices to show that the locally analytic vectors of the completed cohomology  groups vanish for $i>d$.  By Theorem \ref{TheoMainVanishingOla} the latter can be computed in terms of sheaf cohomology of $\s{O}^{la}_{\Shan}$ and  $\n{I}^{la}_{\Shan}$ over $\Shan^{\tor}_{K^p,\infty,C}$, but this  space has cohomological dimension $d$ by the proof of \cite[Corollary IV.2.2]{ScholzeTorsion2015}, which implies the corollary.  
\end{proof}

A corollary of the computation of the geometric Sen operator Theorem \ref{TheoComputationSenOperator} is the vanishing of the action of $\f{n}^0_{\mu}$ on the sheaf $\s{O}^{la}_{\Shan}$.

\begin{cor}\label{CoroSenOpKillOla}
The action of the sub-Lie algebroid  $\f{n}^0_{\mu,\Shan}:=\s{O}^{la}_{\Shan}\otimes_{\pi_{\HT}^{\tor,-1}(\s{O}_{\Fl})} \pi_{\HT}^{\tor,-1}(\f{n}^0_{\mu})\subset \s{O}^{la}_{\Shan}\otimes_{\bb{Q}_p} \widetilde{\f{g}}$ on the sheaf $\s{O}^{la}_{\Shan}$ induced by derivations vanishes. 
\end{cor}
\begin{proof}
This is \cite[Corollary 3.2.6 (2)]{RCGeoSenTheory}, we repeat the argument for completeness of the paper. Let $\widetilde{U}\subset \Shan_{K^p,\infty, C}^{\tor}$ be a qcqs open subspace. 

 By Theorem \ref{TheoMainSenTheory}  (1) and Proposition \ref{LemmaLongExactSequenceOlaSheaves} (1) we have that 
\[
\s{O}^{la}_{\Shan}(\widetilde{U}) = (\widehat{\s{O}}_{\Shan}(\widetilde{U})\widehat{\otimes}_{\bb{Q}_p} C^{la}(\widetilde{K}_p,\bb{Q}_p))^{\f{n}^0_{\mu,\Shan}=0, \widetilde{K}_p}
\]
where $\widetilde{K}_p$ acts diagonally on the coefficients and via the left regular action on the locally analytic functions, and $\f{n}^0_{\mu,\Shan}$ acts by $\widehat{\s{O}}_{\Shan}(\widetilde{U})$-linear derivations via the left derivations of $\widetilde{\f{g}}$ on the locally analytic functions. Consider the orbit map 
\[
O:\s{O}^{la}_{\Shan}(\widetilde{U}) \to \widehat{\s{O}}_{\Shan}(\widetilde{U})\widehat{\otimes}_{\bb{Q}_p} C^{la}(\widetilde{K}_p,\bb{Q}_p),
\]
it sends an element $v\in \s{O}^{la}_{\Shan}(\widetilde{U})\subset \widehat{\s{O}}_{\Shan}(\widetilde{U})$ to the locally analytic function $O_v: \widetilde{K}_p\to \widehat{\s{O}}_{\Shan}(\widetilde{U})$ given by 
\[
O_v(g)= g\cdot v
\]
for $g\in \widetilde{K}_p$.  The orbit map $O$ is $\widetilde{K}_p$-equivariant for the $\widehat{\s{O}}_{\Shan}(\widetilde{U})$-linear right regular action of $\widehat{\s{O}}_{\Shan}(\widetilde{U})\widehat{\otimes}_{\bb{Q}_p} C^{la}(\widetilde{K}_p,\bb{Q}_p)$. Note that, for general vector field  $\f{Y}\in \s{O}^{la}_{\Shan}\otimes_{\bb{Q}_p} \widetilde{\f{g}}$ and $f: \widetilde{K}_p\to \widehat{\s{O}}_{\Shan}(\widetilde{U})$ a locally analytic function one has that 
\[
(\f{Y}\cdot_{L} f)(1)=- (\f{Y}\cdot_{R} f)(1)
\]
where $\cdot_{L}$ and $\cdot_{R}$ are the derivations with respect to the left and right regular action respectively. Indeed, this is a consequence of the fact that for $g\in \widetilde{K}_p$ the left and right regular action of $g$ on $f$ are related as follows
\[
(g\cdot_{L} f)(1)= f(g^{-1})= (g^{-1}\cdot_{R}f)(1).  
\]
 Then, for $\f{X}\in \f{n}^0_{\mu,\Shan}$ and $v\in \s{O}^{la}_{\Shan}(\widetilde{U})$ we have that 
\[
\f{X}\cdot v = (\f{X}\cdot_{R} O_v)(1)= - (\f{X}\cdot_L O_v)(1)=0
\]
proving what we wanted.
\end{proof}

Finally, we have the following corollary of the proof of Proposition \ref{LemmaLongExactSequenceOlaSheaves} describing the base change of $\s{O}^{la}_{\Shan}$ to  $\widehat{\s{O}}_{\Shan}$. 

\begin{cor}\label{CoroDescriptionBaseChangeOla}
Let $U\subset \Shan^{\tor}_{K^pK_p,C}$ be an open affinoid as in Proposition \ref{LemmaLongExactSequenceOlaSheaves}, $\widetilde{U}\subset \Shan^{\tor}_{K^p,\infty, C}$ its pullback at infinite level, and  $\psi\colon U\to \bb{T}^{e}_{C} \times \bb{D}^{d-e}_{C}$ a toric chart. Let $U_{\infty}$ be the $\Gamma$-torsor obtained from taking $p$-th power roots of the coordinates and let $\widetilde{U}_{\infty}=\widetilde{U}\times_{U} U_{\infty}$ be the pullback in the pro-Kummer-\'etale site. Then the orbit map produces a natural $\widehat{\s{O}}_{\Shan}(\widetilde{U}_{\infty})$-linear isomorphism
\[
\widehat{\s{O}}_{\Shan}(\widetilde{U}_{\infty})^{K_p-la,\Gamma-{\sm}}\widehat{\otimes}_{\widehat{\s{O}}_{\Shan}(\widetilde{U}_{\infty})^{K_p\times\Gamma-sm}} \widehat{\s{O}}_{\Shan}(\widetilde{U}_{\infty}) \xrightarrow{\sim } C^{la}(\widetilde{\f{g}}, \widehat{\s{O}}_{\Shan}(\widetilde{U}_{\infty}))^{\f{n}^0_{\mu, \star_1}=0},
\]
where $C^{la}(\widetilde{\f{g}}, \widehat{\s{O}}_{\Shan}(\widetilde{U}_{\infty}))=\varinjlim_{K_p} C^{la}(\widetilde{K}_p,\widehat{\s{O}}_{\Shan}(\widetilde{U}_{\infty}))$. 
\end{cor}
\begin{proof}
Consider the pro-Kummer-\'etale sheaf $C^{la}(\widetilde{K}_p,\widehat{\s{O}}_{\Shan}):=C^{la}(\widetilde{K}_p, \bb{Q}_p)_{\ket}\widehat{\otimes}_{\bb{Q}_p} \widehat{\s{O}}_{\Shan}$ on $\Shan_{K^pK_p,C, \proket}^{\tor}$. Write $C^{la}(K_p,\bb{Q}_p)=\varinjlim_{h} C(\bb{G}^{(h)}, \bb{Q}_p)$ as a colimit of analytic representations as $h\to \infty$. Then $C^{la}(\widetilde{K}_p,\widehat{\s{O}}_{\Shan})=\varinjlim_{h} C(\bb{G}^{(h)}, \widehat{\s{O}}_{\Shan})$. Since  $U_{\infty}$ is perfectoid, and $C(\bb{G}^{(h)} , \widehat{\s{O}}_{\Shan})$ is a relative locally analytic sheaf,   we have that 
\[
  C(\bb{G}^{(h)} , \widehat{\s{O}}_{\Shan}(\widetilde{U}_{\infty}))=  \s{F}_h \widehat{\otimes}_{\widehat{\s{O}}_{\Shan}(U_{\infty})}  \widehat{\s{O}}_{\Shan}(\widetilde{U}_{\infty})
\]
with  $\s{F}_h=( C(\bb{G}^{(h)}, \widehat{\s{O}}_{\Shan}(\widetilde{U}_{\infty})))^{\widetilde{K}_p}$ an  $\widehat{\s{O}}_{\Shan}(U_{\infty})$-semilinear relative locally analytic representation of $\Gamma$.  By \cite[Theorem 2.4.4]{RCGeoSenTheory}, $\s{F}_h$ satisfies $\s{F}_h^{R\Gamma-la}=\s{F}_h^{\Gamma-la}$ and 
\[
\s{F}_h^{\Gamma-la}\widehat{\otimes}_{\widehat{\s{O}}_{\Shan}(U_{\infty})^{\Gamma-la}} \widehat{\s{O}}_{\Shan}(U_{\infty})=  \s{F}_h
\]
where $\widehat{\s{O}}_{\Shan}(U_{\infty})^{\Gamma-la}=\widehat{\s{O}}_{\Shan}(U_{\infty})^{\Gamma-sm}=\varinjlim_{n} \s{O}_{\Shan}(U_n)$ where $U_n\to U$ is the finite Kummer-\'etale map obtained by taking $p^n$-th powers to the coordinates.   Now, consider the map $\theta_{\s{F}_h}\colon \s{F}_h\to \s{F}_h\otimes_{\s{O}_{\Shan}(U)} \Omega^1(\log)_U(-1)$ given by the geometric Sen operator with kernel $K_h$ and cokernel $Q_h$.  Note that  $\varinjlim_h K_h= \widehat{\s{O}}_{\Shan}(\widetilde{U}_{\infty})^{\widetilde{K}_p-la, \f{n}^0_{\mu}=0}$.

We claim that the natural map 
\begin{equation}\label{eqpampsq33d1d1}
\varinjlim_{h} K_h\widehat{\otimes}_{\widehat{\s{O}}_{\Shan}(U_{\infty})} \widehat{\s{O}}_{\Shan}(\widetilde{U}_{\infty}) \xrightarrow{\sim} C^{la}(\widetilde{K}_p,  \widehat{\s{O}}_{\Shan}(\widetilde{U}_{\infty})  )^{\f{n}^0_{\mu,\star_1}=0}
\end{equation}
is an equivalence.  Indeed, by the vanishing of the higher cohomology groups for the action of the geometric Sen operator of Proposition \ref{LemmaLongExactSequenceOlaSheaves}, we have a quasi-isomorphism
\[
\varinjlim_{h} K_h \xrightarrow{\sim} R\Gamma(\theta_{\s{F}}, \varinjlim_{h} \s{F}_h). 
\]
where $\theta_{\s{F}}$ is the geometric Sen operator of the $\s{F}_h$. Taking base change along $\widehat{\s{O}}_{\Shan}(U_{\infty})\to \widehat{\s{O}}_{\Shan}(\widetilde{U}_{\infty})$, we deduce that the map
\[
\varinjlim_{h} K_h  \widehat{\otimes}^L_{\widehat{\s{O}}_{\Shan}(U_{\infty})} \widehat{\s{O}}_{\Shan}(\widetilde{U}_{\infty}) \xrightarrow{\sim}R\Gamma(\theta_{\s{F}}, \varinjlim_{h} \s{F}_h  \widehat{\otimes}^L_{\widehat{\s{O}}_{\Shan}(U_{\infty})} \widehat{\s{O}}_{\Shan}(\widetilde{U}_{\infty}) = R\Gamma(\f{n}^0_{\mu}, C^{la}(\widetilde{K}_p,\widehat{\s{O}}_{\Shan}(\widetilde{U}_{\infty}) )
\]
is a quasi-isomorphism, and hence that  \eqref{eqpampsq33d1d1} is a quasi-isomorphism as wanted.

By \cite[Proposition 2.5.5]{RCGeoSenTheory}  the spaces $K_h$ and $Q_h$ have vanishing higher locally analytic vectors for the action of $\Gamma$. In particular, one has an exact sequence 
\[
0\to K_h^{\Gamma-la}\to \s{F}_h^{\Gamma-la}\to \s{F}^{\Gamma-la}_h\otimes_{\s{O}_{\Shan}(U)} \Omega^1_U(\log)(-1)\to Q^{\Gamma-la}_h\to 0. 
\]
We claim that the map $\s{O}_{\Shan}(U_n)\to  \s{O}_{\Shan}(U_{\infty})$ is flat for the solid tensor product, namely, it is the pullback along $\psi$ of the map $\bb{Q}_p\langle \underline{T}^{\pm 1},\underline{S} \rangle\to \bb{Q}_p\langle \underline{T}^{\pm 1/p^{\infty}},\underline{S}^{1/p^{\infty}} \rangle$,  and  the algebra $ \bb{Q}_p\langle \underline{T}^{\pm 1/p^{\infty}},\underline{S}^{1/p^{\infty}} \rangle$  admits an ON basis over $\bb{Q}_p\langle \underline{T}^{\pm 1},\underline{S} \rangle$, which produces an isomorphism 
\[
\bb{Q}_p\langle \underline{T}^{\pm 1},\underline{S} \rangle\widehat{\otimes}_{\bb{Q}_{p}} V_0\cong \bb{Q}_p\langle \underline{T}^{\pm 1/p^{\infty}},\underline{S}^{1/p^{\infty}} \rangle 
\]
with $V_0$ a suitable ON $\bb{Q}_p$-Banach space. Finally, the flatness of $V_0$ over $\bb{Q}_p$ for the solid tensor product (which follows from \cite[Lemma 3.21]{RRLocallyAnalytic}) yields the flatness of $\bb{Q}_p\langle \underline{T}^{\pm 1},\underline{S} \rangle\to \bb{Q}_p\langle \underline{T}^{\pm 1/p^{\infty}},\underline{S}^{1/p^{\infty}} \rangle$ as wanted. Taking colimits as $n\to \infty$, one deduces   that the map $\widehat{\s{O}}_{\Shan}(U_{\infty})^{\Gamma-\sm}\to \widehat{\s{O}}_{\Shan}(U_{\infty})$ is flat for the solid tensor product and therefore
\[
K_h^{\Gamma-la}\otimes_{\widehat{\s{O}}_{\Shan}(U_{\infty})^{\Gamma-\sm},\sol} \widehat{\s{O}}_{\Shan}(U_{\infty})= K_{h}.
\]
Taking colimits as $h\to \infty$ and a further base change along $\widehat{\s{O}}_{\Shan}(U_{\infty})\to \widehat{\s{O}}_{\Shan}(\widetilde{U}_{\infty})$, \eqref{eqpampsq33d1d1} yields
\begin{equation}\label{eqinfaaebkfabeulaef}
C^{la}(\widetilde{K}_p,\widehat{\s{O}}_{\Shan}(\widetilde{U}_{\infty}) )^{\f{n}^0_{\mu, \star,1}=0} =\varinjlim_{h}  K_h^{\Gamma-la}\otimes_{\widehat{\s{O}}_{\Shan}(U_{\infty})^{\Gamma-\sm},\sol} \widehat{\s{O}}_{\Shan}(\widetilde{U}_{\infty}). 
\end{equation}
Now, $K_h^{\Gamma-la}$ has trivial geometric Sen operator by construction, and so $K_h^{\Gamma-la}=K_p^{\Gamma-\sm}$. Hence $\varinjlim_h K_h^{\Gamma-la}=\widehat{\s{O}}_{\Shan}(\widetilde{U}_{\infty})^{\widetilde{K}_p-la,\Gamma-sm}$. Thus, taking colimits as $K_p\to 1$, we get
\[
C^{la}(\f{g},\widehat{\s{O}}_{\Shan}(\widetilde{U}_{\infty}) )^{\f{n}^0_{\star,1}=0}=  \widehat{\s{O}}_{\Shan}(\widetilde{U}_{\infty})^{\widetilde{K}_p-la,\Gamma-sm}\widehat{\otimes}_{ \widehat{\s{O}}_{\Shan}(\widetilde{U}_{\infty})^{K_p\times \Gamma-sm}} \widehat{\s{O}}_{\Shan}(\widetilde{U}_{\infty})
\]
proving what we wanted. 
\end{proof}

\section{Arithmetic Sen operator of completed cohomology}\label{SectionArithmeticSen}

Let $L$ be a finite extension of $\bb{Q}_p$. In this last section we define a  notion of arithmetic Sen operator for  $\bb{C}_p$-semilinear solid Galois representations. We will prove that the locally analytic completed cohomology admits an arithmetic Sen operator and then  compute it in terms of $\pi_{\HT}^{\tor}$ and the Lie algebras of Section \ref{Section:EquivariantFlag} over the flag variety.

In the rest of the section we assume that all the solid $\bb{Q}_p$-algebras have the induced analytic structure from $\bb{Q}_{p,\sol}$ \cite[Definition 2.3.13 (2)]{mann2022padic}.  We shall need the  notion of the derived category of semilinear solid representations. 

\begin{definition}
Let $\Pi$ be a profinite group and $A$ a solid $\bb{Q}_p$-algebra endowed with an action of $\Pi$. We let $\Rep^{\sol}_{A}(\Pi)$ denote the abelian category of $A$-semilinear solid  representations of $\Pi$, and let $D(\Rep^{\sol}_A(\Pi))$ be its derived category.
\end{definition}

\subsection{Arithmetic Sen operator of solid $\bb{C}_p$-semilinear $\Gal_{L}$-representations}

In order to define a general notion of arithmetic Sen operator that is useful to deal with cohomology  we must work with derived categories and solid $\bb{C}_p$-equivariant $\Gal_{L}$-representations. For this, let us introduce some notation. 

Let $L_{\infty}= L(\zeta_{p^{\infty}})$ be the algebraic extension of $L$ obtained by adding the $p$-th power roots of unit, and let $L^{\cyc}$ be the completion of $L_{\infty}$ to a perfectoid field. For $k\in \bb{N}$ let $L_k=L(\zeta_{p^k})$. Let us write $\Gamma^{\mathrm{arith}}:=\Gal(L_{\infty}/L)$ and $H= \Gal_{L_{\infty}}$ so that we have a short exact sequence
\[
1\to H\to \Gal_{L}\to \Gamma^{\mathrm{arith}}\to 1. 
\] 
We let $\Rep^{la}_{L_{\infty}}(\Gamma^{\mathrm{arith}})$ be the abelian category of $L_{\infty}$-semilinear solid locally analytic representations of $\Gamma^{\mathrm{arith}}$ and let $D(\Rep^{la}_{L_{\infty}}(\Gamma^{\mathrm{arith}}))$ be its derived category.

The following lemma  describes the process of decompletion by taking locally analytic vectors.

\begin{lem}\label{LemDecompletionArith1}
Set $\Gamma=\Gamma^{\mathrm{arith}}$. The natural map 
\[
L_{\infty}\to R\Gamma(H, \bb{C}_p)^{R\Gamma-la}
\]
is an equivalence. In particular, we have a decompletion by locally analytic vectors 
\[
RS_{L_{\infty}}: D(\Rep_{\bb{C}_p}^{\sol}(\Gal_L)) \to D(\Rep^{la}_{L_{\infty}}(\Gamma))
\]
given by 
\[
RS_{L_{\infty}}(P):= R\Gamma(H, P)^{R\Gamma-la}.
\]
\end{lem}
\begin{proof}
The following argument goes back to Tate \cite{Tatepdivisible}. By acyclicity of the   pro\'etale cohomology for  $\widehat{\s{O}}$  in perfectoid affinoid spaces, we have that $R\Gamma(H, \bb{C}_p)=L^{\cyc}$ where $L^{\cyc}$ is the completed cyclotomic extension of $L$. The equivalence $L_{\infty}\to (L^{\cyc})^{R\Gamma-la}$ follows from  \cite[Lemma 2.4.3 (3)]{RCGeoSenTheory} and the fact that  $(L^{\cyc}, \Gamma)$ gives rise to an abstract Sen theory as in \cite[Definition 2.2.1]{RCGeoSenTheory} thanks to Tate's normalized traces. 

 For the second claim, note that $RS_{L_{\infty}}$ is the right derived functor  of the decompletion functor 
\[
S_{L_{\infty}}: \Rep_{\bb{C}_p}^{\sol}(\Gal_{L})\to \Rep_{L_{\infty}}^{la}(\Gamma)
\]
given by 
\[
S_{L_{\infty}}(W)=W^{H, \Gamma-la}. 
\]
\end{proof}

\begin{definition}\label{DefinitionDecompletionByLocAn}
Let $P\in D(\Rep^{\sol}_{\bb{C}_p}(\Gal_{L}))$ be a derived $\bb{C}_p$-semilinear solid representation of $\Gal_{L}$.  We say that $P$ admits a decompletion by locally analytic vectors if the natural base change
\[
RS_{L_{\infty}}(P)\otimes_{L_{\infty},\sol}^{L} \bb{C}_p \xrightarrow{\sim} P
\]
is a quasi-isomorphism. If this is the case, we define the arithmetic Sen operator  of $P$ to be the $\bb{C}_p$-base change of the action by derivations on $RS_{L_{\infty}}(P)$  of the element $\theta^{\mathrm{arith}}\in \Lie \Gamma^{\mathrm{arith}}$ given by the derivative $ \partial_{\chi}$, where $\chi$ is the coordinate given by the cyclotomic character $\chi\colon \Gamma^{\mathrm{arith}}\to \mathbb{Z}_p^{\times}$. 
\end{definition}

The following lemma shows some stability properties of semilinear representations admitting decompletions by locally analytic vectors, and that decompletions are unique  under some mild hypothesis. 

\begin{lem}\label{LemmaStabilityDecompletions}
Set $\Gamma=\Gamma^{\mathrm{arith}}$. Let $\s{C}\subset D(\Rep^{\sol}_{\bb{C}_p}(\Gal_L))$ be the full subcategory of semilinear representations admitting a decompletion by locally analytic vectors. Then $\s{C}$ is stable under cones, finite direct sums and retracts.  Furthermore, if $W_0\in D(\Rep^{la}_{L_{\infty}}(\Gamma))^{\mathrm{nuc}}$ is a nuclear derived $L_{\infty}$-semilinear locally analytic representation of $\Gamma$ (cf. Remark \ref{RemNucMod}), the natural map 
\[
W_0 \to RS_{L_{\infty}}(W_0\otimes^L_{L_{\infty},\sol} \bb{C}_p )
\]
is a quasi-isomorphism.  In particular,  if $W\in D(\Rep_{\bb{C}_p}^{\sol}(\Gal_L))^{\mathrm{nuc}}$  is a nuclear semilinear representation, it admits a decompletion via locally analytic vectors if and only if there is an object $W_0\in D(\Rep^{la}_{L_{\infty}}(\Gamma))^{\mathrm{nuc}}$ such that $W_0\otimes^L_{L_{\infty}} \bb{C}_p\cong W$ as $\bb{C}_p$-semilinear representations of $\Gal_L$, if that is the case, there is a natural quasi-isomorphism $W_0\cong RS_{L_{\infty}}(W)$ as $L_{\infty}$-semilinear locally analytic representations of $\Gamma$.
\end{lem}

\begin{remark}\label{RemNucMod}
In Lemma \ref{LemmaStabilityDecompletions}, we say that a solid $\bb{Q}_p$-linear representation of a profinite group $\Pi$ is nuclear if its underlying  solid $\bb{Q}_p$-vector space is nuclear (cf. \cite[Definition 3.14]{RRLocallyAnalytic}). An object $W$ in the derived category of solid $\bb{Q}_p$-linear representations of $\Pi$ is said nuclear if $H^{i}(W)$ is nuclear for all $i\in \bb{Z}$. By \cite[Proposition 3.29]{RRLocallyAnalytic}, Fr\'echet spaces are nuclear solid $\bb{Q}_p$-vector spaces. 
\end{remark}

\begin{proof}
The first claim follows from the fact that $RS_{L_{\infty}}$ and $\otimes_{L_{\infty},\sol}^L \bb{C}_p$  are exact functors of triangulated categories.    Now let $W_0\in D(\Rep^{la}_{L_{\infty}}(\Gamma))^{\mathrm{nuc}}$, we want to prove that the natural map 
\[
W_0\to RS_{L_{\infty}}(W_0\otimes_{L_{\infty},\sol}^L \bb{C}_p)
\]
is an equivalence. We first prove that the map 
\begin{equation}\label{eqpkfapmaps}
W_0\otimes^L_{L_{\infty},\sol} L^{\cyc} \xrightarrow{\sim} R\Gamma(H, W_0\otimes^L_{L_{\infty},\sol} \bb{C}_p)
\end{equation}
is a quasi-isomorphism. By \cite[Lemma 3.21]{RRLocallyAnalytic},  for any finite extension $F/\bb{Q}_p$,  any quasi-separated solid $F$-vector space $V$ is flat  for the solid tensor product. In particular, by taking filtered colimits, we see that $L_{\infty}\to \bb{C}_p$ and $L_{\infty}\to L^{\cyc}$ are $\otimes_{\sol}$-flat.  Therefore,  by writing $W_0=R\varprojlim_n \tau^{\geq -n} W_0$ as limit of its left canonical truncations, to show that \eqref{eqpkfapmaps} is an equivalence we can assume without loss of generality that $W_0$ is coconnective, i.e. concentrated in cohomological degrees $\geq 0$. On the other hand, by \cite[Lemma 5.2]{RRLocallyAnalytic} (1) the trivial $\bb{Q}_p$-representation of $H$ has a  resolution in terms of solid compact projective  modules over the Iwasawa algebra $\bb{Q}_{p,\sol}[H]=(\bb{Z}_p[[H]])[\frac{1}{p}]$. In particular, $H$-cohomology commutes with filtered colimits of coconective objects, see \cite[Proposition 4.12]{ClausenScholzeCondensed2019}.  Hence,  by writing $W_0=\varinjlim_{n} \tau^{\leq n} W_0$ as the filtered colimit of its right canonical truncations, we can assume without loss of generality that $W_0$ is bounded. By a further inductive argument we can assume that $W_0$ is even in degree $0$.

By hypothesis, $W_0$ is a nuclear $\bb{Q}_p$-vector space, and \cite[Theorem A.43]{Bosco2021padic} implies that both $W_0\otimes_{L_{\infty},\sol} L^{\cyc}$ and $W_0\otimes_{L_{\infty},\sol} \bb{C}_p$ are also nuclear. On the other hand, \cite[Lemma 5.2 (1)]{RRLocallyAnalytic} implies that $R\Gamma(H,  W_0\otimes_{L_{\infty},\sol} \bb{C}_p)$ is represented by the complex of solid cochains
\[
\begin{aligned}
\mathrm{Ch}^{\bullet}(H;W_0\otimes_{L_{\infty},\sol} \bb{C}_p ):=[W_0\otimes_{L_{\infty},\sol} \bb{C}_p  \to \underline{\Hom}_{\bb{Q}_p}(\bb{Q}_{p,\sol}[H],W_0\otimes_{L_{\infty},\sol} \bb{C}_p ) \\ \to \underline{\Hom}_{\bb{Q}_p}(\bb{Q}_{p,\sol}[H^2],W_0\otimes_{L_{\infty},\sol} \bb{C}_p )\to \cdots ]
\end{aligned}
\]
Since $W_0\otimes_{L_{\infty},\sol} \bb{C}_p$ is nuclear and $H$ is profinite,  we have  that 
\[
\underline{\Hom}_{\bb{Q}_p}(\bb{Q}_{p,\sol}[H^k],W_0\otimes_{L_{\infty},\sol} \bb{C}_p ) = C(H^k,\bb{Q}_p)\otimes_{\bb{Q}_p,\sol} W_0\otimes_{L_{\infty},\sol} \bb{C}_p
\] 
for all $k\in \bb{N}$, where $C(H^k,\bb{Q}_p)=\underline{\Hom}_{\bb{Q}_p}(\bb{Q}_{p,\sol}[H^k], \bb{Q}_p)$ is the Banach space of $\bb{Q}_p$-valued continuous functions of $H^k$. Therefore, since $H$ acts trivially on $W_0$ and $L_{\infty}$, we have that
\[
\mathrm{Ch}^{\bullet}(H;W_0\otimes_{L_{\infty},\sol} \bb{C}_p )= W_0\otimes_{L_{\infty},\sol} \mathrm{Ch}^{\bullet}(H;\bb{C}_p ).
\]
In other words, we have the projection formula for group cohomology on nuclear representations
\[
W_0\otimes_{L_{\infty},\sol}^L R\Gamma(H, \bb{C}_p)\xrightarrow{\sim} R\Gamma(H,W_0\otimes_{L_{\infty},\sol} \bb{C}_p).
\]
But pro\'etale descent yields $R\Gamma(H, \bb{C}_p)= L^{\cyc}$, proving that \eqref{eqpkfapmaps} is an equivalence as wanted.  

Finally, we want to prove that the map $W_0\to (W_0\otimes^L_{L_{\infty},\sol} L^{\cyc})^{R\Gamma-la}$ is a quasi-isomorphism.  Set $L_n=\bb{Q}_p(\zeta_{p^n})$, by \cite[Proposition 3.2.6 (3)]{RRLocAnII} the functor of locally analytic vectors preserves filtered colimits, hence it suffices to show that the natural map
\[
W_0\otimes_{L_n,\sol} L_{\infty} \to (W_0\otimes_{L_n,\sol} L^{\cyc})^{R\Gamma-la}
\]
is a quasi-isomorphism. But then, by the projection formula of locally analytic vectors \cite[Corollary 3.2.14 (3)]{RRLocAnII} (that we can apply since an open subgroup of $\Gamma$ acts trivially on $L_n$), we have that 
\[
(W_0\otimes_{L_n,\sol} L^{\cyc})^{R\Gamma-la}= W_0\otimes_{L_n,\sol} (L^{\cyc})^{R\Gamma-la} = W_0\otimes_{L_n,\sol} L_{\infty} 
\]
where in the last equivalence we use Lemma \ref{LemDecompletionArith1}. This finishes the proof. 
\end{proof}

\begin{cor}\label{CoroExactDecompletion}
Keep the notation of Lemma \ref{LemmaStabilityDecompletions}. Let  $W\in D(\Rep_{\bb{C}_p}^{\sol}(\Gal_L))^{\mathrm{nuc}}$  be an object admitting a decompletion $W_0\in D(\Rep^{la}_{L_{\infty}}(\Gamma))^{\mathrm{nuc}}$. Then  the cohomology groups $H^i(W)$ admit decompletions given by $H^i(W_0)$ for $i\in \bb{Z}$.  In particular, if $W$ sits in degree $0$ and admits a decompletion $W_0$, then $W_0$ also sits in degree $0$. 
\end{cor}
\begin{proof}
By Lemma \ref{LemmaStabilityDecompletions} we have an isomorphism of semilinear representations $W=\bb{C}_p\otimes^L_{L_{\infty},\sol} W_0$. Hence, the statement will follow if the functor $\bb{C}_p\otimes^L_{L_{\infty},\sol} $ is $t$-exact. But we have $\bb{C}_p\otimes^L_{L_{\infty},\sol}-= \varinjlim_{n} \bb{C}_p\otimes^L_{L_n,\sol} -$ where $L_n=L(\zeta_{p^n})$, thus it suffices to see that $\bb{C}_p\otimes^L_{L_n,\sol} -$ is $t$-exact, which follows from \cite[Lemma 3.21]{RRLocallyAnalytic} and the fact that $\bb{C}_p$ is quasi-separated as solid $L_n$-vector space. 
\end{proof}

\subsection{Arithmetic Sen operator of completed cohomology}

 Recall from Section \ref{Section:EquivariantFlag} that over $\Fl$ we have Lie algebroids $\f{n}^0_{\mu}\subset \f{p}^0_{\mu}\subset \f{g}^0$ on $\Fl$ corresponding to the $\bbf{P}_{\mu}$-equivariant adjoint actions $\f{n}_{\mu}\subset \f{p}_{\mu}\subset \f{g}$. On the other hand, the Hodge cocharacter $\mu:\bb{G}_m\to \bbf{G}_L$ lands in $\bbf{P}_{\mu}$, and the passage to tangent spaces gives rise an element $\theta_{\mu}\in \f{p}_{\mu}$. By definition of the Levi subgroup $\bbf{M}_{\mu}$ as the centralizer of $\mu$, the  adjoint action of $\theta_{\mu}$ on $\f{m}_{\mu}$  is trivial. In this way, we have a $\bbf{P}_{\mu}$-equivariant extension of adjoint representations
\[
\begin{tikzcd}
0 \ar[r]& \f{n}_{\mu} \ar[r] \ar[d] & \f{p}^+_{\mu} \ar[r]\ar[d] & L\cdot \theta_{\mu} \ar[r]\ar[d] & 0 \\
0 \ar[r] & \f{n}_{\mu}\ar[r] & \f{p}_{\mu}  \ar[r] & \f{m}_{\mu} \ar[r] & 0
\end{tikzcd}
\]
where the right square is a pullback square. Passing to $\bbf{G}$-equivariant vector bundles over $\Fl$ the upper short exact sequence produces an exact sequence
\[
0\to \f{n}^0_{\mu} \to \f{p}^{0,+}_{\mu} \to \s{O}_{\Fl}\cdot \theta_{\mu}\to 0. 
\]
By  Corollary \ref{CoroSenOpKillOla}, the action of $\mathfrak{g}^0$ on $\s{O}^{la}_{\Shan}$ vanishes when restricted to $\mathfrak{n}^0_{\mu}$. This produces an action of $\mathfrak{m}_{\mu}$ on $\s{O}^{la}_{\Shan}$, called the horizontal action, and in particular an action of the operator $\theta_{\mu}\in  \mathfrak{m}_{\mu}$.

Let $\pi_{K_p}: \Shan^{\tor}_{K^p,\infty,L}\to \Shan^{\tor}_{K^pK_p,L}$ be the natural map and $\pi_{\HT}^{\tor }: \Shan^{\tor}_{K^p,\infty,L}\to \Fl_{L}$ the Hodge-Tate period map .

\begin{theo}\label{TheoArithSenOla}
 Let $U\subset \Shan^{\tor}_{K^pK_p,L}$ be an open affinoid subspace  satisfying  the conditions of Proposition \ref{LemmaLongExactSequenceOlaSheaves}. Then   the $\bb{C}_p$-semilinear $\Gal_{L}$-representation $\s{O}^{la}_{\Shan}(\pi_{K_p}^{-1}(U_{\bb{C}_p}))$ admits a decompletion by locally analytic vectors as in Definition \ref{DefinitionDecompletionByLocAn}.  Moreover,  
\[
RS_{L_{\infty}}(\s{O}^{la}_{\Shan}(\pi_{K_p}^{-1}(U_{\bb{C}_p})))=\bigg( \s{O}^{la}_{\Shan}(\pi_{K_p}^{-1}(U_{\bb{C}_p}))^{H} \bigg)^{\Gamma^{\mathrm{arith}}-la}
\]
sits in degree $0$ and the action of the arithmetic Sen operator is given by $-\theta_{\mu}$.  Moreover, the natural map 
\[
\bigg( \s{O}^{la}_{\Shan}(\pi_{K_p}^{-1}(U_{\bb{C}_p}))^{H} \bigg)^{\Gamma^{\mathrm{arith}}-la}\widehat{\otimes}_{L_{\infty}}\bb{C}_p \to \s{O}^{la}_{\Shan}(\pi_{K_p}^{-1}(U_{\bb{C}_p})
\]
is an isomorphism. In other words, $\s{O}^{la}_{\Shan}(\pi_{K_p}^{-1}(U_{\bb{C}_p})$ admits a decompletion as in Definition \ref{DefinitionDecompletionByLocAn}.

 An analogue  statement holds for the sheaves $\s{O}^{la}_{J}$ associated to the boundary divisors, and the sheaf $\n{I}^{la}_{\Shan}$ of Definition \ref{DefinitionOlaD}.
\end{theo}

For simplicity and to light notation, we shall prove the theorem only in the case where $J=\emptyset$, that is when $\s{O}^{la}_J=\s{O}^{la}_{\Shan}$.  The case for general $J$ is proven in the exact same way. The case of $\n{I}^{la}_{\Shan}$ follows from the long exact sequence of divisors of Proposition \ref{LemmaLongExactSequenceOlaSheaves}, and the fact that the  decompletion remains exact thanks to Corollary \ref{CoroExactDecompletion}.

There are two main statements to prove in Theorem \ref{TheoArithSenOla}, that is, the existence of the decompletion of  the locally analytic sheaf $\s{O}^{la}_{\Shan}$ with respect to the Galois action  (and hence the decompletion for locally analytic completed cohomology), and the computation of its arithmetic Sen operator. 

\subsection{Proof of the arithmetic decompletion}\label{SubSectionArithmeticDecompletion}

\begin{proof}[Proof of Theorem  \ref{TheoArithSenOla}: existence of the arithmetic decompletion]
 We keep the notation of the proof of Proposition \ref{LemmaLongExactSequenceOlaSheaves} except that now we consider the Shimura variety as a log adic space of finite type over $L$. In particular, there is an open affinoid $V\subset \Fl_{L}$ such that $\f{n}^0_{\mu}$ and $\f{g}^0/\f{n}^0_{\mu}$ are finite free when restricted to $V$, $\widetilde{U}:=\pi_{K_p}^{-1}(U)\subset \pi_{\HT}^{\tor,-1}(V)$ with $U\subset \Shan^{\tor}_{K^pK_p,L}$ an open affinoid subspace,  and there is a toric chart $\psi: U\to \bb{T}^e_{L}\times \bb{D}_{L}^{d-e}$ over $L$.  We write $U_{\infty}$ for the pre-perfectoid product of tori and polydics over $U$ (resp. $U_n$ for the finite level ones), and $\widetilde{U}_{\infty}=\widetilde{U}\times_U U_{\infty}$ for the product in the pro-Kummer-\'etale site of $U$. Note that the map $U_{p^{\infty},L^{\cyc}}\to U$ is a pro-Kummer-\'etale Galois cover of group  the semidirect product $\Gamma^{\mathrm{arith}}\ltimes \Gamma_p$.

\textit{Step 1.}  Consider the $L^{\cyc}$-base change $\Shan^{\tor}_{K^p,\infty,L^{\cyc}}$ of the infinite level Shimura variety and let $|\Shan^{\tor}_{K^p,\infty,L^{\cyc}}|=\varprojlim_{K_p} |\Shan^{\tor}_{K^pK_p,L^{\cyc}}|$ be its underlying topological space. Define a sheaf of $\widetilde{K}_p$-locally analytic functions $\s{O}^{la}_{\Shan,L^{\cyc}}$ over $|\Shan^{\tor}_{K^p,\infty,L^{\cyc}}|$ as in Definition \ref{DefinitionOla} (2). We  compare the sheaves $\s{O}^{la}_{\Shan}$ and $\s{O}^{la}_{\Shan,L^{\cyc}}$.

By Remark \ref{RemaSimplification}  and  Proposition \ref{LemmaLongExactSequenceOlaSheaves} one finds that 
\begin{equation}\label{eqAciclicityOlaLcyc}
\begin{aligned}
\s{O}^{la}_{\Shan,L^{\cyc}}(\widetilde{U}_{L^{\cyc}}) & \cong R\Gamma_{\proket}(U_{L^{cyc}}, \widehat{\s{O}}_{\Shan}\widehat{\otimes}_{\bb{Q}_p} C^{la}(\widetilde{K}_p, \bb{Q}_p)_{\ket})\\ & \cong  R\Gamma(\widetilde{K}_p\times \Gamma_p,  \widehat{\s{O}}_{\Shan}(\widetilde{U}_{p^{\infty},L^{\cyc}})\widehat{\otimes}_{\bb{Q}_p} C^{la}(\widetilde{K}_p,\bb{Q}_p)),
\end{aligned}
\end{equation}
namely, the only property we used is that $\widetilde{U}_{\infty,L^{cyc}}$ was a log affinoid perfectoid, so that the vanishing of pro-Kummer-\'etale $\s{O}^+/p$-cohomology yields
\[
R\Gamma_{\proket}(\widetilde{U}_{\infty,L^{cyc}}, \widehat{\s{O}}_{\Shan}\widehat{\otimes}_{\bb{Q}_p} W)= R\Gamma(H, \widehat{\s{O}}_{\Shan}(\widetilde{U}_{\infty, \bb{C}_p})\widehat{\otimes}_{\bb{Q}_p} W)=\widehat{\s{O}}_{\Shan}(\widetilde{U}_{\infty, L^{\cyc}})\widehat{\otimes}_{\bb{Q}_p} W
\] 
for any (colimit of) Banach $\bb{Q}_p$-vector space $W$. Then, taking invariants under $\Gamma^p\cong \widehat{\bb{Z}}^{(p),d}$ one gets 
\[
R\Gamma_{\proket}(\widetilde{U}_{p^{\infty},L^{cyc}}, \widehat{\s{O}}_{\Shan}\widehat{\otimes}_{\bb{Q}_p} W)=R\Gamma(H, \widehat{\s{O}}_{\Shan}(\widetilde{U}_{p^{\infty}, \bb{C}_p})\widehat{\otimes}_{\bb{Q}_p} W)=\widehat{\s{O}}_{\Shan}(\widetilde{U}_{p^{\infty}, L^{\cyc}})\widehat{\otimes}_{\bb{Q}_p} W.
\] 

 Moreover, since 
\[
\widehat{\s{O}}_{\Shan}(\widetilde{U}_{p^{\infty},\bb{C}_p})= \widehat{\s{O}}_{\Shan}(\widetilde{U}_{p^{\infty},L^{\cyc}})\widehat{\otimes}_{L^{\cyc}} \bb{C}_p,
\]
we have $\widehat{\s{O}}_{\Shan}(\widetilde{U}_{p^n,\bb{C}_p})=\widehat{\s{O}}_{\Shan}(\widetilde{U}_{p^{n},L^{\cyc}})\widehat{\otimes}_{L^{\cyc}} \bb{C}_p$ for all $n\in \bb{N}$ (after taking $\Gamma_p^{p^n}$-invariants), and by passing to $K_p$-locally analytic vectors, we must have 
\[
\s{O}^{la}_{\Shan}(\widetilde{U}_{p^{n},\bb{C}_p})= \s{O}^{la}_{\Shan,L^{\cyc}}(\widetilde{U}_{p^{n}, L^{\cyc}})\widehat{\otimes}_{L^{\cyc}} \bb{C}_p
\]
  Indeed, this follows from the projection formula of $\widetilde{K_p}$-locally analytic vectors, see \cite[Lemma 2.1.6]{RRLocAnII} or  \cite[Corollary 3.1.15 (3)]{RCGeoSenTheory}.

Therefore, it suffices to show that $\s{O}^{la}_{\Shan,L^{\cyc}}(\widetilde{U}_{p^{n}, L^{\cyc}})$ admits a decompletion by locally analytic vectors for the action of $\Gamma^{\mathrm{arith}}$ and  $n=0$.  We will even show that it admits a decompletion for all $n\in \mathbb{N}$.

\textit{Step 2.} It suffices to show that the colimit  $\varinjlim_{n\in \bb{N}}\s{O}^{la}_{\Shan,L^{cyc}}(\widetilde{U}_{p^{n}, L^{\cyc}})$  admits a decompletion by locally analytic vectors for the action of $\Gamma^{\mathrm{arith}}$.  Indeed,  suppose that this is the case and let  us denote $W=\varinjlim_{n\in \bb{N}}\s{O}^{la}_{\Shan,L^{cyc}}(\widetilde{U}_{p^{n}, L^{\cyc}})$, then $W$ is a $\Pi=\Gamma^{\mathrm{arith}}\ltimes \Gamma_{p}$-representation which is smooth when restricted to $\Gamma_{p}$. By  Lemma \ref{LemmaSemidirectAnVectors} (more precisely by taking colimits with respect to rigid group neighbourhoods of $\Gamma^{\mathrm{arith}}$ and $\Gamma_p$), we have that 
\begin{equation}\label{eq1oiqkwdq}
W^{R\Pi-la}\cong (W^{R\Gamma_p-la})^{R\Gamma^{\mathrm{arith}}-la} = W^{R\Gamma^{\mathrm{arith}}-la } = W_0
\end{equation}
sits in degree $0$ (thanks to Corollary \ref{CoroExactDecompletion}) and is the decompletion of $W$ with respect to $\Gamma^{\mathrm{arith}}$. In \eqref{eq1oiqkwdq} the second quasi-isomorphism follows from the fact that a $\Gamma_p$-smooth representation is already  locally analytic. In particular, $W_0$ has a natural action of $\Pi$ and we have an equivalence of $\Pi$-equivariant solid representations 
\[
W= W_0\otimes^L_{L_{\infty},\sol} L^{\cyc}=W_0\otimes_{L_{\infty},\sol} L^{\cyc}
\]
(where the last equivalence follows from the flatness of $L^{\cyc}$ over $L_{\infty}$ for the solid tensor product as used in the proof of Corollary \ref{CoroExactDecompletion}). Since $\Gamma_p$ acts $L^{\cyc}$-linearly, we find an isomorphism 
\[
\s{O}^{la}_{\Shan,L^{\cyc}}(\widetilde{U}_{p^n},L^{\cyc})= W^{\Gamma_p^{p^n}}= W_0^{\Gamma_p^{p^n}}\otimes_{L_{\infty},\sol} L^{\cyc}
\]
of $\Gamma^{\mathrm{arith}}$-representations. Since $ W_0^{\Gamma_p^{p^n}}$ is $\Gamma^{\mathrm{arith}}$-locally analytic as $W_0$ is so (thanks to Lemma \ref{LemmaInvariantCompactPreservesLocAn}) we deduce that $\s{O}^{la}_{\Shan,L^{\cyc}}(\widetilde{U}_{p^n,L^{\cyc}})$ admits a decompletion as in Definition \ref{DefinitionDecompletionByLocAn}.

\textit{Step 3.} By \eqref{eqAciclicityOlaLcyc} we have that 
\[
\s{O}^{la}_{\Shan,L^{\cyc}}(\widetilde{U}_{p^n,L^{\cyc}})= R\Gamma(\widetilde{K}_p\times \Gamma_p^{p^n}, \widehat{\s{O}}_{\Shan}(\widetilde{U}_{p^{\infty},L^{cyc}})\widehat{\otimes}_{\bb{Q}_p}  C^{la}(\widetilde{K}_p, \bb{Q}_p)). 
\]
Let us write $C^{la}(\widetilde{K}_p,\bb{Q}_p)=\varinjlim_{h\to \infty} C(\bb{G}^{(h)}, \bb{Q}_p)$ as a colimit of $h$-analytic functions of $\widetilde{K}_p$ (depending on a the coordinates of a fixed normal open uniform pro-$p$-subgroup), and let 
\[
\s{F}_{h}:= \Gamma(\widetilde{K}_p,\widehat{\s{O}}_{\Shan}(\widetilde{U}_{p^{\infty},L^{cyc}})\widehat{\otimes}_{\bb{Q}_p}C(\bb{G}^{(h)}, \bb{Q}_p) ) =  R\Gamma(\widetilde{K}_p,\widehat{\s{O}}_{\Shan}(\widetilde{U}_{p^{\infty},L^{cyc}})\widehat{\otimes}_{\bb{Q}_p}C(\bb{G}^{(h)}, \bb{Q}_p) ), 
\]
where the vanishing of higher cohomology follows from \cite[Lemma 3.1.5 and Remark 3.2.2]{RCGeoSenTheory} as $U_{p^{\infty}}$ is perfectoid.    Let $B_{\infty}:=\widehat{\s{O}}_{\Shan}(U_{p^{\infty},L^{\cyc}})$ and $B_n:=\s{O}_{\Shan}(U_{p^n,L^{\cyc}})$,  by \cite[Proposition 2.2.14]{RCGeoSenTheory} there are normalized traces $R_n\colon B_{\infty}\to B_n$ for $n\gg 0$ with respect to the action of $\Gamma_p$, giving rise to a  Sen theory on $B_{\infty}$ as in \cite[Definition 2.2.5]{RCGeoSenTheory}. Moreover, as $C(\mathbb{G}^{(h)}, \bb{Q}_p)$ is a locally analytic Banach representation of $\widetilde{K}_p$, the $\Gamma^{\mathrm{arith}}\ltimes \Gamma_p$-representation  $\s{F}_{h}$ is an $B_{\infty}$-semilinear relative locally analytic (to see this, it suffices to notice that the pro-Kummer-\'etale sheaf  $C(\bb{G}^{(h)}, \bb{Q}_p)_{\ket}\widehat{\otimes}_{\bb{Q}_p} \widehat{\s{O}}_{\Shan}$ is relative locally analytic, and then so it is its evaluation at $U_{p^\infty,L^{\cyc}}$ being perfectoid).   Therefore, by \cite[Theorem 2.4.4]{RCGeoSenTheory} the $\Gamma_p$-representation $\s{F}_h$ admits a decompletion via locally analytic vectors, given  by $\s{F}_h^{\Gamma_p-la}$, and by \cite[Corollary 2.5.2]{RCGeoSenTheory} we have that 
\[
R\Gamma(\Gamma_{p}^{p^n}, \s{F}_{h})= R\Gamma(\Gamma_p^{p^n,\sm}, R\Gamma(\Lie \Gamma_p,\s{F}_{h}^{\Gamma_{p}-la})),
\]
where $R\Gamma(\Gamma_p^{p^n,\sm}, -)$ refers to the smooth or locally compact group cohomology as in \cite[Definition 6.3.1]{RRLocAnII}.
Taking colimits as $h,n\to \infty$, we deduce that 
\[
\varinjlim_{n} \s{O}^{la}_{\Shan,L^{\cyc}}(\widetilde{U}_{p^n,L^{\cyc}})= \varinjlim_{h} R\Gamma(\Lie \Gamma_p,\s{F}_h^{\Gamma_p-la}). 
\]
Hence, by applying Corollary \ref{CoroExactDecompletion}, to show that the left term above admits a decompletion with respect to the action of $\Gamma^{\mathrm{arith}}$, it suffices to show that each $\s{F}_h^{\Gamma_p-la}$ admits a decompletion for the action of this group.

\textit{Step 4.} Let us now study more carefully the spaces  $\s{F}_h^{\Gamma_p-la}$. For $n\in \bb{N}$ let us denote $A_{n}=\s{O}_{\Shan}(U_{p^n})$, we have that $ A_n\widehat{\otimes}_{L} L^{\cyc} =  B_n=\s{O}_{\Shan}(U_{p^n,L^{\cyc}})$. Finally, write $B_{\infty}=\widehat{\s{O}}_{\Shan}(U_{\infty,L^{\cyc}})$.


  The $\Gamma^{\mathrm{arith}}\ltimes \Gamma_p$-representation $\s{F}_h$ over $B_{\infty}$ is relative locally analytic by the Step 3. Hence, there is a basis $\{e_{i}\}_{i\in I}$ of $\s{F}_{h}$ spanning a $\Gamma^{\mathrm{arith}}\ltimes \Gamma_p$-stable $B_{\infty}^{\circ}$-lattice $\s{F}^{\circ}_{h}\subset \s{F}_h$, such that there is an $\epsilon>0$ and a $B_{\infty}^{\circ}/p^{\epsilon}$-semilinear $\Gamma^{\mathrm{arith}}\ltimes \Gamma_p$-equivariant isomorphism 
  \[
  \s{F}^{\circ}_h/p^{\epsilon}\cong \bigoplus_{i\in I} B_{\infty}^{\circ}/p^{\epsilon} \cdot e_i. 
  \]
  where $\Gamma^{\mathrm{arith}}\ltimes \Gamma_p$ acts trivially on $e_i$ modulo $p^{\epsilon} \s{F}^{\circ}_h$. Let $\epsilon >\delta>0$ be fixed and closed enough to $0$, by \cite[Theorem 2.4.4 (1)]{RCGeoSenTheory}, there is $n\gg 0$ and elements $\{e_i'\}_{i\in I}$ in $\s{F}^{\circ}_{h}$ with $e_i'\equiv e_i \mod p^{\delta} \s{F}^{\circ}_{h}$, such that $\s{F}_{h,n}:=\widehat{\bigoplus}_{i\in I} B_{n} e_i'\subset \s{F}_{h}$ is a $\Gamma_p$-stable subrepresentation. In particular  $\s{F}_h=\s{F}_{h,n}\widehat{\otimes}_{B_n} B_{\infty}$. Furthermore, since $B_{\infty}^{\Gamma^{p^n}_p-an}=B_{\infty}^{\Gamma_p^{p^n}}=B_n$,    by \cite[Theorem 2.4.4 (2)]{RCGeoSenTheory} we have $\s{F}_{h,n}=\s{F}_{h}^{\Gamma_p^{p^n}-an}$.

%
     By Lemma \ref{LemmaAnRepNormalSubgroup}, $\s{F}_{h,n}=\s{F}_h^{\Gamma_p^{p^n}-an}$ is stable under the $\Gamma^{\mathrm{arith}}\ltimes \Gamma_{p}$-action in $\s{F}_h$.  In particular, $\s{F}_{h,n}$ is an $B_n$-semilinear locally analytic representation of $\Gamma^{\mathrm{arith}}$.   On the other hand, the normalized Tate traces $\widetilde{\Tr}_k\colon L^{\cyc}\to L_k$ give rise to normalized Tate traces $\widetilde{\Tr}_k\colon B_n=A_{n}\widehat{\otimes}_L L^{\cyc} \to A_{n}\widehat{\otimes}_L L_k $ that when endowed with the action of $\Gamma^{\mathrm{arith}}$ give rise to a Sen theory by Lemma \ref{LemmaBaseChangeSenTrances}. Hence, applying \cite[Theorem 2.4.4 (1)]{RCGeoSenTheory} with respect to the action of $\Gamma^{\mathrm{arith}}$, we deduce that the representations $\s{F}_{h,n}$ admit decompletions by $\Gamma^{\mathrm{arith}}$-locally analytic vectors. Taking colimits as $n\to \infty$, we get that $\s{F}_h^{\Gamma_p-la}=\varinjlim_{n} \s{F}_{h,n}$ admits a locally analytic decompletion for the action of $\Gamma^{\mathrm{arith}}$, proving what we wanted.  
\end{proof}

\begin{lem}\label{LemmaBaseChangeSenTrances}
Let $A$ be an \'etale algebra over $L\langle T_1^{\pm 1},\ldots, T_{e}^{\pm 1}, S_{e+1},\ldots, S_{d} \rangle$ that factors as a composite of rational localizations and finite \'etale maps. Consider the Sen theory on $L^{\cyc}$ given by Tate's normalized traces $\widetilde{\Tr}_k\colon \colon L^{\cyc}\to L_k$ for the action of $\Gamma^{\mathrm{arith}}$. Then the base change $\widetilde{\Tr}_k\colon A_{L^{\cyc}}\to A_{L}$ is a Sen theory as in \cite[Definition 2.1.1]{RCGeoSenTheory} for the action of $\Gamma^{\mathrm{arith}}$.
\end{lem}
\begin{proof}
Let $A_{\infty}=A\widehat{\otimes}_{L\langle T_1^{\pm 1},\ldots, T_{e}^{\pm 1}, S_{e+1},\ldots, S_{d} \rangle} L\langle T_1^{\pm 1/p^{\infty}},\ldots, T_{e}^{\pm 1/p^{\infty}}, S_{e+1}^{1/p^{\infty}},\ldots, S_{d}^{1/p^{\infty}} \rangle$.   Since $A_{L^{\cyc}}$ is a retract of $A_{\infty,L^{\cyc}}$ and the base change of the traces $\widetilde{\Tr}_{k}$ commute with the retract, by \cite[Lemma 2.2.8]{RCGeoSenTheory} it suffices to show that $(A_{\infty, L^{\cyc}}, \widetilde{\Tr}_k,\Gamma^{\mathrm{arith}})$ is a Sen theory as in \cite[Definition 2.1.1]{RCGeoSenTheory}. Note that $A_{\infty,L^{\cyc}}$ is perfectoid, then  by \cite[Lemma 2.2.12 (2)]{RCGeoSenTheory}   we are reduced to the case where  $A=L\langle T_1,\ldots, T_{e}, S_{e+1},\ldots, S_{d} \rangle$, in which case the triple $(A_{\infty, L^{\cyc}}, \widetilde{\Tr}_k,\Gamma^{\mathrm{arith}})$ satisfies the axioms of \cite[Definition 2.1.1]{RCGeoSenTheory} thanks to the fact that $A_{\infty,L^{\cyc}}^{\circ}=A_{\infty}^{\circ}\widehat{\otimes}_{\n{O}_L} \n{O}_{L^{\cyc}}$ and \cite[Proposition 4.1.1]{BC1}.
\end{proof}

\subsection{Computation of the arithmetic Sen operator}

\begin{proof}[Proof of Theorem \ref{TheoArithSenOla}: computation of the arithmetic Sen operator]
We keep the notation of Section \ref{SubSectionArithmeticDecompletion}.  The sheaf $\s{O}^{la}_{\Shan}$ on $|\Shan^{\tor}_{K^p,\infty,\bb{C}_p}|$ admits an arithmetic decompletion, in particular, it is endowed with an arithmetic Sen operator $\theta^{\mathrm{arith}}$. Note that the arithmetic  Sen operator $\theta^{\mathrm{arith}}$ acts as a derivation on  $\s{O}^{la}_{\Shan}$ being constructed as a base change of a derivation on the $\Gamma^{\mathrm{arith}}$-decompletion by locally analytic vectors.

 Our next task is to show that $\theta^{\mathrm{arith}}$ agrees with the operator $-\theta_{\mu}$. To prove that, let $U\subset \Shan_{K^pK_p,L}$ be an open affinoid as in  Theorem \ref{TheoArithSenOla}, by the equation \eqref{eqpampsq33d1d1} in the proof Corollary \ref{CoroDescriptionBaseChangeOla} the orbit map for the action of $\widetilde{K}_p$  induces an $\widehat{\s{O}}_{\Shan}(\widetilde{U}_{p^{\infty}, \bb{C}_p})$-linear  isomorphism 
\begin{equation}\label{qwijioenflaefwe}
\widehat{\s{O}}_{\Shan}(\widetilde{U}_{p^{\infty},\bb{C}_p})^{\widetilde{K}_p-la,\Gamma_p-{\sm}}\widehat{\otimes}_{\widehat{\s{O}}_{\Shan}(U_{p^{\infty},\bb{C}_p})^{\Gamma_p-sm}} \widehat{\s{O}}_{\Shan}(\widetilde{U}_{p^{\infty},\bb{C}_p}) \xrightarrow{\sim } C^{la}(\widetilde{K}_p, \widehat{\s{O}}_{\Shan}(\widetilde{U}_{p^{\infty},\bb{C}_p}))^{\f{n}^0_{\mu,\star_1}=0}.
\end{equation}
Indeed, the equation \eqref{eqpampsq33d1d1} is the isomorphism induced by the orbit map
\[
\widehat{\s{O}}_{\Shan}(\widetilde{U}_{p^{\infty},\bb{C}_p})^{\widetilde{K}_p-la, \f{n}^0_{\mu}=0}\otimes_{\widehat{\s{O}}_{\Shan}(U_{p^{\infty},\bb{C}_p})} \widehat{\s{O}}_{\Shan}(\widetilde{U}_{p^{\infty},\bb{C}_p})\xrightarrow{\sim} C^{la}(\widetilde{K}_{p}, \widehat{\s{O}}_{\Shan}(\widetilde{U}_{p^{\infty},\bb{C}_p}))^{\f{n}^0_{\mu,\star_1}=0}.
\]
Then, we can decomplete $\widehat{\s{O}}_{\Shan}(\widetilde{U}_{p^{\infty},\bb{C}_p})^{\widetilde{K}_p-la, \f{n}^0_{\mu}=0}$ by taking $\Gamma_{p}$-locally analytic vectors, and the equation \eqref{eqinfaaebkfabeulaef} and  its subsequent discussion yields \eqref{qwijioenflaefwe}.

Let us describe the equivariant actions on \eqref{qwijioenflaefwe}.  It is $\Gamma_p$-equivariant for the natural action on the left term given by acting on each factor, and the natural action on the right factor on the coefficients. The isomorphism is $\widetilde{K}_p$-equivariant for two different actions; in one hand one has the action of   $\widetilde{K}_p$ on $\widehat{\s{O}}_{\Shan}(\widetilde{U}_{p^{\infty},\bb{C}_p})^{\widetilde{K}_p-la,\Gamma_p-{\sm}}$ and the right regular action on the second term, on the other hand one has the action on $\widehat{\s{O}}_{\Shan}(\widetilde{U}_{p^{\infty},\bb{C}_p}) $ in the first term and the $\star_{1,3}$ action on the second term (i.e the left regular action and the action on the coefficients).


Since the arithmetic Sen operator $\theta^{\mathrm{arith}}$ is a derivation functorial for the  right regular action on the second term of \eqref{qwijioenflaefwe}, it must arise from a    $\widetilde{K}_p$-right invariant derivation of the space  $C^{la}(\widetilde{K}_{p}, \widehat{\s{O}}_{\Shan}(\widetilde{U}_{p^{\infty},\bb{C}_p}))^{\f{n}^0_{\mu,\star_1}=0}$, this corresponds to an object in 
\[
\theta^{\mathrm{arith}}\in (\widehat{\s{O}}_{\Shan}(\widetilde{U}_{p^{\infty},\bb{C}_p})\otimes \widetilde{\f{g}})/( \widehat{\s{O}}_{\Shan}(\widetilde{U}_{p^{\infty},\bb{C}_p})\otimes _{\s{O}_{\Fl}} \f{n}^0_{\mu}). 
\]
Furthermore, the arithmetic Sen-operator is also a  $\widetilde{K}_p\times \Gamma_p$ equivariant operator of $C^{la}(\widetilde{K}_{p}, \widehat{\s{O}}_{\Shan}(\widetilde{U}_{p^{\infty},\bb{C}_p}))^{\f{n}^0_{\mu,\star_1}=0}$ for the natural action of $\Gamma_p$ and the $\star_{1,3}$ action of $\widetilde{K}_p$, and  it is defined globally over the Shimura variety. The previous implies that the arithmetic Sen operator arises  from an element $\theta^{\mathrm{arith}}\in H^0_{\proket}(\Shan_{K^pK_p,\bb{C}_p}^{\tor}, \widehat{\s{O}}_{\Shan}\otimes_{\bb{Q}_p} \widetilde{\f{g}}_{\ket}/ \pi_{\HT}^{\tor,*} (\f{n}^0_{\mu}))$.

 Now,  write  $\widetilde{\f{g}}=\f{g}^{\mathrm{der}}\oplus \f{z}$ as the direct sum of its derived subalgebra and its center.  
It suffices to see that the projections of $\theta^{\mathrm{arith}}$ to $\widehat{\s{O}}_{\Shan}\otimes _{\bb{Q}_p}\f{z}$ and $\widehat{\s{O}}_{\Shan}\otimes_{\bb{Q}_p}\f{g}^{\mathrm{der}}/\pi^{\tor,*}_{\HT}(\f{n}^0_{\mu})$  are given by   $-\theta_{\mu}$.


\textbf{Case of $\f{z}$.}  Let us fix a connected component $\Shan^{\tor,0}_{K^p,\infty,\bb{C}_p}\subset \Shan^{\tor}_{K^p,\infty,\bb{C}_p}$ of the infinite level Shimura variety, and let $\Shan^{\tor,0}_{K^pK_p,\bb{C}_p}\subset \Shan^{\tor}_{K^pK_p,\bb{C}_p}$ be its image at level $K_p$. Then, by \cite[Corollary 5.2.4]{DiaoLogarithmicHilbert2018} the Galois group of the pro-Kummer-\'etale cover $\Shan^{\tor,0}_{K^p,\infty,\bb{C}_p}\to \Shan^{\tor,0}_{K^pK_p,\bb{C}_p}$ injects into $\bbf{G}^{\mathrm{der}}(\bb{Q}_p)$. Therefore, up to passing to an open subgroup to  guarantee that $\n{Z}(\widetilde{K}_p)\cap \bbf{G}^{\mathrm{der}}(\bb{Q}_p)=1$, the action of the center $\n{Z}(\widetilde{K}_p)$ on the connected components $H^0_{\proket}(\Shan^{\tor}_{K^p,\infty,\bb{C}_p}, \widehat{\s{O}}_{\Shan})$ is faithful.  Thus, for computing the image in $H^0_{\proket}(\Shan^{\tor}_{K^pK_p,\bb{C}_p},\widehat{\s{O}}_{\Shan}\otimes_{\bb{Q}_p}\f{z})$ of the arithmetic Sen operator, it suffices to compute the arithmetic Sen operator of  $H^0_{\proket}(\Shan^{\tor}_{K^p,\infty,\bb{C}_p}, \widehat{\s{O}}_{\Shan})$. 

By taking a special point of the Shimura datum $(\bbf{G}, X)$ as in \cite[Section 2.2.4]{DeligneShimura}, and by the naturality of the  arithmetic Sen operator, we  reduce the question when $\bbf{G}=\bbf{T}$ is a torus\footnote{Another way to reduce to the case of tori is by using \cite[Theorem 2.6.3]{DeligneShimura} that describes the group action on connected components of  Shimura varieties in terms of class field theory, we thank the referee for pointing this out.}. Let $E$ be the reflex field and let $\mathrm{res}_E:\bb{G}_m(\bb{A}^{\infty}_E)/E^{\times}_+ \to \Gal_{E}^{\mathrm{ab}}$ be the arithmetic reciprocity map (i.e. $ \mathrm{res}_E$ maps uniformizers in the unramified places dividing $\ell$ to the arithmetic Frobenius, and $E^{\times}_+\subset E^{\times}$ is the subgroup of positive units, that is, those units $a$ such that $\sigma(a)>0$ for all $\sigma\colon E\to \mathbb{R}$ a real embedding). In this case, the action of an element  $\sigma\in  \Gal_{E}^{\mathrm{ab}}$  in the image of $\mathrm{res}_E$ on the set 
\[
\Shan_{K^p,\bb{C}_p}=\bbf{T}(\bb{A}^{\infty}_{\bb{Q}})/K^p \overline{\bbf{T}(\bb{Q})}
\]
is by right multiplication of $N_{E/\bb{Q}}(\mu( \mathrm{res}_E^{-1}(\sigma)))$,  where $N_{E/\bb{Q}}:  \bbf{T}(\bb{A}_{E}^{\infty})\to \bbf{T}(\bb{A}_{\bb{Q}}^{\infty})$ is the norm map,  see Section 2.2.3 of \textit{loc. cit}.   Let $\f{p}$ be a place over $p$ and take  $L=E_{\f{p}}$,   the reciprocity map $\mathrm{res}_E$ is compatible with the local reciprocity map $\mathrm{res}_L:  L^{\times}\to \Gal_{L}^{\mathrm{ab}}$ mapping a uniformizer to Frobenius,  and the completed cohomology in this case  is nothing but the space continuous functions of $|\Shan_{K^p,\bb{C}_p}|$ to $\bb{Q}_p$.   Let $f:  |\Shan_{K^p,\bb{C}_p}|\to \bb{C}_p$ be a locally analytic function,  and let $\sigma\in W_L\subset \Gal_{L}$ be an element in the Weil group of $L$ (so that it admits a preimage along $\mathrm{res}_L$),  we have that 
\[
\sigma\cdot_{L} (f)(x)= f(\sigma^{-1}(x))= f( x N_{L/\bb{Q}_p}(\mu(\mathrm{res}^{-1}_L(\sigma))) ).
\]
Thus, if $\lambda: \bbf{T}(\bb{A}^{\infty}_{\bb{Q}})/K^p \overline{\bbf{T}(\bb{Q})} \to \bb{C}_p$  is a character and  $f\in C(\bbf{T}(\bb{A}^{\infty}_{\bb{Q}})/K^p \overline{\bbf{T}(\bb{Q})}, \bb{C}_p)$  satisfies $f(x t)= \lambda(t) f(x)$ for $t\in  K^pK_p/ (K^p,  K^pK_p\cap \overline{\bbf{T}(\bb{Q})})= \widetilde{K}_p$,  one has $\sigma\cdot f= \lambda(N_{L/ \bb{Q}_p}(\mu(\mathrm{res}^{-1}_{L}(\sigma)))) f$ for all $\sigma\in W_L$ close enough to $1$.  But the representation $N_{L/ \bb{Q}_p}\circ \mathrm{res}_L^{-1}:  W_L \to  \bb{Q}_p^{\times}$ has Hodge-Tate weight $1$,   this proves that $\theta^{\mathrm{arith}}= \theta_{\mu, R}=-\theta_{\mu,L}$   with $\theta_{\mu,R}$ and $\theta_{\mu,L}$ being the action of $\theta_{\mu}$ via the right and left regular action respectively. This proves the claim as we are considering the completed cohomology with respect to the left regular action.

\textbf{Case of $\f{g}^{\mathrm{der}}$.}  
  It suffices to compute the image of $\theta^{\mathrm{arith}}$ in $\widehat{\s{O}}_{\Shan}\otimes_{\bb{Q}_p} \f{g}^c/ \pi_{\HT}^{\tor,*}(\f{n}^0_{\mu})$. By Theorem \ref{TheoHodgeTatePeriod}  the $\widehat{\s{O}}_{\Shan}$-module $\widehat{\s{O}}_{\Shan}\otimes_{\bb{Q}_p} \f{g}^c$ has Hodge-Tate weights $-1,0,1$ and  its Hodge-Tate filtration is given by  the pullback along $\pi_{\HT}$ of 
  \[
\f{n}^0_{\mu}\subset \f{p}^{c,0}_{\mu} \subset \f{g}^{c,0}
  \]
  with graded pieces the pullback of $\f{n}^0_{\mu}$, $\f{m}^{c,0}_{\mu}$ and $\f{g}^{c,0}/\f{p}^{c,0}_{\mu}$ with Hodge-Tate weights $-1$, $0$ and $1$ respectively.  The arithmetic Sen operator is then depicted as  a Galois-equivariant  map 
  \[
  \theta^{\mathrm{arith}}\colon \widehat{\s{O}}_{\Shan}\to \widehat{\s{O}}_{\Shan}\otimes_{\bb{Q}_p} \f{g}^c/ \pi_{\HT}^{\tor,*}(\f{n}^0_{\mu}). 
  \]
  Since the source has Hodge-Tate weight zero,  it must land in the subspace 
  \[
   \theta^{\mathrm{arith}}\in \pi_{\HT}^{\tor,*}(\f{m}^{c,0}_{\mu}) \subset  \widehat{\s{O}}_{\Shan}\otimes_{\bb{Q}_p} \f{g}^c/ \pi_{\HT}^{\tor,*}(\f{n}^0_{\mu}). 
  \]
  Thus, to compute $\theta^{\mathrm{arith}}$ it suffices to understand the relation between the Galois action and the  horizontal action of $\mathfrak{m}^{c,0}_{\mu}$ on $\s{O}^{la}_{\Shan,\bb{C}_p}$. 
  
  Using the isomorphism \eqref{qwijioenflaefwe}, we can look at the $\widetilde{K}_{p}\times \Gamma_p$-equivariant subspace
  \begin{equation}\label{eihh9o3aoaefaef}
(\widehat{\s{O}}_{\Shan}(\widetilde{U}_{p^{\infty},\bb{C}_p}) \otimes_L  C^{\alg}(\bbf{G}^c,L))^{\f{n}^0_{\mu,\star_1}=0}  \subset C^{la}(\widetilde{K}_p, \widehat{\s{O}}_{\Shan}(\widetilde{U}_{p^{\infty},\bb{C}_p}))^{\f{n}^0_{\mu,\star_1}=0} 
  \end{equation}
  consisting of $\widehat{\s{O}}_{\Shan}(\widetilde{U}_{p^{\infty},\bb{C}_p})$-linear algebraic functions of $\bbf{G}^{c}$ which are killed by the left regular action of  $\f{n}^0_{\mu,\star_1}$. We let $C^{\mathrm{alg}}(\bbf{G}^c, \widehat{\s{O}}_{\Shan})^{\f{n}_{\mu,\star_1}^0=0}$ denote the descent of the left term of \eqref{eihh9o3aoaefaef} to a pro-Kummer-\'etale sheaf on $\Shan_{K^pK_p,\bb{C}_p}$. We claim that the action of $\pi_{\HT}^*(\f{m}^{c,0})$ on $\s{G}:=C^{\mathrm{alg}}(\bbf{G}^c, \widehat{\s{O}}_{\Shan})^{\f{n}_{\mu,\star_1}^0=0}$ is faithful. Indeed,   $\s{G}$ can be described as the pullback along $\pi_{\HT}$ of the $\bbf{G}^c$-equivariant sheaf on the flag variety $\Fl_{\mu}$ given by $C^{\mathrm{alg}}(\bbf{G}^c, \s{O}_{\Fl_{\mu}})^{\f{n}^0_{\mu,\star_1}=0}$.  By the dictionary between $\bbf{G}^c$-equivariant  quasi-coherent sheaves on $\Fl_{\mu}$ and  algebraic $\bbf{P}_{\mu}^c$-representations (see Proposition  \ref{PPropEqGSheavesFL}), the sheaf  $C^{\mathrm{alg}}(\bbf{G}^c, \s{O}_{\Fl_{\mu}})^{\f{n}^0_{\mu,\star_1}=0}$ corresponds to the $\bbf{P}_{\mu}^c$-representation $\Gamma(\f{n}_{\mu,\star_1}, \s{O}(\bbf{G}^{c}))=\s{O}(\bbf{N}_{\mu}\backslash \bbf{G}^c)$ consisting on left $\bbf{N}_{\mu}$-invariant algebraic regular functions on $\bbf{G}^c$. Writing $\s{O}(\bbf{G}^c)\cong \bigoplus_{\lambda} V_{\lambda}\otimes V_{\lambda}^{\vee}$ as a direct sum of its irreducible factors as $\bbf{G}^{c}\times \bbf{G}^{c}$-representation (with respect to the left and right regular action respectively, indexed by highest weight), we see that 
  \[
  \s{O}(\bbf{N}_{\mu}\backslash \bbf{G}^c)\cong \bigoplus_{\lambda} \Gamma(n_{\mu},V_{\lambda})\otimes V_{\lambda}^{\vee}\cong \bigoplus_{\lambda} W_{\lambda}\otimes V_{\lambda}^{\vee}
  \]
  where $W_{\lambda}\subset V_{\lambda}$ is the $\bbf{P}^c_{\mu}$-subrepresentation whose action factors through the highest weight representation of $\bbf{M}^c$ with weight $\lambda$. Therefore, we have an isomorphism of $\bbf{G}^c$-equivariant vector bundles 
  \begin{equation}\label{eq9ho3noanlefnir}
  \bigoplus_{\lambda} \mathcal{W}(W_\lambda)\otimes_{\bb{Q}_p} V_{\lambda}^{\vee} \cong  C^{\mathrm{alg}}(\bbf{G}^c, \s{O}_{\Fl_{\mu}})^{\f{n}^0_{\mu,\star_1}=0}
  \end{equation}
  where $\bbf{G}^c$ acts trivially on the $V_{\lambda}^{\vee} $-factor.  Is is then clear that the action of $\f{m}^{c,0}_{\mu}$ on the left term of \eqref{eq9ho3noanlefnir} is faithful (namely, this follows from the fact that the action of $\f{m}_{\mu}$ on the family of  representations $W_{\lambda}$ is faithful), and by taking pullbacks then so is the action on $\s{G}$ as wanted.

We are then reduced to describe the Galois action of $\s{G}$. For that, we can use the isomorphism \eqref{eq9ho3noanlefnir} to deduce that 
\[
\s{G}\cong \bigoplus_{\lambda} \pi_{\HT}^*(\n{W}(W_{\lambda}))\otimes_{\bb{Q}_p} V_{\lambda}^{\vee}.
\]  
Thus, to compute $\theta^{\mathrm{arith}}$ it suffices to identify the Hodge-Tate weights of $\pi_{\HT}^*(\n{W}(W_{\lambda}))$.   
But Corollary \ref{CoroHodgeComparison} says that 
\[
\pi_{\HT}^{\tor,*}(\n{W}(W_{\lambda})) = \pi_{K_p}^*(W_{\lambda,\mathrm{Hod}})\otimes_{\widehat{\s{O}}_{\Shan}} \widehat{\s{O}}_{\Shan}(-\mu(\lambda)),
\]
showing that the action of $\theta^{\mathrm{arith}}$ on $\s{G}$ is given by $-\theta_{\mu}$ as wanted. 
\end{proof}

\begin{cor}\label{coroArithSenCompleted}
The locally analytic completed cohomology \[(R\Gamma_{\proket}(\Shan^{\tor}_{K^p,\infty,\bb{C}_p}, \bb{Q}_p)\widehat{\otimes}_{\bb{Q}_p} \bb{C}_p )^{R\widetilde{K}_p-la}\]  admits a decompletion by locally analytic vectors as  in Definition \ref{DefinitionDecompletionByLocAn}.  More precisely, we have a $K_p\times \Gamma^{\mathrm{arith}}$-equivariant quasi-isomorphism 
\begin{equation}\label{eqpmaglr72}
RS_{L_{\infty}}((R\Gamma_{\proket}(\Shan^{\tor}_{K^p,\infty,\bb{C}_p}, \bb{Q}_p)\widehat{\otimes}_{\bb{Q}_p} \bb{C}_p )^{R\widetilde{K}_p-la}) \cong R\Gamma_{\an}(\Shan^{\tor}_{K^p,\infty,L_{\infty}}, \s{O}^{la, \Gamma^{\mathrm{arith}-la}}_{\Shan,L^{\cyc}})
\end{equation}
where $\s{O}^{la, \Gamma^{\mathrm{arith}-la}}_{\Shan}\subset \s{O}^{la}_{\Shan,L^{\cyc}}$ is the subsheaf of $K_p\times \Gamma^{\mathrm{arith}}$-locally analytic sections of Theorem \ref{TheoArithSenOla}.  The action of the arithmetic Sen operator is given by the global section of $-\theta_{\mu}$ acting on $\s{O}^{la}_{\Shan}$.

Moreover, for $i\in \bb{Z}$ the locally analytic completed cohomology groups  admit  decompletions,  and  we have $K_p\times \Gamma^{\mathrm{arith}}$-equivariant isomorphisms 
\begin{equation}\label{eqpamadjawefkf}
RS_{L_{\infty}}\bigg(\big( H^i_{\proket}(\Shan^{\tor}_{K^p,\infty,\bb{C}_p}, \bb{Q}_p)\widehat{\otimes}_{\bb{Q}_p}\bb{C}_p \big)^{\widetilde{K}_p-la}\bigg) = H^i_{\an}(\Shan^{\tor}_{K^p,\infty,L^{\cyc}}, \s{O}^{la,\Gamma^{\mathrm{arith}-la}}_{\Shan}).
\end{equation}
The analogue statement holds for the cohomology with compact support and the cohomologies of the boundary divisors of the Shimura variety. 
\end{cor}
\begin{proof}
We only prove the statement for the pro-Kummer-\'etale cohomology of $\Shan^{\tor}_{K^p,\infty,\bb{C}_p}$, the case of boundary divisors is proven in the exact same way (where the only input is Theorem \ref{TheoArithSenOla}), and in the case of the cohomology with compact support one uses the long exact sequence  \eqref{eqLongExactSequneceIla} which remains exact after applying the functor  $RS_{L_{\infty}}$ (cf. Corollary \ref{CoroExactDecompletion}).

   Let $\f{U}=\{U_i\}_{i\in I}$ be a finite cover by open affinoids of $\Shan^{\tor}_{K^pK_p,L}$  satisfying  the conditions of Proposition \ref{LemmaLongExactSequenceOlaSheaves}. Then, by the acyclicity on affinoids of $\s{O}^{la}_{\Shan}$ of  \textit{loc. cit.}  and  Theorem \ref{TheoMainVanishingOla}, the locally analytic completed cohomology 
   \[
   (R\Gamma_{\proket}(\Shan^{\tor}_{K^p,\infty,\bb{C}_p}, \bb{Q}_p)\widehat{\otimes}_{\bb{Q}_p} \bb{C}_p )^{R\widetilde{K}_p-la}\] is represented by a \v{C}ech complex 
\[
R \check{\Gamma}(\f{U}, \s{O}^{la}_{\Shan})
\]
whose terms are finite direct sums of locally analytic functions $\s{O}^{la}_{\Shan}(\pi_{K_p}^{-1}(U_{J,\bb{C}_p}))$, where $U_J=\cap_{j\in J} U_j$ and $J\subset I$. Since each term in the \v{C}ech complex admits a decompletion, the stability under cones and finite direct sums of Lemma \ref{LemmaStabilityDecompletions} imply that the locally analytic completed cohomology complex admits a decompletion.  By taking $H$-invariants and $\Gamma^{\mathrm{arith}}$-locally analytic vectors, the acyclicity of Theorem  \ref{TheoArithSenOla} yields the desired equivariant equivalence \eqref{eqpmaglr72}.

We now prove the decompletion for cohomology groups of \eqref{eqpamadjawefkf}. By  Lemma \ref{LemmaStabilityDecompletions},  it suffices to show that the natural map 
\begin{equation}\label{eqwsfpasas}
H^i_{\an}(\Shan^{\tor}_{K^p,\infty,L^{\cyc}}, \s{O}^{la,\Gamma^{\mathrm{arith}-la}}_{\Shan})\otimes_{L_{\infty},\sol} \bb{C}_p\to \big( H^i_{\proket}(\Shan^{\tor}_{K^p,\infty,\bb{C}_p}, \bb{Q}_p)\widehat{\otimes}_{\bb{Q}_p}\bb{C}_p \big)^{\widetilde{K}_p-la} 
\end{equation}
is an equivalence. But then, by \eqref{eqpmaglr72} and the fact  that completed cohomology admits a decompletion,  one has  the $\bb{C}_p$-semilinear $K_p\times \Gal_{\bb{Q}_p}$-equivariant quasi-isomorphism 
\begin{equation}\label{eqwpkapsfjapsd}
(R\Gamma_{\proket}(\Shan^{\tor}_{K^p,\infty,\bb{C}_p}, \bb{Q}_p)\widehat{\otimes}_{\bb{Q}_p} \bb{C}_p )^{R\widetilde{K}_p-la} \cong R\Gamma_{\an}(\Shan_{K^p,\infty, L^{\cyc}}^{\tor}, \s{O}^{la, \Gamma^{\mathrm{arith}}-la}_{\Shan}) \otimes^L_{L_{\infty},\sol} \bb{C}_p.
\end{equation}
Since $L_{\infty}\to \bb{C}_p$ is flat for the solid tensor product (cf.  \cite[Lemma 3.21]{RRLocallyAnalytic}), the presentation of $R\Gamma_{\an}(\Shan_{K^p,\infty, L^{\cyc}}^{\tor}, \s{O}^{la, \Gamma^{\mathrm{arith}}-la}_{\Shan})$ as \v{C}ech cohomology yields 
\[
H^i(R\Gamma(\Shan_{K^p,\infty, L^{\cyc}}^{\tor}, \s{O}^{la, \Gamma^{\mathrm{arith}}-la}_{\Shan}) \otimes^L_{L_{\infty},\sol} \bb{C}_p)= H^i_{\an}(\Shan_{K^p,\infty, L^{\cyc}}^{\tor}, \s{O}^{la, \Gamma^{\mathrm{arith}}-la}_{\Shan}) \otimes_{L_{\infty},\sol} \bb{C}_p. 
\]
Thus,  the isomorphism \eqref{eqwsfpasas} is obtained by taking cohomology groups of \eqref{eqwpkapsfjapsd}. 
\end{proof}

\begin{remark}
A priori it is unclear whether the  cohomology groups  $H^i_{\an}(\Shan^{\tor}_{K^p,\infty,L^{\cyc}}, \s{O}^{la,\Gamma^{\mathrm{arith}-la}}_{\Shan})$ have good topological properties. The isomorphism  \eqref{eqpamadjawefkf} shows that it is a colimit of Banach representations along injective transitions maps. With a little bit more effort one should be able to  prove that the transition maps are even compact, proving   that these cohomologies are LB spaces of compact type over $\bb{Q}_p$.  One could  also have proven the decompletion of the cohomology groups in a more direct way as in \cite[Remark 5.1.16]{LuePan}. Nevertheless, Lemma \ref{LemmaStabilityDecompletions} tells us that the decompletion is unique in the derived category, even for the complex computing locally analytic completed cohomology, and for the locally analytic sheaf $\s{O}^{la}_{\Shan}$. 
\end{remark}

\bibliographystyle{alpha}
\bibliography{LACC}

 \begin{center}
 $\bullet$
 \end{center}

\end{document}